\documentclass[11pt]{amsart}
\usepackage{amsmath,amssymb}
\usepackage{amssymb}
\usepackage{amsmath}
\usepackage{mathrsfs}
\usepackage{amsmath,amssymb}
\usepackage{color}

\newtheorem{theorem}{Theorem}[section]
\newtheorem{lemma}[theorem]{Lemma}
\newtheorem{proposition}[theorem]{Proposition}

\theoremstyle{definition}
\newtheorem{definition}[theorem]{Definition}
\newtheorem{example}[theorem]{Example}

\theoremstyle{remark}
\newtheorem{remark}[theorem]{Remark}

\def\rr{{\mathbb R}}

\def\cm{{\mathrm M}}

\def\az{\alpha}

\def\dz{\delta}

\def\hs{\hspace{0.3cm}}

\def\ls{\lesssim}

\def\r{\right}
\def\lf{\left}

\numberwithin{equation}{section}

\usepackage{color}

\begin{document}

\title[Fractional heat semigroups]
{Fractional heat semigroups on  metric measure  spaces with finite densities and applications to fractional dissipative equations}


\author[J. Huang]{Jizheng Huang}
\address{ School of Science, Beijing University of Posts and Telecommunications, Beijing 100876, P.R.China.}

\email{hjzheng@163.com}

\author[P. Li]{Pengtao Li}
\address{College of Mathematics, Qingdao University, Shandong 266071, China}

\email{ptli@qdu.edu.cn}

\author[Y. Liu]{Yu Liu}
\address{School of Mathematics and Physics, University of Science and Technology Beijing, Beijing 100083, China}
\email{liuyu75@pku.org.cn}

\author[S. Shi]{Shaoguang Shi}
\address{Department of Mathematics, Linyi University, Shandong 276005, China}
\email{shishaoguang@mail.bnu.edu.cn}



\date{}

\dedicatory{}


\begin{abstract}
Let $(\mathbb M, d,\mu)$ be a metric measure space with upper and lower
  densities:
$$
\begin{cases}
|||\mu|||_{\beta}:=\sup_{(x,r)\in \mathbb M\times(0,\infty)}
\mu(B(x,r))r^{-\beta}<\infty;\\
|||\mu|||_{\beta^{\star}}:=\inf_{(x,r)\in \mathbb M\times(0,\infty)}
\mu(B(x,r))r^{-\beta^{\star}}>0,
\end{cases}
$$
where $\beta, \beta^{\star}$ are two positive constants which are
less than or equal to the Hausdorff dimension of $\mathbb M$. Assume
that $p_t(\cdot,\cdot)$ is a heat kernel on $\mathbb M$ satisfying
Gaussian upper estimates and $\mathcal L$ is the generator of the
semigroup associated with $p_t(\cdot,\cdot)$. In this paper, via a
method independent of Fourier transform, we establish the decay
estimates for the kernels of the fractional heat
 semigroup $\{e^{-t \mathcal{L}^{\alpha}}\}_{t>0}$  and  the operators
$\{{\mathcal{L}}^{\theta/2} e^{-t \mathcal{L}^{\alpha}}\}_{t>0}$,
respectively. By these estimates, we obtain the regularity for the
Cauchy problem of the fractional dissipative
  equation associated with $\mathcal L$ on $(\mathbb M, d,\mu)$. Moreover, based on the geometric-measure-theoretic analysis of a new $L^p$-type capacity defined in $\mathbb{M}\times(0,\infty)$,
we also characterize  a nonnegative Randon measure $\nu$ on $\mathbb
M\times(0,\infty)$ such that $R_\alpha L^p(\mathbb M)\subseteq
L^q(\mathbb M\times(0,\infty),\nu)$ under $(\alpha,p,q)\in (0,1)\times(1,\infty)\times(1,\infty)$,  where $u=R_\alpha f$ is the weak solution of the fractional diffusion equation $(\partial_t+ \mathcal{L}^\alpha)u(t,x)=0$
 in $\mathbb M\times(0,\infty)$ subject to $u(0,x)=f(x)$ in $\mathbb M$.
\end{abstract}

\maketitle

\tableofcontents \pagenumbering{arabic}

 \vspace{0.1in}

\section{Introduction}
For $\alpha\in (0,1]$, the fractional Laplace operator with respect
to the spatial variable is defined by
$$\widehat{(-\triangle)^{\alpha}u}(t,\xi)=|\xi|^{2\alpha}\widehat{u}(t,\xi). $$
In mathematical physics,  the fractional Laplace operators are
widely applied to construct   partial differential equations in
order to study the physical phenomenon, e.g. the generalized
Naiver-Stokes equation, the quasi-geostrophic equation, the
Fokker-Planck equation, the anomalous diffusion equation and so on.
We refer the reader to \cite{Chen-w, C, JMF, Lio, Pod, SZ} and the
references therein.

In the study of the non-linear fractional power dissipative
equations, the space-time estimates  for the corresponding linear
equations play an important role. Let $\alpha\in (0, 1]$. In
\cite{miao}, Miao, Yuan and Zhang  established the space-time
estimates for the following fractional dissipative equations
\begin{equation}\label{eq-frac}
\begin{cases}
\partial_tu(x,t)+(-\Delta)^\alpha u(x,t)=F(t,x) &\ \forall\ (x,t)\in\mathbb{R}_+^{n+1};\\
u(x,0)=f(x)&\ \forall\ x\in\mathbb{R}^n,
\end{cases}
\end{equation}
thereby obtaining the well-posedness of a class of semi-linear fractional power dissipative equations.
The solution  of (\ref{eq-frac}) can be written as
$$
  u(x,t)=R_\alpha f(x,t)+S_\alpha F(x,t),
$$
where
\begin{equation*}
\begin{cases}
&R_\alpha f(x,t):=e^{-t(-\Delta)^\alpha} f(x);\\
&S_\alpha F(x,t):=\int_0^te^{-(t-s)(-\Delta)^\alpha} F(x,s)ds.
\end{cases}
\end{equation*}
In \cite{Zhai}, Zhai obtained a Strichartz type estimate for
$S_\alpha F$ and proved the global existence and uniqueness of
regular solutions for the generalized Naiver-Stokes equation. The
papers \cite{jiangrenjin2, chang, sx}  explored some
analytic-geometric properties of the regularity and the capacity
associated with $\partial_{t}+(-\Delta)^{\alpha}$.

Motivated by \cite{miao, Zhai,jiangrenjin2}, we consider the same
questions in the setting of the metric measure space  $(\mathbb M,
d,\mu)$   with:
\begin{equation}\label{eq1.2}
(\text{upper density})\ \ |||\mu|||_{\beta}:=\sup_{(x,r)\in
\mathbb M_+} \frac{\mu(B(x,r))}{r^{\beta}}<\infty
\end{equation}
 and
\begin{equation}\label{eq1.3}(\text{lower density})\ \
|||\mu|||_{\beta^{\star}}:=\inf_{(x,r)\in \mathbb M_+}
\frac{\mu(B(x,r))}{r^{\beta^{\star}}}>0,
\end{equation}
where $\mathbb M_+:=\mathbb M\times(0,\infty)$ and the positive
constants $\beta, \beta^{\star}$ are less than or equal to the Hausdorff
dimension of $\mathbb M$, see Definition \ref{defn1.2} below for
some details of metric measure spaces. It should be noted that if
$\beta=\beta^{\star}$, then (\ref{eq1.2})\ \&\ (\ref{eq1.3}) imply
\begin{equation}\label{eqa2.1}
    \mu(B(x,r))\simeq r^\beta
\end{equation}
for all $x\in\mathbb M$ and $r>0$. The metric measure space $(\mathbb{M}, d, \mu)$ satisfying
(\ref{eqa2.1}) is called an Ahlfors-David regular space - for example -
$$
\begin{cases}
\mathbb{M}=\mathbb R^n;\\
d(x,y)=|x-y|\ \ \forall\ \ x,y\in\mathbb R^n;\\
d\mu(x)=|x|^{-n<\gamma<n}\,dx\ \ \forall\ \ x\in\mathbb R^n,\\
\end{cases}
$$
which has
been investigated extensively; see e.g. \cite{Cheeger,
Heinonen, Capogna} and the references therein.

 The metric
spaces with finite densities cover many classical geometric models,
for example, Euclidean spaces,   hyperbolic spaces, nilpotent Lie
groups, connected Riemannian manifolds, etc. In recent years, the
problems related to analysis and partial differential equations on
metric spaces attract the attentions of many researchers. In
\cite{GHL}, under  some additional assumptions on $\mathbb{M}$,
Grigor'yan, Hu and Lau proved an embedding theorem and obtained the
existence results for weak solutions to semilinear elliptic
equations. Grigor'yan and Hu \cite {GH} established equivalent
characterizations for off-diagonal upper bounds of the heat kernel
of a regular Dirichlet form on the metric measure space in two
settings, see also  Grigor'yan, Hu and Hu \cite{GHH, GHH2}.
 In \cite{AGS}, Ambrosio, Gigli and Savar\'e studied the
heat flow and the calculus tools on metric measure spaces. Auscher
and Hytonen \cite{AH}  introduced the notion of spline function in
geometrically doubling quasi-metric spaces and obtained a universal
Calder\'on reproducing formula to study and develop function space
theories, singular integrals and $T(1)$ theorem. For further
information on function spaces on metric measure spaces, we refer
the reader to \cite{Lu-Yang-Yuan, Yan-Yang, Yang-1, yang}  and the
references therein.

Throughout this paper, let $\mathcal L$ be  the generator of the
semigroup $\{P_t\}_{t>0}$ on $\mathbb M$ introduced in Section
\ref{sec-2.1} and we assume the semigroup $\{P_t\}_{t>0}$ possesses
an integral kernel which is denoted by $p_{t}(\cdot,\cdot)$. For
$\alpha\in(0,1)$, the fractional power of $\mathcal{L}$ denoted by
$\mathcal{L}^{\alpha}$ is defined as (cf. \cite{yosida} or
\cite{Ma})
\begin{equation}\label{eq-fractional operator}
\mathcal{L}^{\alpha}=\frac{1}{\Gamma(-\alpha)}\int_0^\infty
\big[e^{-t\sqrt{\mathcal{L}}}f(x)-f(x)\big]\frac{dt}{t^{1+2\alpha}}\ \forall\
f\in L^2(\mathbb M)
\end{equation}
and  the fractional dissipative operator associated with
$\mathcal{L}$ is defined as
$$\mathcal{L}^{(\alpha)}=\partial_{t}+\mathcal{L}^{\alpha}.$$
We will investigate  the space-time estimates and regularity  for
the Cauchy problem of the fractional dissipative equation on the
metric measure space:
\begin{equation}\label{eq-1.7}
\begin{cases}
\mathcal{L}^{(\alpha)}u(x,t)=f(x,t) &\ \forall\ (x,t)\in\mathbb{M}\times (0, \infty);\\
u(x,0)=\varphi(x)&\ \forall\ x\in\mathbb{M}.
\end{cases}
\end{equation}
We call $u(\cdot,\cdot)$ a weak solution to equation (\ref{eq-1.7})
provided that for any function $v(\cdot, \cdot)\in
C^{\infty}_{0}(\mathbb{M}\times (0,\infty))$ one has
\begin{eqnarray*}
&&\int_{\mathbb{M}\times (0, \infty)}u(x,t)\mathcal{L}^{\alpha}v(x,t)d\mu(x)dt-\int_{\mathbb{M}\times
(0,
\infty)}u(x,t)\partial_tv(x,t)d\mu(x)dt\\
&&=-\int_{\mathbb{M}\times
(0, \infty)}f(x,t)v(x,t)d\mu(x)dt-\int_{\mathbb{M}}\varphi(x)v(x,0)d\mu(x).
\end{eqnarray*}
By the
Duhamel principle,  the weak solution $u(\cdot,\cdot)$ can be
written in the integral form as
\begin{eqnarray*}
u(x,t)&=&e^{-t{\mathcal{L}}^{\alpha}}(\varphi)(x)+\int^{t}_{0}e^{-(t-\tau){\mathcal{L}}^{\alpha}}(f)(x,\tau)d\tau\\
&:=&e^{-t{\mathcal{L}}^{\alpha}}(\varphi)(x)+G(f)(t,x).
\end{eqnarray*}
For the case of Laplace operator on $\mathbb{R}^{n}$,  the results of \cite{miao, jiangrenjin2} are based on a point-wise estimate of the integral kernel
of the fractional heat semigroups $e^{-t(-\Delta)^{\alpha}}$ in \cite{miao}. Denote by $K_{\alpha, t}$ the integral kernel of the operator $e^{-t(-\Delta)^\alpha}$, i.e.,
$$ K_{\alpha, t}(x)=(2\pi)^{-n/2}\int_{\mathbb R^n}e^{ix\cdot y-t|y|^{2\alpha}}dy $$
and denote by $K^{\theta}_{\alpha, t}$ the kernel
$(-\Delta)^{\theta/2}K_{\alpha, t}(x)$. In \cite{miao}, by an
invariant derivative technique and Fourier analysis  method, the
authors conclude that
 $K_{\alpha, t}$ and $K^{\theta}_{\alpha, t}$ satisfy  the following point-wise estimates, respectively (cf. \cite[Lemmas 2.1\ \&\
 2.2]{miao}),

\begin{equation}\label{eq-miao-result}
\left\{ \begin{aligned}
K_{\alpha, t}(x)&\lesssim \frac{t}{(t^{{1}/{2\alpha}}+|x|)^{n+2\alpha}} \quad \forall (x,t)\in\mathbb R_+^{n+1};\\
K_{\alpha, t}^{\theta}(x)&\lesssim
\frac{1}{(t^{1/2\alpha}+|x|)^{n+\theta}} \quad \forall
(x,t)\in\mathbb R_+^{n+1}.
\end{aligned} \right.
\end{equation}
Unfortunately, for the case of general operator $\mathcal{L}$ on
$\mathbb{M}$, the method of \cite{miao} is no longer applicable. To
overcome this difficulty, we use the subordinative formula to
represent the integral kernel $K^{\mathcal{L}}_{\alpha,t}$ of the
semigroup $e^{-t\mathcal{L}^\alpha}$:
\begin{equation}\label{eq-sub-for-1}
K^{\mathcal{L}}_{\alpha,t}(x,y)=\int^{\infty}_{0}\eta^{\alpha}_{t}(s)p_{s}(x,y)ds,
\end{equation}
where $p_{s}(\cdot,\cdot)$ is the integral kernel of the heat
semigroup $\{P_s\}_{s>0}$ and $\eta_{t}^{\alpha}(\cdot)$ satisfies the
conditions (\ref{eq-heat}) below. Without loss of generality,
 we introduce several assumptions of $p_{t}(\cdot,\cdot)$: {\bf(A1)}-{\bf(A4)}, see Sections \ref{sec-2.1}\ \&\ \ref{sec-2}. Under the assumptions {\bf(A1)}\ \&\ {\bf(A2)},
 the formula (\ref{eq-sub-for-1})
  enables us to obtain the point-wise
 estimates
 for $K^{\mathcal{L}}_{\alpha,t}(\cdot,\cdot)$ and ${\mathcal{L}}^{\theta/2}K^{\mathcal{L}}_{\alpha,t}(\cdot,\cdot)$ and see Propositions \ref{prop-1}\ \&\ \ref{prop-2}, respectively.
  In Section \ref{sec-3},
 we assume that the heat kernel $p_t(\cdot,\cdot)$ satisfies {\bf(A1)}. We apply the point-wise estimates of $K^{\mathcal{L}}_{\alpha,t}(\cdot,\cdot)$ to derive the space-time
 estimates for $e^{-t{\mathcal{L}}^{\alpha}}(\varphi)(x)$ and $G(f)(t,x)$, respectively,  see Theorems \ref{th-1.6}, \ref{th-1.7}\ \&\ \ref{th-4.3} for the details.
 In Section \ref{sec-4}, by the aid of the space-time
  estimates obtained in Section \ref{sec-31},   we prove some  regularity results of the Strichartz type for solutions to the problem  (\ref{eq-1.7}) if  $p_t(\cdot,\cdot)$
  satisfies {\bf (A1)}-{\bf (A3)} (see Theorem
  \ref{th-1.8}). Finally, if  $p_t(\cdot,\cdot)$ satisfies  {\bf (A4)}, we
  introduce the $L^p$ capacities in $\mathbb{M}_+$ and investigate $L^q(\mathbb M_+)$-extensions of $L^p(\mathbb
  M)$ in Sections \ref{sec-cap}\ \&\ \ref{sec-5}.

\begin{remark}
\item{(i)} Let $\mathbb{M}=\mathbb{R}^{n}$ and $\mathcal{L}=-\Delta$. For this case, the kernel of $e^{-t(-\Delta)^{\alpha}}$  obviously satisfies {\bf (A1)}-{\bf (A3)}. It is easy to see that
 Propositions \ref{prop-1}\ \&\ \ref{prop-2} go back to (\ref{eq-miao-result}). In other words, our method is also adequate for the classical fractional heat semigroup on the setting of $\mathbb{R}^{n}$.
\item{(ii)} We should point out that for the kernel $p_{t}(\cdot, \cdot)$ associated with $e^{-t\mathcal{L}}$,  the assumptions {\bf (A1)}-{\bf (A3)} are reasonable. In fact, there are many
operators $\mathcal{L}$ on $\mathbb{M}$ satisfying {\bf
(A1)}-{\bf(A3)}; see also Examples \ref{rem1.1}-\ref{exam-4}.
\end{remark}

Throughout this article, we will use  $c$ and  $C$ to denote the
positive constants, which are independent of main parameters and may
be different at each occurrence. In the above and below,
$\mathsf{X}\simeq\mathsf{Y}$ means
$\mathsf{Y}\lesssim\mathsf{X}\lesssim\mathsf{Y}$, where the second
estimate means that there is a positive constant $C$, independent of
main parameters, such that $\mathsf{X}\le C\mathsf{Y}$. For $\mathbb
X=\mathbb M$ or $\mathbb M_+$ the symbols $C_0(\mathbb X)$   stands
for all continuous functions with compact support in $\mathbb X$.

\section{Preliminaries}
\subsection{Heat kernels on metric measure spaces}\label{sec-2.1}

In what follows, we recall the definition of the metric measure space (cf. \cite{Gri}) and other related facts.
\begin{definition}\label{defn1.2}
We say that a triple $(\mathbb M, d,\mu)$ is a metric measure space
if $(\mathbb M,d)$ is a  non-empty metric space and $\mu$ is a Borel
measure on $\mathbb M$. Moreover, we always assume $\mathbb{M}$ is
locally compact and separable.
\end{definition}
Let $(\mathbb M,\mu)$ be a measure space. Denote by $L^q(\mathbb
M,\mu)$, $1\le q\le\infty$,  the Lebesgue spaces on $(\mathbb
M,\mu)$.
\begin{definition}\label{defn1.1}
A family $\{p_t\}_{t>0}$ of measurable functions on $\mathbb M\times\mathbb M$ is called a heat kernel if for almost all $x,y\in\mathbb M$ and $s,t>0$, it satisfies
\begin{eqnarray}
 &\rm(i)&\ \ p_t(x,y)\geq 0;\nonumber\\
 &\rm(ii)&\ \ \int_{\mathbb M}p_t(x,y)d\mu(y)\le 1;\label{a1.2}\\
 &\rm(iii)&\ \ p_t(x,y)=p_t(y,x);\nonumber\\
 &\rm(iv)&\ \ p_{s+t}(x,y)=\int_{\mathbb M}p_s(x,z)p_t(z,y)d\mu(z);\label{a1.3}\\
 &\rm(v)&\ \ \lim_{t\to 0^+}\int_{\mathbb M}p_t(x,y)f(y)d\mu(y)=f(x)\ \ \forall\ \ f\in L^2(\mathbb M,\mu).\label{a1.4}
\end{eqnarray}
\end{definition}
It is easy to see that the heat kernel and the Poisson kernel on
$\mathbb R^n$ satisfy Definition \ref{defn1.1}. For any
Riemannian manifold $\mathbb M$, the heat
 kernel associated with the Laplace-Beltrami operator satisfies Definition \ref{defn1.1} under certain mild hypotheses about $\mathbb M$ (cf. \cite{Gri1} and \cite{Gri2}).

Any heat kernel satisfying (i)-(v) above gives rise to the heat semigroup $\{P_t\}_{t>0}$, where $P_t$ is the operator defined on $L^2(\mathbb M,\mu)$ by
\begin{equation*}\label{a1.5}
  P_tf(x)=\int_{\mathbb M}p_t(x,y)f(y)d\mu(y).
\end{equation*}
In fact, by H\"older's inequality and $(\ref{a1.2})$, we can get $
\|P_tf\|_2\le \|f\|_2.$ This shows that $P_t$ is a bounded operator
on $L^2(\mathbb M)$ with $\|P_t\|_{op}\leq 1$. The symmetry of the
heat kernel implies that $P_t$ is a self-adjoint operator. Also,
$(\ref{a1.3})$ implies that $P_sP_t=P_{s+t}$, i.e., the family
$\{P_t\}_{t>0}$ is a semigroup. Furthermore, by $(\ref{a1.4})$, we
know that $\{P_t\}_{t>0}$ is a
 strongly continuous, self-adjoint, contraction semigroup on $L^2(\mathbb M,\mu)$.

The generator $\mathcal L$ of the semigroup $\{P_t\}_{t>0}$ is defined by
$$ \mathcal L f:=\lim_{t\rightarrow 0}\frac{f-P_tf}{t},$$
where the limit is in $L^2(\mathbb M,\mu)$. $\mathcal L$ is a self-adjoint, positive definite operator, we also have
\begin{equation*}\label{a1.6}
 P_t=e^{-t\mathcal L}.
\end{equation*}

Let $B(x,r)$ be the ball in $\mathbb M$ with radius $r$ centered at the point $x\in\mathbb M$, that is,
$$B(x,r):=\{y\in\mathbb M: d(x,y)<r\}.$$

Next, we give the following assumptions for the heat kernel $p_{t}(\cdot,\cdot)$ and a constant $C>0$.

{\bf Assumption (A1):} The heat kernel   satisfies the
  upper estimate $$ 0< p_s(x,y) \lesssim  \frac{1}{\mu(B(x,\sqrt{s}))}e^{-C{d(x,y)^2}/{s}} $$
for all $s > 0$  and   all $x, y\in \mathbb M$.

{\bf Assumption (A2):} The heat kernel satisfies the inequality
$$|\partial_s p_s(x,y)|\lesssim
\frac{1}{s\mu(B(x,\sqrt{s}))}e^{-C{d(x,y)^2}/{s}} $$ for all $s > 0$
and   all $x, y\in \mathbb M$.

{\bf Assumption (A3):} The heat kernel   satisfies the inequality
$$|p_s(x,y)-p_s(x_0,y)|\lesssim
\frac{1}{\mu(B(x,\sqrt{s}))}\big(\frac{d(x,x_0)}{\sqrt{s}}\big)^{\varepsilon}e^{-C{d(x,y)^2}/{s}}
$$ for all $s> 0$  and   all $x, y, x_0\in \mathbb M$ with some
$\varepsilon>0$.

There are many examples of heat kernels satisfying   {\bf(A1),
\bf(A2)} and {\bf(A3)}. The  heat kernel of the operator
$\mathcal{L}=-\Delta$ on $\mathbb{R}^n$, which is also called Gaussian
kernel,
 obviously satisfies {\bf(A1), \bf(A2)} and {\bf(A3)}. But beyond that,  we will give several typical examples on more general settings.

\begin{example}\label{rem1.1}  Following from \cite{Hebisch}, we know that the heat kernel of the operator
$$\mathcal{L}=-\frac{1}{\omega(x)}\sum_{i,j}^{}\partial_{i}(a_{ij}\partial_{j})$$ on $\mathbb{R}^n$ satisfies {\bf(A1), \bf(A2)} and {\bf(A3)},
where $(a_{ij}(\cdot))_{i,j}$ is a real symmetric matrix satisfying
$$\omega(x)|\xi|^{2}\simeq \sum_{i,j}^{}a_{ij}(x)\xi_{i}\bar{\xi_{j}}.$$
with $\omega$ being a nonnegative weight from the Muckenhoupt class
$A_{2}$.

\end{example}

\begin{example}\label{exam-1}
 Suppose $V$ is a nonnegative
potential that belongs to a certain reverse H\"{o}lder class (cf. \cite{shen}). \cite{dz1} implies that the heat kernel of the Schr\"{o}dinger operator $ \mathcal{L}=-\Delta+V$ on $\mathbb{R}^n$
satisfies {\bf(A1), \bf(A2)} and {\bf(A3)}.  Furthermore, the heat kernel of the  degenerate Schr\"{o}dinger operator  $\mathcal{L}=-\frac{1}{\omega(x)}\sum_{i,j}^{}\partial_{i}(a_{ij}\partial_{j})+V $
on $\mathbb{R}^n$ satisfies {\bf(A1), \bf(A2)} and {\bf(A3)} (cf.  \cite{huang} and \cite{Dz2}).
\end{example}

\begin{example}\label{exam-2}
  Let $\mathbb{G}$ be a stratified Lie group and $\Delta_\mathbb{G}$ be the sub-Laplacian on $\mathbb{G}$. Suppose $V$ is a nonnegative
potential that belongs to a certain reverse H\"{o}lder class (cf.
\cite{Li1}). Let   $ \mathcal{L}=-\Delta_\mathbb{G}+V$ be the
Schr\"{o}dinger operator on $\mathbb{G}$. It follows from
\cite{lin1} that the heat kernel of the operator $\mathcal{L}$
satisfies {\bf(A1), \bf(A2)} and {\bf(A3)}.
\end{example}

\begin{example}\label{exam-3}
  Let $\mathbb{G}$ be a connected and simply
connected nilpotent Lie group. Let $X\equiv\{X_1,\cdots,X_k\}$ be
left invariant vector fields on $\mathbb{G}$ satisfying the
H\"ormander condition that $X$ together with
their commutators of order $\le m$ generates the tangent space of
$\mathbb{G}$ at each point of $\mathbb{G}$. The sub-Laplacian is
given by $\Delta_\mathbb{G}\equiv\sum_{j=1}^kX_j^2.$ Suppose $V$ is
a nonnegative potential that belongs to a certain reverse H\"{o}lder
class (cf. \cite{Li1}). The sub-Laplace Schr\"odinger operator
$\mathcal{L}$ is defined by $\mathcal{L}=-\Delta_\mathbb{G}+V.$ It
follows from \cite{yang} that the heat kernel of the operator
$\mathcal{L}$ satisfies {\bf(A1), \bf(A2)} and {\bf(A3)}.
Especially, if $V=0$, then it is easy to check that the heat kernel
of the operator $\Delta_\mathbb{G}$   satisfies {\bf(A1), \bf(A2)}
and {\bf(A3)} (or see \cite{Varopoulos}).
\end{example}

\begin{example}\label{exam-4}
Let $\mathbb{M}$ be a complete Riemannian manifold satisfying the
doubling volume property, $d\mu$   be the Riemannian measure  and
$\Delta_{\mathbb{M}}$ be the Laplace-Beltrami operator on
$\mathbb{M}$. The main results in \cite{grigoryan} imply that the
heat kernel $p_t(\cdot,\cdot)$ satisfies {\bf(A1)}  and  {\bf(A2)}
if the heat kernel satisfies the additional estimate
$p_t(x,x)\lesssim {1}/{\mu(B(x,\sqrt{t}))}$  for any  geodesic ball
$ B(x,\sqrt{t})$ in $\mathbb{M}$. If $\mathbb{M}$ has nonnegative
Ricci density, it follows from \cite{yau}  that the   heat kernel
of $\Delta_{\mathbb{M}}$   satisfies {\bf(A1)}  and  {\bf(A3)}.

\end{example}

\subsection{Estimates for $e^{-t\mathcal{L}^\alpha}$}\label{sec-2}

In this section, we give the estimates for the heat kernel of the
semigroup $e^{-t\mathcal{L}^\alpha}$. For $\alpha>0$, let
$K^{\mathcal{L}}_{\alpha,t}(\cdot,\cdot)$ be the integral kernel of the semigroup
$e^{-t\mathcal{L}^\alpha}$. The subordinative formula (cf.
\cite{Gri}) indicates that $K^{\mathcal{L}}_{\alpha,t}(\cdot,\cdot)$ can be expressed as
\begin{equation}\label{eq-sub-for}
K^{\mathcal{L}}_{\alpha,t}(x,y)=\int^{\infty}_{0}\eta^{\alpha}_{t}(s)p_{s}(x,y)ds,
\end{equation}
where $p_{s}(\cdot,\cdot)$ be the integral kernel of the heat
semigroup $\{e^{-s\mathcal{L}}\}_{s>0}$. Here the non-negative
continuous  function $\eta_{t}^{\alpha}(\cdot)$ satisfies
\begin{equation}\label{eq-heat}
\begin{cases}
&\eta^{\alpha}_{t}(s)=\frac{1}{t^{1/\alpha}}\eta^{\alpha}_{1}(s/t^{1/\alpha});\\
&\eta^{\alpha}_{t}(s)\lesssim \frac{t}{s^{1+\alpha}}  \ \forall\ s,t>0;\\
&\int^{\infty}_{0}s^{-\gamma}\eta^{\alpha}_{1}(s)ds<\infty,\ \gamma>0;\\
&\eta^{\alpha}_{t}(s)\simeq\frac{t}{s^{1+\alpha}} \,\ \forall\,\
s\geq t^{1/\alpha}>0.
\end{cases}
\end{equation}
Please  see \cite{Gri} for some examples of the function
$\eta^{\alpha}_{t}$.

Now we give a point-wise estimate of the kernel $K^{\mathcal{L}}_{\alpha,t}(\cdot,\cdot)$.
\begin{proposition}\label{prop-1} Assume that the heat kernel $p_t(\cdot,\cdot)$ satisfies {\bf
(A1)}. Let $0<\alpha<1$. If the measure $\mu$ satisfies
(\ref{eq1.3}), then for all $(x,y,t)\in\mathbb M\times\mathbb M\times(0,\infty)$,
$$
K^{\mathcal{L}}_{\alpha, t}(x,y)\lesssim\min\Big\{t^{-\beta^{\star}/2\alpha},\ \frac{t}{d(x,y)^{\beta^{\star}+2\alpha}}\Big\}\lesssim\frac{t}{(t^{1/2\alpha}+d(x,y))^{\beta^{\star}+2\alpha}}.
$$
\end{proposition}
\begin{proof} Since  the measure $\mu$ satisfies
(\ref{eq1.3}),    $\mu(B(x,r))\gtrsim r^{\beta^{\star}}$ for any
$(x,r)\in \mathbb{M}_+$. So it can be deduced from {\bf (A1)} and
(\ref{eq-sub-for}) that
$$K^{\mathcal{L}}_{\alpha,t}(x,y)\lesssim\int^{\infty}_{0}\frac{t}{s^{1+\alpha}}s^{-\beta^{\star}/2}e^{-d(x,y)^{2}/s}ds.$$
Let $r=d(x,y)/\sqrt{s}$. Then
\begin{eqnarray*}
K^{\mathcal{L}}_{\alpha,t}(x,y)&\lesssim&\int^{\infty}_{0}\frac{t}{({d(x,y)}/{r})^{2+2\alpha}}\Big(\frac{d(x,y)}{r}\Big)^{-\beta^{\star}}e^{-r^{2}}\frac{d(x,y)^{2}}{r^{3}}dr\\
&\lesssim&\frac{t}{d(x,y)^{\beta^{\star}+2\alpha}}\int^{\infty}_{0}r^{2\alpha+\beta^{\star}-1}e^{-r^{2}}dr\\
&\lesssim&\frac{t}{d(x,y)^{\beta^{\star}+2\alpha}}.
\end{eqnarray*}
On the other hand, letting $\tau=s/t^{1/\alpha}$, we have
\begin{eqnarray*}
K^{\mathcal{L}}_{\alpha,t}(x,y)&\lesssim&\int^{\infty}_{0}s^{-\beta^{\star}/2}\frac{1}{t^{1/\alpha}}\eta_{1}^{\alpha}(s/t^{1/\alpha})ds\\
&\lesssim&t^{-\beta^{\star}/2\alpha}\int^{\infty}_{0}\frac{1}{\tau^{2\alpha}}\eta^{\alpha}_{1}(\tau)d\tau\\
&\lesssim&t^{-\beta^{\star}/2\alpha},
\end{eqnarray*}
which gives
$$K^{\mathcal{L}}_{\alpha, t}(x,y)\lesssim\min\Big\{t^{-\beta^{\star}/2\alpha},\ \frac{t}{d(x,y)^{2\alpha}}\Big\}.$$
Below we consider two cases.

{\it Case 1: $t^{1/2\alpha}>d(x,y)$}. This ensures
$$\frac{t}{(t^{1/2\alpha}+d(x,y))^{\beta^{\star}+2\alpha}}\gtrsim t^{-\beta^{\star}/2\alpha}.$$

{\it Case 2: $t^{1/2\alpha}\leq d(x,y)$}. This ensures
$$\frac{t}{(t^{1/2\alpha}+d(x,y))^{\beta^{\star}+2\alpha}}\gtrsim \frac{t}{(2d(x,y))^{\beta^{\star}+2\alpha}}.$$
This completes the proof of Proposition \ref{prop-1}.
\end{proof}
If we add one more condition {\bf (A4)}  to the above {\bf  (A1)-(A2)-(A3)}:\\
{\bf Assumption (A4):}  $$ p_s(x,y) \simeq
\frac{1}{\mu(B(x,\sqrt{s}))}e^{-C{d(x,y)^2}/{s}}\ \ \forall \ \
(x,y,s)\in\mathbb M\times\mathbb M\times(0,\infty),
$$
then we can also get a lower bound for $K^{\mathcal{L}}_{\alpha,t}(\cdot,\cdot)$.
\begin{proposition}\label{prop-4}
 Assume that the heat kernel $p_t(\cdot,\cdot)$ satisfies {\bf (A4)}.
Let $0<\alpha<1$. If the measure $\mu$ satisfies (\ref{eq1.2}), then
$$  K^{\mathcal{L}}_{\alpha, t}(x,y) \gtrsim \, \frac{t}{(t^{1/2\alpha}+d(x,y))^{\beta+2\alpha}}\ \ \forall\ \ (x,y,t)\in\mathbb M\times\mathbb M\times(0,\infty).$$
\end{proposition}
\begin{proof}
In fact, it can be deduced from {\bf (A4)}, (\ref{eq-sub-for}) and (\ref{eq-heat}) that
\begin{eqnarray*}
K^{\mathcal{L}}_{\alpha,t}(x,y)&\gtrsim&\int^{\infty}_{\max\{t^{1/\alpha},\, d(x,y)^2\}}\frac{t}{s^{1+\alpha}}s^{-\beta/2}e^{-d(x,y)^{2}/s}ds\\
&\gtrsim&\int^{\infty}_{\max\{t^{1/\alpha},\, d(x,y)^2\}}\frac{t}{s^{1+\alpha}}s^{-\beta/2}ds\\
&\gtrsim& t \left(\max\{t^{1/\alpha},\, d(x,y)^2\}\right)^{-\beta/2-\alpha}\\
&\simeq&\min\Big\{t^{-{\beta}/{2\alpha}},\, \frac{t}{d(x,y)^{\beta+2\alpha}}\Big\}\\
&\gtrsim&\frac{t}{(t^{1/2\alpha}+d(x,y))^{\beta+2\alpha}},
\end{eqnarray*}
which   is our desired result.
\end{proof}

Upon using the stronger condition on the heat kernel $p_t(\cdot,\cdot)$, we can
obtain the following estimate.
\begin{proposition}\label{prop-3}  Assume that the heat kernel $p_t(\cdot,\cdot)$ satisfies {\bf (A1)} and {\bf (A2)}.
Let $\alpha\in(0,1)$. If the measure $\mu$ satisfies (\ref{eq1.3}),
then
$$| \partial_t K^{\mathcal{L}}_{\alpha,t}(x,y)|\lesssim \frac{1}{(t^{1/2\alpha}+d(x,y))^{\beta^{\star}+2\alpha }}\ \ \forall\ \ (x,y,t)\in\mathbb M\times\mathbb M\times(0,\infty).$$
\end{proposition}
\begin{proof}
Since
$\eta^{\alpha}_{t}(s)=\frac{1}{t^{1/\alpha}}\eta^{\alpha}_{1}(s/t^{1/\alpha})$,
via $(\ref{eq-sub-for})$, we get
$$ K^{\mathcal{L}}_{\alpha,t}(x,y)=\int_0^\infty \frac{1}{t^{1/\alpha}}\eta^{\alpha}_{1}(s/t^{1/\alpha})p_s(x,y)ds.$$
Let $r=\frac{s}{t^{1/\alpha}}$. Then
$$ K^{\mathcal{L}}_{\alpha,t}(x,y)=\int_0^\infty \eta^{\alpha}_{1}(r)p_{rt^{1/\alpha}}(x,y)dr.$$
Therefore, by {\bf (A2)}, we can get
\begin{eqnarray*}
  \left|\partial_tK^{\mathcal{L}}_{\alpha,t}(x,y)\right|&\lesssim&\frac{1}{t}\, \int_0^\infty \eta^{\alpha}_{1}(r)(rt^{1/\alpha})^{-{\beta^{\star}}/{2}}e^{-\frac{d(x,y)^2}{rt^{1/\alpha}}}dr.
\end{eqnarray*}
In a way similar to verifying Proposition \ref{prop-1}, we can prove
Proposition \ref{prop-3}.
\end{proof}

Now we  are in a position to give an estimate of the fractional
power ${\mathcal{L}}^{\theta/2} $ acting on the kernel of the
fractional heat semigroup $\{e^{-t \mathcal{L}^{\alpha}}\}_{t>0}$.
\begin{proposition}\label{prop-2}
Assume that the heat kernel $p_t(\cdot,\cdot)$ satisfies {\bf (A1)} and {\bf (A2)}.
Let $\alpha\in(0,1)$ and $\theta>0$. If the measure $\mu$ satisfies
(\ref{eq1.3}), then for $x,y\in\mathbb M$ and $t>0$,
$$|{\mathcal{L}}^{\theta/2}K^{\mathcal{L}}_{\alpha,t}(x,y)|\lesssim \frac{1}{(t^{1/2\alpha}+d(x,y))^{\beta^{\star}+\theta}}.$$
\end{proposition}
\begin{proof} Let $\sigma$ be a positive number such that $\alpha \sigma \in
(0,1]$. By (\ref{eq-fractional operator}), we can see
\begin{equation}\label{eq-2.1}
 \mathcal{L}^{\alpha\sigma}=\frac{1}{\Gamma(-\alpha)}\int_0^\infty
 \big[e^{-t\mathcal{L}^\alpha}f(x)-f(x)\big]\frac{dt}{t^{1+\sigma}},
\end{equation}
whence reaching
\begin{eqnarray*}
\Big|{\mathcal{L}}^{\alpha\sigma}K^{\mathcal{L}}_{\alpha,t}(x,y)\Big|&=&\Big|\frac{1}{\Gamma(-\alpha)}
\int^{\infty}_{0}\Big[e^{-s\mathcal{L}^\alpha}(K^{\mathcal{L}}_{\alpha,t})(x,y)-K^{\mathcal{L}}_{\alpha,t}(x,y)\Big]\frac{ds}{s^{1+\sigma}}\Big|\\
&=&\Big|\frac{1}{\Gamma(-\alpha)}\int^{\infty}_{0}\int^{s}_{0}\partial_{r}e^{-r\mathcal{L}^\alpha}(K^{\mathcal{L}}_{\alpha,t})(x,y)\frac{dr ds}{s^{1+\sigma}}\Big|.
\end{eqnarray*}
We apply Proposition \ref{prop-3} to obtain
\begin{eqnarray*}
\Big|{\mathcal{L}}^{\alpha\sigma}K^{\mathcal{L}}_{\alpha,t}(x,y)\Big|&\lesssim&\frac{1}{|\Gamma(-\alpha)|}\int^{\infty}_0\int^{s}_{0}\Big|\partial_{r}
e^{-r\mathcal{L}^\alpha}(K^{\mathcal{L}}_{\alpha,t})(x,y)\Big|
\frac{dr ds}{s^{1+\sigma}}\\
&\lesssim&\int^{\infty}_{0}\int^{s}_{0}\frac{1}{[(r+t)^{1/2\alpha}+d(x,y)]^{\beta^{\star}+2\alpha}}\frac{drds}{s^{1+\sigma}}\\
&:=&I_{1}+I_{2},
\end{eqnarray*}
where
$$I_{1}:=\int^{t+d(x,y)^{2\alpha}}_{0}\int^{s}_{0}\frac{1}{[(r+t)^{1/2\alpha}+d(x,y)]^{\beta^{\star}+2\alpha}}\frac{drds}{s^{1+\sigma}}$$
and
$$I_{2}:=\int_{t+d(x,y)^{2\alpha}}^{\infty}\int^{s}_{0}\frac{1}{[(r+t)^{1/2\alpha}+d(x,y)]^{\beta^{\star}+2\alpha}}\frac{drds}{s^{1+\sigma}}.$$
For $I_{1}$, we have
\begin{eqnarray*}
I_{1}&\lesssim&\int^{t+d(x,y)^{2\alpha}}_{0}\int^{s}_{0}\frac{1}{[t+d(x,y)^{2\alpha}]^{\beta^{\star}/2\alpha+1}}\frac{ds}{s^{\sigma}}\\
&\lesssim&\frac{1}{(t+d(x,y)^{2\alpha})^{\beta^{\star}/2\alpha+1}}[t+d(x,y)^{2\alpha}]^{1-\sigma}\\
&\lesssim&\frac{1}{(t+d(x,y)^{2\alpha})^{\beta^{\star}/2\alpha+\sigma}}.
\end{eqnarray*}
For $I_{2}$, we obtain
\begin{align*}
I_{2}&\lesssim\int_{t+d(x,y)^{2\alpha}}^{\infty}\int^{s}_{0}\frac{1}{[(r+t)+d(x,y)^{2\alpha}]^{\beta^{\star}/2\alpha+1}}\frac{drds}{s^{1+\sigma}}\\
&\lesssim\int_{t+d(x,y)^{2\alpha}}^{\infty}\Big[-(s+t+d(x,y)^{2\alpha})^{-\beta^{\star}/2\alpha}+(t+d(x,y)^{2\alpha})^{-\beta^{\star}/2\alpha}\Big]\frac{ds}{s^{1+\sigma}}\\
&\lesssim\int_{t+d(x,y)^{2\alpha}}^{\infty}(t+d(x,y)^{2\alpha})^{-\beta^{\star}/2\alpha}\frac{ds}{s^{1+\sigma}}\\
&\lesssim \frac{1}{(t+d(x,y)^{2\alpha})^{\beta^{\star}/2\alpha+\sigma}}.
\end{align*}
It follows from the estimates of $I_{i}, i=1,2$, that
$$\Big|{\mathcal{L}}^{\alpha\sigma}K^{\mathcal{L}}_{\alpha,t}(x,y)\Big|\lesssim\frac{1}{(t+d(x,y)^{2\alpha})^{\beta^{\star}/2\alpha+\sigma}}\lesssim \frac{1}{(t^{1/2\alpha}+d(x,y))^{\beta^{\star}+2\alpha \sigma}}.$$
Then
$$|{\mathcal{L}}^{\theta/2}K^{\mathcal{L}}_{\alpha,t}(x,y)|\lesssim \frac{1}{(t^{1/2\alpha}+d(x,y))^{\beta^{\star}+\theta}}$$
is obtained via letting $\alpha\sigma={\theta}/{2}$.  This completes
the proof of Proposition \ref{prop-2}.
\end{proof}
\begin{remark}
Let $\mathbb{M}=\mathbb{R}^{n}$ and $\mathcal{L}=-\Delta$.
Proposition \ref{prop-2} agrees with \cite[Lemma 2.2]{miao}.
\end{remark}
\section{Estimating solutions of the fractional dissipative equation}\label{sec-3}
\subsection{Estimation - part A}\label{sec-31}
In this part, we give some basic space-time estimates for the solution to:
\begin{equation}\label{eq-4.1}
\begin{cases}
\partial_{t}u(x,t)+{\mathcal{L}}^{\alpha}u(x,t)=f(t,x) &\ \forall\ (x,t)\in\mathbb{M}\times (0, \infty);\\
u(x,0)=\varphi(x)&\ \forall\ x\in\mathbb{M}.
\end{cases}
\end{equation}
For the case $\mathbb{M}=\mathbb{R}^{n}$ and $\mathcal{L}=-\Delta$,
the space-time estimates for (\ref{eq-4.1}) have been investigated
by Miao-Yuan-Zhang \cite{miao}. By the Duhamel principle, the
solution to   (\ref{eq-4.1}) can be written in the integral form as
\begin{eqnarray*}
u(x,t)&=&e^{-t{\mathcal{L}}^{\alpha}}(\varphi)(x)+\int^{t}_{0}e^{-(t-\tau){\mathcal{L}}^{\alpha}}(f)(x,\tau)d\tau\\
&:=&e^{-t{\mathcal{L}}^{\alpha}}(\varphi)(x)+G(f)(t,x).
\end{eqnarray*}
Before we give the main results in this section, we state an
integral inequality which can be seen as a generalization of Young's
convolution inequality in Euclidean spaces (\cite[Theorem
0.3.1]{sogge} ).
\begin{lemma}\label{le-Young-inequality}
Assume that $\mathbb{X}$ and $\mathbb{Y}$ are measure spaces. Let
$K: \mathbb{X}\times \mathbb{Y}\rightarrow \mathbb{R}$ and
$1/q+1/r=1/p+1$. If
$$\sup_y\int_{\mathbb{X}}|K(x,y)|^{q}d\mu(x)\lesssim 1$$
and
$$\sup_x\int_{\mathbb{Y}}|K(x,y)|^{q}d\mu(y)\lesssim 1,$$
then
$$\int_{\mathbb{X}}\Big(\int_{\mathbb{Y}}K(x,y)f(y)d\mu(y)\Big)^{p}d\mu(x)\lesssim\Big(\int_{\mathbb{Y}}|f(y)|^{r}d\mu(y)\Big)^{p/r}.$$
\end{lemma}

The following lemma can be deduced from Propositions \ref{prop-1}\
\&\ \ref{prop-2} immediately.
\begin{lemma}\label{le-4.1}
Let $1\leq r\leq p\leq\infty$ and   $\varphi\in L^{r}(\mathbb{M})$.
\begin{itemize}
\item[\rm(i)] Assume that the heat kernel $p_t(\cdot,\cdot)$ satisfies {\bf (A1)} and  the measure $\mu$ satisfies
(\ref{eq1.3}). For $\alpha\in(0,1)$ and $t>0$, we have
$$\|e^{-t{\mathcal{L}}^{\alpha}}(\varphi)\|_{L^p(\mathbb{M})}\lesssim t^{-\beta^{\star}(1/r-1/p)/2\alpha}\|\varphi\|_{L^r(\mathbb{M})}.$$
\item[\rm(ii)] Assume that the heat kernel $p_t(\cdot,\cdot)$ satisfies {\bf (A1)}-{\bf (A2)} and the measure $\mu$ satisfies
(\ref{eq1.3}). For $\theta>0$, $\alpha\in(0,1)$ and $t>0$, we have
$$\|{\mathcal{L}}^{\theta/2}e^{-t{\mathcal{L}}^{\alpha}}(\varphi)\|_{L^p(\mathbb{M})}\lesssim t^{-\theta-\beta^{\star}(1/r-1/p)/2\alpha}\|\varphi\|_{L^r(\mathbb{M})}.$$
\end{itemize}
\end{lemma}

\begin{proof}
We begin with the proof of (i). Let $q$ obey $1/r+1/q=1/p+1$. By
Proposition \ref{prop-1}, we get
\begin{eqnarray*}
|e^{-t{\mathcal{L}}^{\alpha}}\varphi(x)|&=&\Big|\int_{\mathbb{M}}K^{\mathcal{L}}_{\alpha,t}(x,y)\varphi(y)d\mu(y)\Big|\\
&\leq&\int_{\mathbb{M}}\Big|K^{\mathcal{L}}_{\alpha,t}(x,y)\Big||\varphi(y)|d\mu(y)\\
&\lesssim& \int_{\mathbb{M}}\frac{
t}{(t^{1/2\alpha}+d(x,y))^{\beta^{\star}+2\alpha}}|\varphi(y)|d\mu(y).
\end{eqnarray*}
In what follows,  we estimate the integral
$$\int_{\mathbb{M}}\frac{t^{q}}{(t^{1/2\alpha}+d(x,y))^{q(\beta^{\star}+2\alpha)}}d\mu(y)$$
via writing
$$
\begin{cases}
\int_{\mathbb{M}}\frac{t^{q}}{(t^{1/2\alpha}+d(x,y))^{q(\beta^{\star}+2\alpha)}}d\mu(y):=
I_{0}+\displaystyle \sum^{\infty}_{k=1}I_{k};\\
I_{0}:=\int_{d(x,y)<2t^{1/2\alpha}}\frac{t^{q}}{(t^{1/2\alpha}+d(x,y))^{q(\beta^{\star}+2\alpha)}}d\mu(y);\\
I_{k}:=\int_{2^{k}t^{1/2\alpha}<d(x,y)<2^{k+1}t^{1/2\alpha}}\frac{t^{q}}{(t^{1/2\alpha}+d(x,y))^{q(\beta^{\star}+2\alpha)}}d\mu(y).
\end{cases}
$$

For $I_{0}$, a direct computation gives
$$I_{0}\leq t^{q}t^{-q(\beta^{\star}+2\alpha)/2\alpha}\mu(B(x, 2t^{1/2\alpha}))\lesssim t^{q}t^{-q(\beta^{\star}+2\alpha)/2\alpha}t^{\beta^{\star}/2\alpha}\lesssim t^{\beta^{\star}(1-q)/2\alpha}.$$
For $k\geq 1$, we can get
\begin{eqnarray*}
I_{k}&\leq&\int_{2^{k}t^{1/2\alpha}<d(x,y)<2^{k+1}t^{1/2\alpha}}\frac{t^{q}}{(t^{1/2\alpha}+2^{k}t^{1/2\alpha})^{q(\beta^{\star}+2\alpha)}}d\mu(y)\\
&\leq&\frac{t^{q}}{(t^{1/2\alpha}+2^{k}t^{1/2\alpha})^{q(\beta^{\star}+2\alpha)}}\mu(B(x, 2^{k+1}t^{1/2\alpha}))\\
&\leq&\frac{1}{2^{k(q\beta^{\star}+2q\alpha-\beta^{\star})}}t^{\beta^{\star}(1-q)/2\alpha},
\end{eqnarray*}
whence
\begin{align*}
&\int_{\mathbb{M}}\frac{t^{q}}{(t^{1/2\alpha}+d(x,y))^{q(\beta^{\star}+2\alpha)}}d\mu(y)\\
&\ \lesssim t^{\beta^{\star}(1-q)/2\alpha}\Big[1+\sum^{\infty}_{k=1}\frac{1}{2^{k(q\beta^{\star}+2q\alpha-\beta^{\star})}}\Big]\\
&\ \lesssim t^{\beta^{\star}(1-q)/2\alpha}.
\end{align*}
Similarly, we can show
$$\int_{\mathbb{M}}\frac{t^{q}}{(t^{1/2\alpha}+d(x,y))^{q(\beta^{\star}+2\alpha)}}d\mu(y) \lesssim t^{\beta^{\star}(1-q)/2\alpha}.$$
By Lemma \ref{le-Young-inequality}, we obtain
\begin{align*}
&\|e^{-t{\mathcal{L}}^{\alpha}}\varphi\|_{L^p(\mathbb{M})}\\
&\ \lesssim \Big[\int_{\mathbb{M}}\Big(\int_{\mathbb{M}}\frac{t}{(t^{1/2\alpha}+d(x,y))^{\beta^{\star}+2\alpha}}|\varphi(y)|d\mu(y)\Big)^{p}d\mu(x)\Big]^{1/p}\\
&\ \lesssim t^{\beta^{\star}(1/p-1/r)/2\alpha}\|\varphi\|_{L^r(\mathbb{M})}.
\end{align*}

Now we begin to  prove (ii). In a similar way to verify (i), we have
\begin{eqnarray*}
&&\int_{\mathbb{M}}\frac{d\mu(y)}{[t^{1/2\alpha}+d(x,y)]^{q(\beta^{\star}+\theta)}}\\
&&=\Big[\int_{d(x,y)<2t^{1/2\alpha}}+\sum^{\infty}_{k=1}\int_{2^{k}t^{1/2\alpha}<d(x,y)<2^{k+1}t^{1/2\alpha}}\Big]
\frac{d\mu(y)}{[t^{1/2\alpha}+d(x,y)]^{q(\beta^{\star}+\theta)}}\\
&&\lesssim  t^{-q\theta/2\alpha+\beta^{\star}(1-q)/2\alpha}\Big[1+\sum^{\infty}_{k=1}\frac{1}{2^{k[q(\beta^{\star}+\theta)-\beta^{\star}]}}\Big]\\
&&\lesssim  t^{-q\theta/2\alpha+\beta^{\star}(1-q)/2\alpha},
\end{eqnarray*}
thereby getting
$$\int_{\mathbb{M}}\frac{d\mu(x)}{[t^{1/2\alpha}+d(x,y)]^{q(\beta^{\star}+2\alpha v)}}\lesssim t^{-q\theta/2\alpha+\beta^{\star}(1-q)/2\alpha}.$$
By Proposition \ref{prop-2}, we use Lemma \ref{le-Young-inequality} again to get
\begin{eqnarray*}
\|{\mathcal{L}}^{\theta/2}K^{\mathcal{L}}_{\alpha,t}\varphi\|_{L^p(\mathbb{M})}&\leq&\Big[\int_{\mathbb{M}}\Big(\int_{\mathbb{M}}
\frac{|\varphi(r)|}{(t^{1/2\alpha}+|x|)^{\beta^{\star}+\theta}}d\mu(y)\Big)^{p}d\mu(x)\Big]^{1/p}\\
&\lesssim &t^{-\theta/2\alpha-\beta^{\star}(1/r-1/p)/2\alpha}\|\varphi\|_{L^r(\mathbb{M})}.
\end{eqnarray*}
Hence,  Lemma \ref{le-4.1} is proved.
\end{proof}

\begin{definition}
The triplet $(q,p,r)$ is called an admissible triplet or a
generalized admissible triplet provided that
$1/q=\beta^{\star}(1/r-1/p)/2\alpha$,  where
$$1<r\leq p<
\begin{cases}
{\beta^{\star} r}/{(\beta^{\star}-2\alpha)}, & \beta^{\star}>2r\alpha,\\
\infty,& \beta^{\star}\leq 2r\alpha
\end{cases}$$
or $1/q=\beta^{\star}(1/r-1/p)/2\alpha$,  where
$$1<r\leq p<
\begin{cases}
{\beta^{\star} r}/{(\beta^{\star}-2\alpha r)}, &  \beta^{\star}>2r\alpha,\\
\infty,&  \beta^{\star}\leq 2r\alpha.
\end{cases}$$
\end{definition}

Let $\mathbb{X}$ be a Banach space and let $I=[0,T)$. We define the
time-weighted space-time Banach space $C_{\sigma}(I; \mathbb{X})$
and the corresponding homogeneous space $\dot{C}_{\sigma}(I;
\mathbb{X})$ as follows:
$$C_{\sigma}(I; \mathbb{X}):=\Big\{f\in C(I; \mathbb{X}):\ \|f\|_{C_{\sigma}(I; \mathbb{X})}=\sup_{t\in I}t^{1/\sigma}\|f\|_{\mathbb{X}}<\infty\Big\}$$
and
$$\dot{C}_{\sigma}(I; \mathbb{X}):=\Big\{f\in C_{\sigma}(I; \mathbb{X}):\ \lim_{t\rightarrow 0+}t^{1/\sigma}\|f\|_{\mathbb{X}}=0\Big\}.$$

Now we give a space-time estimate for the term $e^{-t{\mathcal{L}}^{\alpha}}(\varphi)$. In the sequel, for a Banach space $\mathbb{X}$, we denote by $C_{b}(I; \mathbb{X})$ the
 space of all bounded continuous functions from $I$ to $\mathbb{X}$.

\begin{theorem}\label{th-1.6} Assume that the heat kernel $p_t(\cdot,\cdot)$ satisfies {\bf (A1)} and  the measure $\mu$ satisfies
(\ref{eq1.3}).
\begin{itemize}
\item[\rm(i)]  Let $(q,p,r)$ be any admissible triplet and let
$\varphi\in L^{r}(\mathbb{M})$. Then for $0<T\leq \infty$,
$e^{-t{\mathcal{L}}^{\alpha}}(\varphi)\in L^{q}(I; L^{p})\cap
C_{b}(I; L^{r})$ with the estimate
$$\|e^{-t{\mathcal{L}}^{\alpha}}(\varphi)(x)\|_{L^{q}(I; L^{p})}\lesssim \|\varphi\|_{L^{r}(\mathbb{M})}.$$

\item[\rm(ii)]  Let $(q,p,r)$ be any generalized admissible triplet
and let $\varphi\in L^{r}(\mathbb{M})$. Then for $0<T\leq \infty$,
$e^{-t{\mathcal{L}}^{\alpha}}(\varphi)\in C_{q}(I; L^{p})\cap
C_{b}(I; L^{r})$ with the estimate
$$\|e^{-t{\mathcal{L}}^{\alpha}}(\varphi)(x)\|_{C_{q}(I; L^{p})}\lesssim\|\varphi\|_{L^{r}(\mathbb{M})}.$$
\end{itemize}
\end{theorem}

\begin{proof}

(i) We divide the argument into two cases.

{\it Case 1: $p=r$ and $q=\infty$.} By Lemma \ref{le-4.1}, for any $t>0$, we have
\begin{eqnarray*}
\|e^{-t{\mathcal{L}}^{\alpha}}(\varphi)(x)\|_{L^{\infty}(I; L^{r})}&=&\sup_{t>0}\|e^{-t{\mathcal{L}}^{\alpha}}(\varphi)(x)\|_{L^{r}(\mathbb{M})}\\
&\lesssim &\sup_{t>0}t^{-\beta^{\star}(1/r-1/r)/2\alpha}\|\varphi\|_{L^{r}(\mathbb{M})}\\
&\lesssim&\|\varphi\|_{L^{r}(\mathbb{M})}.
\end{eqnarray*}

{\it Case 2: $p\neq r$}. Denote by $F(t)(\varphi):=
\|e^{-t{\mathcal{L}}^{\alpha}}(\varphi)\|_{L^p(\mathbb{M})}$. Since
$(q,p,r)$ is an admissible triplet, a further use of Lemma
\ref{le-4.1} can be deduced that for a positive constant $C$,
$$F(t)(\varphi)=\|e^{-t{\mathcal{L}}^{\alpha}}(\varphi)\|_{L^{p}(\mathbb{M})}\le Ct^{-\beta^{\star}(1/r-1/p)/2\alpha}\|\varphi\|_{L^{r}(\mathbb{M})}\le C t^{-1/q}\|\varphi\|_{L^{r}(\mathbb{M})}$$
and
\begin{eqnarray*}
\Big|\Big\{t:\ |F(t)(\varphi)|>\tau\Big\}\Big|&\lesssim&\Big|\Big\{t:\ Ct^{-1/q}\|\varphi\|_{L^{r}}>\tau\Big\}\Big|\\
&\leq&\Big|\Big\{t:\ t<\Big(\frac{C\|\varphi\|_{L^{r}}}{\tau}\Big)^{q}\Big\}\Big|\\
&\lesssim&\Big(\frac{\|\varphi\|_{L^{r}}}{\tau}\Big)^{q},
\end{eqnarray*}
which means that $F(t)$ is a weak $(r,q)$ type operator. On the other hand, notice that
\begin{eqnarray*}
|e^{-t{\mathcal{L}}^{\alpha}}\varphi(x)|&\leq&\int_{\mathbb{M}}\left|K^{\mathcal{L}}_{\alpha,t}(x,y)\right||\varphi(y)|d\mu(y)\\
&\lesssim& \int_{\mathbb{M}}\frac{
t}{(t^{1/2\alpha}+|x-y|)^{\beta^{\star}+2\alpha}}|\varphi(y)|d\mu(y).
\end{eqnarray*}
We can use Lemma \ref{le-Young-inequality} again to deduce that for $\varphi\in L^{p}(\mathbb{M})$,
$$F(t)(\varphi)=\|e^{-t{\mathcal{L}}^{\alpha}}(\varphi)\|_{L^p(\mathbb{M})}\lesssim t^{\beta^{\star}(1/p-1/p)/2\alpha}\|\varphi\|_{L^{p}(\mathbb{M})}.$$
This implies that $F(t)$ is a $(p,\infty)$ type operator.

For any admissible triplet $(q,p,r)$, we can find another admissible triplet $(q_{1},p_{1},r_{1})$ such that $q_{1}<q<\infty$, $r_{1}<r<p$ and
$$\begin{cases}
&1/q=\theta/q_{1}+(1-\theta)/\infty,\\
&1/r=\theta/r_{1}+(1-\theta)/p.
\end{cases}$$
Marcinkiewicz interpolation theorem implies that $F(t)$ is a strong $(r,q)$ type operator and
$$\|e^{-t{\mathcal{L}}^{\alpha}}(\varphi)\|_{L^{q}(I; L^{p})}\lesssim \|\varphi\|_{L^{r}(\mathbb{M})}.$$

(ii) The argument can be also divided into two
cases.

{\it Case 3: $p=r$ and $q=\infty$}. We have
$$\|e^{-t{\mathcal{L}}^{\alpha}}(\varphi)\|_{L^{\infty}(I; L^{p})}=\sup_{t>0}\|e^{-t{\mathcal{L}}^{\alpha}}(\varphi)\|_{{L^p(\mathbb{M})}}\lesssim \|\varphi\|_{L^{r}(\mathbb{M})}.$$

{\it Case 4: $p\neq r$}. Because $(q,p,r)$ is an admissible triplet, that is,
$1/q=\beta^{\star}(1/r-1/p)/2\alpha$, then taking $q'$ such that
$1/p+1=1/r+1/q'$, we can apply Lemma \ref{le-4.1} to get
\begin{eqnarray*}
\|e^{-t{\mathcal{L}}^{\alpha}}(\varphi)\|_{C_{q}(I; L^{p})}&=&\sup_{t>0}t^{1/q}\|e^{-t{\mathcal{L}}^{\alpha}}(\varphi)\|_{L^{p}(\mathbb{M})}\\
&\lesssim&t^{1/q}t^{-\beta^{\star}(1/r-1/p)/2\alpha}\|\varphi\|_{L^{r}(\mathbb{M})}\\
&\lesssim&\|\varphi\|_{L^{r}(\mathbb{M})}.
\end{eqnarray*}
On the other hand, for $t\in I$, Lemma \ref{le-4.1} implies
$$\|e^{-t{\mathcal{L}}^{\alpha}}(\varphi)\|_{L^{r}(\mathbb{M})}\lesssim \|\varphi\|_{L^{r}(\mathbb{M})},
$$
consequently,
$$e^{-t{\mathcal{L}}^{\alpha}}(\varphi)\in C_{b}(I; L^{r}).
$$
This completes the
proof of Theorem \ref{th-1.6}.
\end{proof}

Now we derive the space-time estimate  of the non-homogeneous part
$G(f)$ of the solution $u$ to the  equation (\ref{eq-4.1}).
\begin{theorem}\label{th-1.7}
Assume that the heat kernel $p_t(\cdot,\cdot)$ satisfies {\bf (A1)} and  the
measure $\mu$ satisfies (\ref{eq1.3}). For $b>0$ and $T>0$, let
$r_{0}=\beta^{\star} b/2\alpha$, $I=[0,T)$. Assume that $r\geq r_{0}>1$ and
that $(q,p,r)$ is an admissible triplet satisfying that $p>b+1$.
\begin{itemize}
\item[\rm(i)] If $f\in L^{q/(b+1)}(I; L^{p/(b+1)})$, then for $p<r(b+1)$,
$$\|G(f)\|_{L^{\infty}(I; L^{r})}\lesssim T^{1-\beta^{\star} b/2r\alpha}\|f\|_{L^{q/(b+1)}(I; L^{p/(b+1)})}.$$
For $p\geq r(b+1)$ and
$\theta=[p-r(b+1)]/(b+1)(p-r)$,
$$\|G(f)\|_{L^{\infty}(I; L^{r})}\lesssim T^{1-\beta^{\star} b/2r\alpha}\||f|^{1/(b+1)}\|^{\theta(b+1)}_{L^{\infty}(I; L^{r})}
\||f|^{1/(b+1)}\|^{(1-\theta)(b+1)}_{L^{q}(I; L^{p})}.$$
\item[\rm(ii)] If $f\in L^{q/(b+1)}(I; L^{p/(b+1)})$, then for $p<r(b+1)$,
$$\|G(f)\|_{L^{q}(I; L^{p})}\lesssim T^{1-\beta^{\star} b/2r\alpha}\|f\|_{L^{q/(b+1)}(I; L^{p/(b+1)})}.$$
For $p\geq r(b+1)$ and $\theta=[p-r(b+1)]/(b+1)(p-r)$,
$$\|G(f)\|_{L^{q}(I; L^{p})}\lesssim T^{1-\beta^{\star} b/2r\alpha}\||f|^{1/(b+1)}\|^{\theta(b+1)}_{L^{\infty}(I; L^{r})}
\||f|^{1/(b+1)}\|^{(1-\theta)(b+1)}_{L^{q}(I; L^{p})}.$$
\end{itemize}
\end{theorem}
\begin{proof}
We first prove (i).  For the case $p<r(b+1)$, we have
\begin{align*}
&\|G(f)\|_{L^{\infty}(I; L^{r})}\\
&\ \ \leq\sup_{t\in I}\|G(f)(t,\cdot)\|_{L^{r}(\mathbb{M})}\\
&\ \ \leq\sup_{t\in I}\Big\|\int^{t}_{0}e^{-(t-s){\mathcal{L}}^{\alpha}}f(s,x)ds\Big\|_{L^{r}(\mathbb{M})}\\
&\ \ \leq\sup_{t\in I}\int^{t}_{0}\Big\|e^{-(t-s){\mathcal{L}}^{\alpha}}f(s,x)\Big\|_{L^{r}(\mathbb{M})}ds.
\end{align*}
Take $q'$ such that $(b+1)/p+1/q'=1/r+1$. By Lemma \ref{le-4.1}, we get
\begin{align*}
&\|G(f)\|_{L^{\infty}(I; L^{r})}\\
&\ \ \leq\sup_{t\in I}\int^{t}_{0}(t-s)^{-\beta^{\star}(1-1/q')/2\alpha}\|f(s,\cdot)\|_{L^{p/(b+1)}(\mathbb{M})}ds\\
&\ \ \leq\sup_{t\in
I}\int^{t}_{0}(t-s)^{-\beta^{\star}[(b+1)/p-1/r]/2\alpha}\|f(s,\cdot)\|_{L^{p/(b+1)}(\mathbb{M})}ds.
\end{align*}
Let $\chi$ be an index such that $(b+1)/q+1/\chi=1$. We can use H\"older's inequality to deduce that
\begin{align*}
&\|G(f)\|_{L^{\infty}(I; L^{r})}\\
&\ \leq\sup_{t\in
I}\Big(\int^{t}_{0}(t-s)^{\beta^{\star}[(b+1)/p-1/r]\chi/2\alpha}ds\Big)^{1/\chi}
\Big(\int^{t}_{0}\|f(s,\cdot)\|^{q/(b+1)}_{L^{p/(b+1)}(\mathbb{M})}ds\Big)^{(b+1)/q}\\
&\ \lesssim T^{1/\chi-\beta^{\star}[(b+1)/p-1/r]/2\alpha}\Big(\int^{T}_{0}\|f(s,\cdot)\|^{q/(b+1)}_{L^{p/(b+1)}(\mathbb{M})}ds\Big)^{(b+1)/q}\\
&\ \lesssim T^{1-\beta^{\star} b/2\alpha r}\|f\|_{L^{q/(b+1)}(I; L^{p/(b+1)})}.
\end{align*}

For the case of $p\geq r(b+1)$, applying Lemma \ref{le-4.1} again,
we can get
\begin{align*}
&\|G(f)\|_{L^{\infty}(I; L^{r})}\\
&\ \leq\sup_{t\in I}\int^{t}_{0}\|e^{-(t-\tau){\mathcal{L}}^{\alpha}}f(s,\cdot)\|_{L^{r}(\mathbb{M})}ds\\
&\ \leq\sup_{t\in I}\int^{t}_{0}\|f(s,\cdot)\|_{L^{r}(\mathbb{M})}ds\\
&\ \leq\sup_{t\in I}\int^{t}_{0}\Big[\int_{M}|f(s,x)|^{r(b+1)/(b+1)}d\mu(x)\Big]^{(b+1)/r(b+1)}ds\\
&\ =\sup_{t\in I}\int^{t}_{0}\||f(s,\cdot)|^{1/(b+1)}\|^{b+1}_{L^{r(b+1)}(\mathbb{M})}ds.
\end{align*}
Let
$$\begin{cases}
&\theta\in (0,1);\\
&r(b+1)\theta p'=r;\\
&r(b+1)(1-\theta)q'=p;\\
&1/p'+1/q'=1.
\end{cases}$$
Then
$$\theta=[p-r(b+1)]/(p-r).
$$
The H\"older
inequality can be used again to obtain
\begin{align*}
&\|G(f)\|_{L^{\infty}(I; L^{r})}\\
&\ \leq\sup_{t\in I}\int^{t}_{0}\||f(s,\cdot)|^{1/(b+1)}\|^{b+1}_{L^{r(b+1)}(\mathbb{M})}ds\\
&\ \leq\sup_{t\in I}\int^{t}_{0}\Big(\int_{\mathbb{M}}|f(s,x)|^{r\theta(b+1) p'/(b+1)}d\mu(x)\Big)^{1/rp'}\\
&\ \ \ \times\Big(\int_{\mathbb{M}}|f(s,x)|^{r(1-\theta)(b+1)q'/(b+1)}d\mu(x)\Big)^{1/rq'}ds\\
&\ \leq\sup_{t\in I}\int^{t}_{0}\Big(\int_{\mathbb{M}}|f(s,x)|^{r/(b+1)}dx\Big)^{\theta(b+1)/r}\\
&\ \ \ \times\Big(\int_{\mathbb{M}}|f(s,x)|^{p/(b+1)}d\mu(x)\Big)^{(b+1)(1-\theta)/p}ds\\
&\ \leq\sup_{t\in I}\int^{t}_{0}\||f(s,\cdot)|^{1/(b+1)}\|^{\theta(b+1)}_{L^{r}(\mathbb{M})}\||f(s,\cdot)|^{1/(b+1)}\|^{(1-\theta)(b+1)}_{L^{p}(\mathbb{M})}ds\\
&\ \leq\||f|^{1/(b+1)}\|^{(b+1)\theta}_{L^{\infty}(I; L^{r})}\sup_{t\in I}\int^{t}_{0}\||f(s,\cdot)|^{1/(b+1)}\|^{(1-\theta)(b+1)}_{L^{p}(\mathbb{M})}ds.
\end{align*}
Applying the H\"older inequality on $s$, we obtain
\begin{align*}
&\|G(f)\|_{L^{\infty}(I; L^{r})}\\
&\ \leq\||f|^{1/(b+1)}\|^{(b+1)\theta}_{L^{\infty}(I; L^{r})}\sup_{t\in I}\Big(\int^{t}_{0}ds\Big)^{1-(b+1)(1-\theta)/q}\\
&\ \ \ \times\Big(\int^{t}_{0}\||f(s,\cdot)|^{1/(b+1)}\|^{(1-\theta)(b+1)q/(b+1)(1-\theta)}_{L^{p}}ds\Big)^{(b+1)(1-\theta)/q}\\
&\ \lesssim T^{1-(b+1)(1-\theta)/q}\||f|^{1/(b+1)}\|^{(b+1)\theta}_{L^{\infty}(I; L^{r})}\||f|^{1/(b+1)}\|^{(b+1)(1-\theta)}_{L^{q}(I; L^{p})}\\
&\ \lesssim T^{1-b\beta^{\star}/2\alpha
r}\||f|^{1/(b+1)}\|^{(b+1)\theta}_{L^{\infty}(I;
L^{r})}\||f|^{1/(b+1)}\|^{(b+1)(1-\theta)}_{L^{q}(I; L^{p})},
\end{align*}
where  we have used the fact that  $(q,p,r)$ is an admissible
triplet in the last inequality.

Next we prove (ii). For the case $p<r(b+1)$, using Lemma
\ref{le-4.1}, we obtain
\begin{align*}
&\|G(f)\|_{L^{q}(I; L^{p})}\\
&\ =\Big[\int_{0}^{T}\Big\|\int^{t}_{0}e^{-(t-s){\mathcal{L}}^{\alpha}}f(s,x)ds\Big\|^{q}_{L^p(\mathbb{M})}dt\Big]^{1/q}\\
&\ \leq\Big[\int_{0}^{T}\Big(\int^{t}_{0}\Big\|e^{-(t-s){\mathcal{L}}^{\alpha}}f(s,x)\Big\|_{L^p(\mathbb{M})}ds\Big)^{q}dt\Big]^{1/q}\\
&\ \leq\Big[\int_{0}^{T}\Big(\int^{t}_{0}(t-s)^{-\beta^{\star}[(1+b)/p-1/p]/2\alpha}\|f(s,\cdot)\|_{L^{p/(b+1)}(\mathbb{M})}ds\Big)^{q}dt\Big]^{1/q}.
\end{align*}
Choosing $\chi$ such that $1+1/q=(1+b)/q+1/\chi$, we apply Young's
inequality on the variable $t$ to obtain
\begin{align*}
&\|G(f)\|_{L^{q}(I; L^{p})}\\
&\ \leq\Big\|\int^{t}_{0}(t-s)^{-n[(1+b)/p-1/p]/2\alpha}\|f(s,\cdot)\|_{L^{p/(b+1)}(\mathbb{M})}ds\Big\|_{L^{q}(dt)}\\
&\ \leq\|f\|_{L^{q/(b+1)}(I; L^{p/(b+1)})}\Big(\int^{T}_{0}t^{-\beta^{\star}[(b+1)/p-1/p]\chi/2\alpha}dt\Big)^{1/\chi}\\
&\ \lesssim T^{1-b\beta^{\star}/2\alpha r}\|f\|_{L^{q/(b+1)}(I; L^{p/(b+1)})}.
\end{align*}

For the case $p\geq  r(b+1)$, we apply Lemma \ref{le-4.1} again to get
\begin{align*}
&\|G(f)\|_{L^{q}(I; L^{p})}\\
&\ \leq\Big\|\int^{t}_{0}\Big\|e^{-(t-s){\mathcal{L}}^{\alpha}}f(s,\cdot)\Big\|_{L^p(\mathbb{M})}ds\Big\|_{L^{q}(dt)}\\
&\ \leq\Big\|\int^{t}_{0}(t-s)^{-\beta^{\star}(1/r-1/p)/2\alpha}\||f(s,\cdot)|^{1/(b+1)}\|_{L^{r(b+1)}(\mathbb{M})}^{b+1}ds\Big\|_{L^{q}(dt)}.
\end{align*}
Take  a constant $\theta\in (0,1)$ such that
$$\begin{cases}
&r(b+1)\theta p'=r,\\
&r(b+1)(1-\theta)q'=p,\\
&1/p'+1/q'=1.
\end{cases}$$
Then $$\theta=[p-r(b+1)]/(p-r).$$ Using H\"older's inequality on the spatial variable, we obtain
\begin{eqnarray*}
&&\int^{t}_{0}(t-s)^{-\beta^{\star}(1/r-1/p)/2\alpha}\left\||f(s,\cdot)|^{1/(b+1)}\right\|_{L^{r(b+1)}(\mathbb{M})}^{b+1}ds\\
&&\leq\int^{t}_{0}(t-s)^{-\beta^{\star}(1/r-1/p)/2\alpha}\Big(\int_{\mathbb{M}}|f(s,x)|^{r(b+1)\theta p'/(b+1)}dx\Big)^{1/rp'}\\
&&\quad\times\Big(\int_{\mathbb{M}}|f(s,x)|^{r(b+1)(1-\theta)q'/(1+b)}dx\Big)^{1/rq'}ds\\
&&=\int^{t}_{0}\frac{\left\||f(s,\cdot)|^{1/(b+1)}\right\|_{L^{r}(\mathbb{M})}^{\theta(b+1)}
\left\||f(s,\cdot)|^{1/(b+1)}\right\|_{L^{p}(\mathbb{M})}^{(1-\theta)(b+1)}}{(t-s)^{\beta^{\star}(1/r-1/p)/2\alpha}} ds\\
&&\leq\int^{t}_{0}\frac{\left\||f|^{1/(b+1)}\right\|_{L^{\infty}(I;
L^{r})}^{\theta(b+1)}
\left\||f(s,\cdot)|^{1/(b+1)}\right\|_{L^{p}(\mathbb{M})}^{(1-\theta)(b+1)}}{(t-s)^{\beta^{\star}(1/r-1/p)/2\alpha}} ds.
\end{eqnarray*}
The above estimation indicates
\begin{align*}
&\|G(f)\|_{L^{q}(I; L^{p})}\\
&\ \leq \left\||f|^{1/(b+1)}\right\|_{L^{\infty}(I; L^{r})}^{\theta(b+1)}\\
&\ \ \ \times\left\|\int^{t}_{0}(t-s)^{-\beta^{\star}(1/r-1/p)/2\alpha}\left\||f(s,\cdot)|^{1/(b+1)}\right\|_{L^{p}(\mathbb{M})}^{(1-\theta)(b+1)}ds\right\|_{L^{q}(dt)}.
\end{align*}
Suppose that $\chi$ obeys $$1+1/q=(1+b)(1-\theta)/q+1/\chi.
$$
Notice that
$(q,p,r)$ is an admissible triplet, i.e.,
$$1/q=\beta^{\star}(1/r-1/p)/2\alpha.$$  So, we use Young's inequality
for the variable $t$ to get
\begin{align*}
&\|G(f)\|_{L^{q}(I; L^{p})}\\
&\ \leq\left\||f|^{1/(b+1)}\right\|_{L^{\infty}(I; L^{r})}^{\theta(b+1)}\left\||f|^{1/(b+1)}\right\|_{L^{q}(I; L^{p})}^{(b+1)(1-\theta)}\Big(\int^{T}_{0}t^{-n(1/r-1/p)\chi/2\alpha}dt\Big)^{1/\chi}\\
&\ \ \lesssim T^{1-nb/2r\alpha}\left\||f|^{1/(b+1)}\right\|_{L^{\infty}(I; L^{r})}^{\theta(b+1)}\left\||f|^{1/(b+1)}\right\|_{L^{q}(I; L^{p})}^{(b+1)(1-\theta)}.
\end{align*}
This completes the proof of Theorem \ref{th-1.7}.
\end{proof}

We can also establish the following assertion.

\begin{theorem}\label{th-4.3}
Assume that the heat kernel $p_t(\cdot,\cdot)$ satisfies {\bf (A1)}
and  the measure $\mu$ satisfies (\ref{eq1.3}). For $b>0$ and $T>0$,
let $r_{0}=b\beta^{\star}/2\alpha$, $I=[0,T)$. Assume that $r\geq
r_{0}>1$ and that $(q,p,r)$ is a generalized admissible triplet
satisfying  $p>b+1$.
\begin{itemize}
\item[\rm(i)] If $f\in C_{q/(b+1)}(I; L^{p/(b+1)})$, then for $p<r(b+1)$,
$$\|G(f)\|_{L^{\infty}(I; L^{r})}\leq T^{1-b\beta^{\star}/2r\alpha}\|f\|_{C_{q/(b+1)}(I; L^{p/(b+1)})}.$$
For $p\geq r(b+1)$ and
$\theta=[p-r(b+1)]/(b+1)(p-r)$,
$$\|G(f)\|_{L^{\infty}(I; L^{r})}\leq CT^{1-b\beta^{\star}/2r\alpha}\||f|^{1/(b+1)}\|^{\theta(b+1)}_{L^{\infty}(I; L^{r})}
\||f|^{1/(b+1)}\|^{(1-\theta)(b+1)}_{C_{q}(I; L^{p})}.$$
\item[\rm(ii)] If $f\in C_{q/(b+1)}(I; L^{p/(b+1)})$, then for $p<r(b+1)$,
$$\|G(f)\|_{C_{q}(I; L^{p})}\leq T^{1-b\beta^{\star}/2r\alpha}\|f\|_{C_{q/(b+1)}(I; L^{p/(b+1)})}.$$
For $p\geq r(b+1)$ and $\theta=[p-r(b+1)]/(b+1)(p-r)$,
$$\|G(f)\|_{C_{q}(I; L^{p})}\lesssim T^{1-b\beta^{\star}/2r\alpha}\||f|^{1/(b+1)}\|^{\theta(b+1)}_{L^{\infty}(I; L^{r})}
\||f|^{1/(b+1)}\|^{(1-\theta)(b+1)}_{C_{q}(I; L^{p})}.$$
\end{itemize}
\end{theorem}
\begin{proof}
(i) For the case $p<r(b+1)$, we use Lemma \ref{le-4.1} to derive
$$\|e^{-(t-s){\mathcal{L}}^{\alpha}}f(s,\cdot)\|_{L^{r}(\mathbb{M})}\lesssim (t-s)^{-\beta^{\star}[(b+1)/p-1/r]/2\alpha}\|f(s,\cdot)\|_{L^{p/(b+1)}(\mathbb{M})},$$
thereby getting
\begin{align*}
&\|G(f)\|_{L^{\infty}(I; L^{r})}\\
&\ =\sup_{t\in[0,T)}\|G(f)\|_{L^{r}(\mathbb{M})}\\
&\ \leq\sup_{t\in[0,T)}\int^{t}_{0}\|e^{-(t-s){\mathcal{L}}^{\alpha}}f(s,\cdot)\|_{L^{r}(\mathbb{M})}ds\\
&\ \lesssim \sup_{t\in[0,T)}\int^{t}_{0}(t-s)^{-\beta^{\star}[(b+1)/p-1/r]/2\alpha}\|f(s,\cdot)\|_{L^{p/(b+1)}(\mathbb{M})}ds\\
&\ \lesssim \|f\|_{C_{q/(b+1)}(I; L^{p/(b+1)})}\sup_{t\in[0,T)}\int^{t}_{0}(t-s)^{-\beta^{\star}[(b+1)/p-1/r]/2\alpha}s^{-(b+1)/q}ds\\
&\ \lesssim \|f\|_{C_{q/(b+1)}(I; L^{p/(b+1)})}\sup_{t\in[0,T)}t^{1-b\beta^{\star}/2\alpha r}\int^{1}_{0}\frac{u^{-(b+1)/q}}{(1-u)^{\beta^{\star}[(b+1)/p-1/r]}}du\\
&\ \lesssim T^{1-b\beta^{\star}/2\alpha r}\|f\|_{C_{q/(b+1)}(I; L^{p/(b+1)})}.
\end{align*}

Next  we consider the case $p\geq r(b+1)$. Lemma \ref{le-4.1}
implies
$$\|e^{-(t-s){\mathcal{L}}^{\alpha}}f(s,\cdot)\|_{L^{r}(\mathbb{M})}\leq C\|f(s,\cdot)\|_{L^{r}(\mathbb{M})}.$$
If
$$\begin{cases}
&\theta\in (0,1);\\
&r(b+1)\theta p'=r;\\
&r(b+1)(1-\theta)q'=p;\\
&1/p'+1/q'=1,
\end{cases}$$
then $$\theta=[p-r(b+1)]/(p-r),
$$
and hence H\"older's inequality is used to derive
\begin{align*}
&\|G(f)\|_{L^{\infty}(I; L^{r})}\\
&\ \leq\sup_{t\in [0,T)}\int^{t}_{0}\|e^{-(t-s){\mathcal{L}}^{\alpha}}f(s,\cdot)\|_{L^{r}(\mathbb{M})}ds\\
&\ \lesssim \sup_{t\in [0,T)}\int^{t}_{0}\|f(s,\cdot)\|_{L^{r}(\mathbb{M})}ds\\
&\ =\sup_{t\in [0,T)}\int^{t}_{0}\left\||f(s,\cdot)|^{1/(b+1)}\right\|_{L^{r(b+1)}(\mathbb{M})}ds\\
&\ =\sup_{t\in [0,T)}\int^{t}_{0}\left\||f(s,\cdot)|^{1/(b+1)}\right\|^{\theta(b+1)}_{L^{r}(\mathbb{M})}\left\||f(s,\cdot)|^{1/(b+1)}\right\|^{(b+1)(1-\theta)}_{L^{p}(\mathbb{M})}ds\\
&\ \lesssim \left\||f|^{1/(b+1)}\right\|^{\theta(b+1)}_{L^{\infty}(I; L^{r})}\left\||f(s,\cdot)|^{1/(b+1)}\right\|^{(b+1)(1-\theta)}_{C_{q}(I;L^{p})}\\
&\ \ \ \times \Big(\sup_{t\in [0,T)}\int^{t}_{0}s^{-(b+1)(1-\theta)/q}ds\Big)\\
&\lesssim T^{1-\beta^{\star} b/2\alpha
r}\left\||f|^{1/(b+1)}\right\|^{\theta(b+1)}_{L^{\infty}(I;
L^{r})}\left\||f(s,\cdot)|^{1/(b+1)}\right\|^{(b+1)(1-\theta)}_{C_{q}(I;L^{p})}.
\end{align*}

Finally,  we begin to prove (ii). Suppose $$f\in C_{q/(b+1)}(I;
L^{p/(b+1)}).
$$
 For the case $p<r(b+1)$, we use Lemma \ref{le-4.1} to obtain
$$\left\|e^{-(t-s){\mathcal{L}}^{\alpha}}f(s,\cdot)\right\|_{L^{p}(\mathbb{M})}\leq C(t-s)^{-\beta^{\star}[(b+1)/p-1/p]/2\alpha}\left\|f(s,\cdot)\right\|_{L^{p/(b+1)}(\mathbb{M})}.$$
Then a direct computation gives
\begin{align*}
&\|G(f)\|_{C_{q}(I; L^{p})}\\
&\ :=\sup_{t\in [0,T)}t^{1/q}\|G(f)\|_{L^{p}(\mathbb{M})}\\
&\ \leq\sup_{t\in [0,T)}t^{1/q}\left\|\int^{t}_{0}e^{-(t-s){\mathcal{L}}^{\alpha}}f(s,\cdot)ds\right\|_{L^{p}(\mathbb{M})}\\
&\ \leq\sup_{t\in [0,T)}t^{1/q}\int^{t}_{0}\left\|e^{-(t-s){\mathcal{L}}^{\alpha}}f(s,\cdot)\right\|_{L^{p}(\mathbb{M})}ds\\
&\ \leq\sup_{t\in [0,T)}t^{1/q}\int^{t}_{0}(t-s)^{-\beta^{\star}[(b+1)/p-1/p]/2\alpha}\left\|f(s,\cdot)\right\|_{L^{p/(b+1)}(\mathbb{M})}ds\\
&\ \leq\left\|f\right\|_{C_{q/(b+1)}(I;L^{p/(b+1)})}\sup_{t\in [0,T)}t^{1/q}\int^{t}_{0}(t-s)^{-\beta^{\star}[(b+1)/p-1/p]/2\alpha}s^{-(b+1)/q}ds\\
&\ \lesssim T^{1-\beta^{\star} b/2\alpha
r}\left\|f\right\|_{C_{q/(b+1)}(I;L^{p/(b+1)})}.
\end{align*}

For the case $p\geq r(b+1)$, Lemma \ref{le-4.1} implies
\begin{align*}
&\left\|e^{-(t-s){\mathcal{L}}^{\alpha}}f(s,\cdot)\right\|_{L^{p}(\mathbb{M})}\\
&\ \lesssim (t-s)^{-\beta^{\star}(1/r-1/p)/2\alpha}\left\|f(s,\cdot)\right\|_{L^{r}(\mathbb{M})}\\
&\ =(t-s)^{-\beta^{\star}(1/r-1/p)/2\alpha}\left\||f(s,\cdot)|^{1/(b+1)}\right\|_{L^{r(b+1)}(\mathbb{M})}^{b+1}.
\end{align*}
Similar to (i), taking $\theta=[p-r(b+1)]/(p-r)$,   H\"older's
inequality can be applied to get
\begin{align*}
&\|G(f)\|_{C_{q}(I; L^{p})}\\
&\ \leq\sup_{t\in[0, T)}t^{1/q}\int^{t}_{0}(t-s)^{-\beta^{\star}(1/r-1/p)/2\alpha}\left\||f(s,\cdot)|^{1/(b+1)}\right\|_{L^{r(b+1)}(\mathbb{M})}^{b+1}ds\\
&\ \leq\sup_{t\in[0, T)}t^{1/q}\int^{t}_{0}(t-s)^{-\beta^{\star}(1/r-1/p)/2\alpha}\left\||f(s,\cdot)|^{1/(b+1)}\right\|_{L^{r}(\mathbb{M})}^{(b+1)\theta}\\
&\ \ \times\left\||f(s,\cdot)|^{1/(b+1)}\right\|_{L^{p}(\mathbb{M})}^{(b+1)(1-\theta)}ds\\
&\ \leq\left\||f|^{1/(b+1)}\right\|_{L^{\infty}(I; L^{r})}^{(b+1)\theta}\left\||f|^{1/(b+1)}\right\|_{C_{q}(I; L^{p})}^{(b+1)(1-\theta)}\\
&\ \ \times\Big(\sup_{t\in[0, T)}t^{1/q}\int^{t}_{0}(t-s)^{-\beta^{\star}(1/r-1/p)/2\alpha}s^{-(b+1)(1-\theta)/q}ds\Big)\\
&\ \lesssim T^{1-\beta^{\star}
b/2r\alpha}\left\||f|^{1/(b+1)}\right\|_{L^{\infty}(I;
L^{r})}^{(b+1)\theta}\left\||f|^{1/(b+1)}\right\|_{C_{q}(I;
L^{p})}^{(b+1)(1-\theta)}.
\end{align*}
This completes the proof of Theorem \ref{th-4.3}.

\end{proof}

\subsection{Estimation - part B}\label{sec-4}
 This part is designed to give the Strichartz type estimates for solutions to (\ref{eq-4.1}) whose situation $\mathbb{M}=\mathbb{R}^n$ was investigated in Jiang-Xiao-Yang-Zhai's paper  \cite{jiangrenjin2}.

For $(t_0,x_0,r)\in    (0,\infty)\times \mathbb{M}\times(0,\infty)$,
the parabolic ball is defined as
$$B^{(\alpha)}_{r}(t_{0},x_{0}):=\Big\{(t,x)\in  \mathbb{M}_+:\
r^{2\alpha}<t-t_{0}<2r^{2\alpha}\ \&\ d(x,x_0)<r\Big\}
$$
and its volume is denoted by
$\widetilde{\mu}(B_{r}^{(\alpha)}{(t_0,x_0)})= r^{2\alpha}
{\mu}(B_{r}^{(\alpha)}{(t_0,x_0)})$, where $\widetilde{\mu}$ is the
product measure of $\mu$ and the Lebesgue measure on $(0,\infty)$.

As in \cite{Zhai}, we immediately have
\begin{theorem}\label{t11}
Suppose that the heat kernel $p_t(\cdot,\cdot)$ satisfies {\bf (A1)} and  the
measure $\mu$ satisfies (\ref{eq1.3}).
   If
$$
\begin{cases}
1\le p<\tilde{p}\le\infty;\\
1<q<\tilde{q}<\infty;\\
\big({1}/{q}-{1}/{\tilde{q}}\big)+{\beta^{\star}}\big({1}/{p}-{1}/{\tilde{p}}\big)/{2\alpha}=1,
\end{cases}
$$
then
$$
\|G(F)\|_{L^{\tilde{q}}((0,\infty);L^{\tilde{p}})}\lesssim\|F\|_{L^{q}((0,\infty);L^{p})}.
$$
\end{theorem}

\begin{proof}
 Assume that $(q, p, 2)$ and $(q_1, p_1, 2)$
satisfy $1\le p<\tilde{p}\le\infty$, $1<q<\tilde{q}<\infty$ and
$\big({1}/{q}-{1}/{\tilde{q}}\big)+{\beta^{\star}}\big({1}/{p}-{1}/{\tilde{p}}\big)/{2\alpha}=1$.
Via Lemma \ref{le-4.1}, we have, for any $s<t,$
\begin{equation*}\| e^{-(t-s)\mathcal{L}^\alpha}F(s,x)\|_{L^{\tilde{p}}(\mathbb{M})}\lesssim |t-s|^{-{\beta^{\star}}({1}/{p}-{1}/{\tilde{p}})/{2\alpha}}\|F(s,x)\|_{L^{{p}}(\mathbb{M})}.
\end{equation*}
Then the Hardy-Littlewood-Sobolev inequality implies
\begin{eqnarray*}
&&\Big\|\int^t_0 e^{-(t-s)\mathcal{L}^\alpha}F(s,x)ds\Big\|_{L^{\tilde{q}}((0,\infty);L^{\tilde{p}})}\\
&&\lesssim \Big\| \int^t_0 \| e^{-(t-s)\mathcal{L}^\alpha}F(s,x)\|_{L^{\tilde{p}}(\mathbb{M})}ds\Big\|_{L^{\tilde{q}}(0,\infty)}\\
&&\lesssim \Big\| \int^t_0 |t-s|^{-{\beta^{\star}}({1}/{p}-{1}/{\tilde{p}})/{2\alpha}}\|F(s,x)\|_{L^{{p}}(\mathbb{M})}ds\Big\|_{L^{\tilde{q}}(0,\infty)}\\
&&\lesssim \|F\|_{L^{q}((0,\infty);L^{p})}.
\end{eqnarray*}

\end{proof}

\begin{theorem}\label{th-1.8}
The following statements are valid.
\begin{itemize}	
\item[\rm (i)] Suppose that the heat kernel $p_t(\cdot,\cdot)$ satisfies {\bf (A1)}-{\bf (A3)} and  the measure $\mu$ satisfies
(\ref{eq1.3}). If $p\in[1,\infty]$ and $f\in L^p( \mathbb{ M})$,
then $e^{-t{\mathcal{L}}^{\alpha}}(f)$ is continuous on $
\mathbb{M}_+$.

\item[\rm (ii)] Suppose that the heat kernel $p_t(x,y)$ satisfies {\bf (A1)} and  the measure $\mu$ satisfies
(\ref{eq1.3}).  If
$$
\begin{cases}
p\in[1,\infty);\\
1<q<\infty; \\
{\beta^{\star}}/{p}+{2\alpha}/{q}=2\alpha;\\
(t_0,x_0)\in   \mathbb{M}_+;\\
r_0=t_0^{ {1}/{2\az}};\\
0<\|F\|_{L^{q}((0,\infty);L^{p})}<\infty,
\end{cases}
$$
then there exists a constant $C>0$ such that
\begin{align*}
\frac{1}{\widetilde{\mu}(B_{r_0}^{(\alpha)}{(t_0,x_0)})}\iint_{B_{r_0}^{(\alpha)}{(t_0,x_0)}}\exp\lf(\frac{G( F)(t,x)}
   {C\|F\|_{L^{q}((0,\infty);L^{p})}}\r)^{{q}/{(q-1)}}\,d\mu(x)\,dt\ls 1.
\end{align*}

\item[\rm(iii)] Suppose that the heat kernel $p_t(\cdot,\cdot)$ satisfies {\bf (A1)}-{\bf (A3)} and  the measure $\mu$ satisfies
(\ref{eq1.3}).  If
$$
\begin{cases}
p\in[1,\infty);\\
1<q<\infty; \\
{\beta^{\star}}/{p}+{2\alpha}/{q}<2\alpha;\\
(t,x)\in   \mathbb{M}_+;\\
\|F\|_{L^{q}((0,\infty);L^{p})}<\infty,
\end{cases}
$$
then $G(F)$ is H\"older continuous in the sense that
\begin{align*}
   &|G(F)(t,x)-G(F)(t_0,x_0)|\\
   &\ \lesssim  \Big(   \min\{ |t_2-t_1|^{2-{1}/{2\alpha}-{1}/{q}-{\beta^{\star}}/{2\az p}}, |t_2-t_1|^{1-{1}/{q}-{\beta^{\star}}/{2\az p}}
   \}\\
   &\ \ +d(x,x_0)^{{2\az(q-1)}/{q}-{\beta^{\star}}/{p}}
   \Big)
   \|F\|_{L^{q}((0,\infty);L^{p})}
\end{align*}
holds for any two sufficient close points $(t_0,x_0), (t,x)\in
\mathbb{M}_+$.
\end{itemize}
\end{theorem}
\begin{proof}
(i) Let
$$
\begin{cases}
(t,x)\in \mathbb{M}_+;\\
(t_0,x_0)\in \mathbb{M}_+;\\
f\in L^p(  \mathbb{M});\\
p\in [1,\infty];\\
0\le t_1<t_2<\infty,
\end{cases}
$$
and $\{\eta^{\alpha}_t\}$ be a family of non-negative
continuous functions on $(0,+\infty)$ defined in (\ref{eq-heat}) such that for all $t>0$
\begin{equation*}e^{-t\mathcal{L}^\alpha}f(x)=\int^{\infty}_{0} \eta^{\alpha}_t(s)e^{-s\mathcal{L}}f(x)ds=\int_{\mathbb{M}} K^{\mathcal{L}}_{\alpha,t}(x,y)f(y)d\mu(y),
\end{equation*}
where \begin{equation*}K^{\mathcal{L}}_{\alpha,t}(x,y)=\int^{\infty}_{0} \eta^{\alpha}_t(s) p_s(x,y)ds
\end{equation*}
and $p_s(\cdot,\cdot)$ is the kernel of $e^{-s\mathcal{L}}$.

Via {\bf (A3)}, we know that
$K^{\mathcal{L}}_{\alpha,t}(\cdot,\cdot)$ is continuous  with
respect to the variable $x$. One gets that
$K^{\mathcal{L}}_{\alpha,t}
f(t_0,x)=e^{-t_0{\mathcal{L}}^\alpha}f(x)$ is also continuous.
Meanwhile, for $x\in  M$, one gets
$$
 e^{-t{\mathcal{L}}^\alpha} f(t_1,x)- e^{-t{\mathcal{L}}^\alpha}  f(t_2,x)=\int_{t_1}^{t_2}{\mathcal{L}}^\alpha e^{-t{\mathcal{L}}^\alpha}f(x)\,dt.
$$
Via Lemma \ref{le-4.1}, we know
$$
\Big\|{\mathcal{L}}^\alpha
e^{-t{\mathcal{L}}^\alpha}f\Big\|_{L^\infty(\mathbb{M})} \lesssim
\begin{cases}
   t^{-{1}/{2\alpha}-{\beta^{\star}}/{2\alpha p}}\|f\|_{L^p(\mathbb{M})},&  p\in [1,\infty);\\
    t^{-{1}/{2\alpha}}\|f\|_{L^\infty(\mathbb{M})},&  p=\infty,
\end{cases}
$$
and hence
\begin{align*}
&|e^{-t{\mathcal{L}}^\alpha} f(t_1,x)-e^{-t{\mathcal{L}}^\alpha}
f(t_2,x)|\\
&\ \lesssim \|f\|_{L^p(\mathbb{M})}
\begin{cases}\big|t_1^{1-{1}/{2\alpha}-{\beta^{\star}}/{2\alpha p}}-t_2^{1-{1}/{2\alpha}-{\beta^{\star}}/{2\alpha p}}\big|,& p\in [1,\infty);\\
|  t_1^{1-{1}/{2\alpha}} - t_2^{1-{1}/{2\alpha}} |,&  p=\infty.
\end{cases}
\end{align*}
By the above facts, we conclude that if $(t,x)\to (t_0,x_0)$ then
\begin{align*}
&|e^{-t{\mathcal{L}}^\alpha}f(t,x)-e^{-t{\mathcal{L}}^\alpha} f(t_0,x_0)|\\
&\le|e^{-t{\mathcal{L}}^\alpha} f(t_0,x)-e^{-t{\mathcal{L}}^\alpha} f(t_0,x_0)|
+|e^{-t{\mathcal{L}}^\alpha} f(t,x)-e^{-t{\mathcal{L}}^\alpha} f(t_0,x)|\to 0,
\end{align*}
and hence $e^{-t{\mathcal{L}}^\alpha}f$ is continuous on $ \mathbb{M}_+$.

(ii) Let $(t,x)\in   \mathbb{M}_+$ be fixed.
Then we have
$$
    |G(F)(t,x)|\le\int_0^t \int_\mathbb{M} K^{\mathcal{L}}_{\alpha,t-s} (x, y)|F(s,y)|\,d\mu(y)\,ds=\mathrm{I}+\mathrm{II},
    $$
    where
    $$    \begin{cases}
    \mathrm{I}:=\int_0^r\int_\mathbb{M} K^{\mathcal{L}}_{\alpha,t-s}(x,y)|F(s,y)|\,d\mu(y)\,ds; \\
    \mathrm{II}:=\int_r^t\int_\mathbb{M} K^{\mathcal{L}}_{\alpha,t-s}(x,y)|F(s,y)|\,d\mu(y)\,ds.
\end{cases}
$$
By using the H\"older inequality and the assumption $ \beta^{\star} /p+
{2\az}/q=2\az$, we get
\begin{align*}
\mathrm{I}&\lesssim\int_0^r\int_{\mathbb{M}}
   \frac{|t-s|}{(|t-s|^{{1}/{2\az}}+d(x,y))^{\beta^{\star}+2\alpha}}|F(s,y)|\,d\mu(y)\,ds\\
   &\ \lesssim\int_0^r\frac{|t-s| \|F(s,\cdot)\|_{L^p(\mathbb{M})}}{\lf(\int_{\mathbb{M}}
   \frac{d\mu(y)}{(|t-s|^{{1}/{2\az}}+d(x,y))^{-(\beta^{\star}+2\alpha)p/({p-1})}}\r)^{{(p-1)}/{p}}}\,ds\\
&\ \lesssim\int_0^r\frac{\|F(s,\cdot)\|_{L^p(\mathbb{M})}}{|t-s|^{{\beta^{\star}}/{2p\az}}}\,ds\\
&\ \lesssim\|F\|_{L^{q}((0,\infty);L^{p})}
\lf(\int_0^r\frac{ds}{|t-s|^{{q\beta^{\star}}/{2p\az(q-1)}}}\r)^{{(q-1)}/{q}}\\
&\ \lesssim\|F\|_{L^{q}((0,\infty);L^{p})}\lf(\ln \frac{t}{t-r}\r)^{(q-1)/q}.
\end{align*}
 Denote by $\mathrm{M}_\mathbb R$ the Hardy-Littlewood maximal function on $\mathbb R$.  Similarly, we obtain
\begin{align*}
\mathrm{II}&\lesssim\int_r^t\int_{ M}
   \frac{|t-s|}{(|t-s|^{{1}/{2\az}}+d(x,y))^{\beta^{\star}+2\alpha}}|F(s,y)|\,d\mu(y)\,ds\\
&\ \lesssim\int_r^t\frac{\|F(s,\cdot)\|_{L^p( M)}}{|t-s|^{{\beta^{\star}}/{2p\az}}}\,ds\\
&\ \lesssim\sum_{k=-\infty}^{0}\int_{t-2^k|t-r|}^{t-2^{k-1}|t-r|}
\frac{\|F(s,\cdot)\|_{L^p(\mathbb{M})}}{|t-s|^{{\beta^{\star}}/{2p\az}}}\,ds\\
&\ \lesssim\sum_{k=-\infty}^{0}\frac{1}{(2^k|t-r|)^{{\beta^{\star}}/{2p\az}}}
\int_{t-2^k|t-r|}^{t}
\|F(s,\cdot)\|_{L^p(\mathbb{M})}\,ds\\
&\ \lesssim\sum_{k=-\infty}^{0}(2^k|t-r|)^{1-{\beta^{\star}}/{2p\az}}\cm_\rr(\|F(\cdot,\cdot)\|_{L^p(\mathbb{M})})(t) \\
&\ \lesssim|t-r|^{1/q}\cm_\rr(\|F(\cdot,\cdot)\|_{L^p(\mathbb{M})})(t).
\end{align*}
Via choosing $$
r\in (0,t)\ \ \&\ \ |t-r|^{1/q}=\min\lf\{t^{1/q}, \,
\frac{\|F\|_{L^{q}((0,\infty);L^{p})}}{\cm_\rr(\|F(\cdot,\cdot)\|_{L^p(\mathbb{M})})(t)}\r\},
$$
we see
\begin{eqnarray*}
    |G(F)(t,x)|\lesssim\|F\|_{L^{q}((0,\infty);L^{p})}
    \max\Big\{1,\,\Big[\ln\frac{t^{1/q}\cm_{\rr} (\|F\|_{L^p(\mathbb{M})})(t)}
{\|F\|_{L^{q}((0,\infty);L^{p})}}\Big]^{{(q-1)}/{q}}\Big\}.
\end{eqnarray*}
Let $r_0=t_0^{{1}/{2\az}}$. Then there exists a  constant $C>0$ such
that
\begin{eqnarray*}
&&\iint_{B_{r_0}^{(\alpha)}{(t_0,x_0)}}\exp\Big(\frac{G(F)(t,x)}
   {C\|F\|_{L^{q}((0,\infty);L^{p})}}\Big)^{{q}/({q-1})}\,d\mu(x)\,dt\\
&&\hs\lesssim
\iint_{B_{r_0}^{(\alpha)}{(t_0,x_0)}}\frac{t^{1/q}\cm_{\rr}
   (\|F\|_{L^p(\mathbb{M})})(t)}
{\|F\|_{L^{q}((0,\infty);L^{p})}}\,d\mu(x)\,dt\\
&&\hs\lesssim
\mu(B_{r_0}^{(\alpha)}{(t_0,x_0)})t_0^{1/q}\int_{0}^{2t_0}
\frac{\cm_{\rr} (\|F\|_{L^p(\mathbb{M})})(t)}
{\|F\|_{L^{q}((0,\infty);L^{p})}}\,dt\\
&&\hs\lesssim t_0^{1/q}r_0^{2\az-2\az/q}\mu(B_{r_0}^{(\alpha)}{(t_0,x_0)})\\
&&\hs= \widetilde{\mu}(B_{r_0}^{(\alpha)}{(t_0,x_0)}),
\end{eqnarray*}
which completes the proof of (ii).

(iii) Given a point $(t_0,x_0)\in   \mathbb{M}_+$, let $x\in \mathbb{M}$ be sufficiently
close to $x_0$ and $\dz=d(x,x_0)$. Then
\begin{align*}
&|G(F)(t_0,x_0)- G(F)(t_0,x)|\\
&\ \leq \int_0^{t_0}\int_ \mathbb{M}|K^{\mathcal{L}}_{\alpha,t_0-s}(x_0,y)
-K^{\mathcal{L}}_{\alpha,t_0-s}(x,y)||F(y,s)|\,d\mu(y)\,ds\\
&\ :=\hs\mathrm{I}+\mathrm{II},
\end{align*}
where
$$
\begin{cases}
I:=\int_0^{t_0}\int_{B(x_0,3\dz)}|K^{\mathcal{L}}_{\alpha,t_0-s}(x_0,y)
-K^{\mathcal{L}}_{\alpha,t_0-s}(x,y)||F(y,s)|\,d\mu(y)ds;\\
II:=\int_0^{t_0}\int_{ \mathbb{M}\setminus B(x_0,3\dz)}|K^{\mathcal{L}}_{\alpha,t_0-s}(x_0,y)
-K^{\mathcal{L}}_{\alpha,t_0-s}(x,y)||F(y,s)|\,d\mu(y)ds.
\end{cases}
$$
We first estimate the term $I$. By Proposition \ref{prop-4}, we can see
\begin{eqnarray*}
&&\int_{0}^{t_0}\int_{B(x_0,3\dz)}|K^{\mathcal{L}}_{\alpha,t_0-s}(x_0,y)||F(y,s)|\,d\mu(y)\,ds\\
&&\leq \int_{0}^{t_{0}-(2\delta)^{2\alpha}}\int_{B(x_0,3\dz)}\Big(\frac{|t_{0}-s|}{|t_{0}-s|^{1+\beta^{\star}/2\alpha}}\Big)|F(y,s)|\,d\mu(y)ds\\
&&+ \int_{t_{0}-(2\delta)^{2\alpha}}^{t_{0}}\int_{B(x_0,3\dz)}\frac{|t_{0}-s|}{[|t_{0}-s|^{1/2\alpha}+d(y,x_{0})]^{\beta^{\star}+2\alpha}}|F(y,s)|\,d\mu(y)ds\\
&&:=I_{1}+I_{2}.
\end{eqnarray*}
Applying H\"older's inequality, we obtain
\begin{eqnarray*}
I_{1}&\leq&\int_{0}^{t_{0}-(2\delta)^{2\alpha}}\Big(\int_{B(x_0,3\delta)}|F(y,s)|^{p}d\mu(y)\Big)^{1/p}
\frac{\delta^{\beta^{\star}(1-1/p)}}{|t_{0}-s|^{\beta^{\star}/2\alpha}} ds\\
&\leq&\delta^{\beta^{\star}(1-1/p)}\|F(\cdot,\cdot)\|_{L^{q}((0,\infty); L^{p})}\Big(\int_{0}^{t_{0}-(2\delta)^{2\alpha}}\frac{ds}{|t_{0}-s|^{\beta^{\star} q/2\alpha(q-1)}}\Big)^{q-1/q}\\
&\leq&\|F(\cdot,\cdot)\|_{L^{q}((0,\infty);
L^{p})}\delta^{2\alpha(1-1/p)-\beta^{\star}/p}.
\end{eqnarray*}
Similarly, for $I_{2}$,  we have
\begin{eqnarray*}
I_{2}&\leq&\int_{t_{0}-(2\delta)^{2\alpha}}^{t_{0}}\|F(\cdot,s)\|_{L^{p}(\mathbb{M})}\Big(\int_{\mathbb{M}}
\frac{|t_{0}-s|^{p'}d\mu(y)}{[|t_{0}-s|^{1/2\alpha}+d(x_{0},y)]^{(\beta^{\star}+2\alpha)p'}}\Big)^{1/p'}ds
\end{eqnarray*} via letting $p'=p/(p-1)$.
A direct computation gives
\begin{eqnarray*}
&&\int_{\mathbb{M}}
\frac{|t_{0}-s|^{p'}d\mu(y)}{[|t_{0}-s|^{1/2\alpha}+d(x_{0},y)]^{(\beta^{\star}+2\alpha)p'}}\\
&&\ \ \leq\int_{d(y,x_{0})<|t_{0}-s|^{1/2\alpha}}\frac{|t_{0}-s|^{p'}d\mu(y)}{[|t_{0}-s|^{1/2\alpha}+d(x_{0},y)]^{(\beta^{\star}+2\alpha)p'}}\\
&&\ \ \ +\sum^{\infty}_{k=0}\int_{2^{k}|t_{0}-s|^{1/2\alpha}\leq d(y,x_{0})<2^{k+1}|t_{0}-s|^{1/2\alpha}}\frac{|t_{0}-s|^{p'}d\mu(y)}{[|t_{0}-s|^{1/2\alpha}+d(x_{0},y)]^{(\beta^{\star}+2\alpha)p'}}\\
&&\ \ \leq\frac{|t_{0}-s|^{\beta^{\star}/2\alpha}}{|t_{0}-s|^{\beta^{\star} p'/2\alpha}}+\sum^{\infty}_{k=0}\frac{2^{(k+1)\beta^{\star}}}{2^{k(\beta^{\star}+2\alpha)p'}}\frac{|t_{0}-s|^{p'+\beta^{\star}/2\alpha}}{|t_{0}-s|^{(\beta^{\star}/2\alpha+1)p'}}\\
&&\ \ \leq\frac{1}{|t_{0}-s|^{\beta^{\star}(p'-1)/2\alpha}},
\end{eqnarray*}
which implies
\begin{eqnarray*}
I_{2}&\leq&\int_{t_{0}-(2\delta)^{2\alpha}}^{t_{0}}\|F(\cdot,s)\|_{L^{p}(\mathbb{M})}\frac{ds}{|t_{0}-s|^{\beta^{\star}/2\alpha p}}\\
&\leq&\|F(\cdot,\cdot)\|_{L^{q}((0,\infty); L^{p})}\Big(\int_{t_{0}-(2\delta)^{2\alpha}}^{t_{0}}\frac{ds}{|t_{0}-s|^{q\beta^{\star}/2\alpha p(q-1)}}\Big)^{(q-1)/q}\\
&\leq&\|F(\cdot,\cdot)\|_{L^{q}((0,\infty);
L^{p})}\delta^{2\alpha(1-1/p)-\beta^{\star}/p}.
\end{eqnarray*}

The integral
$$\int_{0}^{t_0}\int_{B(x_0,3\dz)}|K^{\mathcal{L}}_{\alpha,t_0-s}(x,y)||F(y,s)|\,d\mu(y)\,ds$$
can be handled similarly. Summarizing the above estimates, we have
proved that
\begin{align*}
\mathrm{I} & \le \int_0^{t_0}\int_{B(x_0,3\dz)}|K^{\mathcal{L}}_{\alpha,t_0-s}(x_0,y)||F(y,s)|\,d\mu(y)\,ds\\
&\hs+ \int_0^{t_0}\int_{B(x,4\dz)}|K^{\mathcal{L}}_{\alpha,t_0-s}(x,y)||F(y,s)|\,d\mu(y)\,ds\\
   &\lesssim\|F\|_{L^{q}((0,\infty);L^{p})}d(x,x_0)^{{2\az(q-1)}/{q}-{\beta^{\star}}/{p}}.
\end{align*}

To estimate the term $\mathrm{II}$, we utilize {\bf (A3)} and
\begin{equation*}  K^{\mathcal{L}}_{\alpha,t}(x,y)=\int^{\infty}_{0} \eta^{\alpha}_{t}(s)  p_s(x,y)ds
\end{equation*} to derive
$$| K^{\mathcal{L}}_{\alpha,t}(x,y)- K^{\mathcal{L}}_{\alpha,t}(x_0,y)|\lesssim \frac{td(x,x_0)^{\varepsilon}}{\big[t^{1/{2\alpha}}+d(x,y)\big]^{(\beta^{\star}+\varepsilon+2\alpha)}}.
$$
This, plus  the H\"older inequality, yields
\begin{align*}
\mathrm{II} &\le \int_0^{t_0}\int_{\mathbb{M}\setminus B(x_0,3\dz)}{\dz^{\varepsilon}}
\Big(\frac{|t_0-s|}{(|t_0-s|^{{1}/{2\az}}+d(x_0,y))^{\beta^{\star}+\varepsilon+2\alpha}}\Big)|F(y,s)|\,d\mu(y)\,ds\\
&\ \lesssim\int_{0}^{t_0-(2\dz)^{2\az}}\int_{\mathbb{M}\setminus B(x_0,3\dz)}\Big(\frac{\dz^{\varepsilon} |t_0-s|}
{(|t_0-s|^{1/{2\az}}+d(x_0,y))^{\beta^{\star}+\varepsilon+2\alpha}}\Big)|F(y,s)|\,d\mu(y)\,ds\\
&\ \hs\hs+\int_{t_0-(2\dz)^{2\az}}^{t_0}\int_{\mathbb{M}\setminus B(x_0,3\dz)}\Big(\frac{\dz^{{\varepsilon}} |t_0-s|}
{d(x_0,y)^{\beta^{\star}+{\varepsilon}+2\alpha}}\Big)|F(y,s)|\,d\mu(y)\,ds\\
&\ \leq \mathrm{II}_1+\mathrm{II}_2,
\end{align*}
where
\begin{align*}
II_{1}&:=\int_{0}^{t_0-(2\dz)^{2\az}}|t_0-s|\delta^\varepsilon
\|F(\cdot,s)\|_{L^p(\mathbb{M})}\\
&\ \ \times \Big(\int_{\mathbb{M}\setminus B(x_0,3\dz)}
\frac{d\mu(y)}{(|t_0-s|^{1/{2\az}}+d(x_0,y))^{p'(\beta^{\star}+\varepsilon+2\alpha)}}\Big)^{{1}/{p'}}\,ds
\end{align*}
and
$$II_{2}:=\int_{t_0-(2\dz)^{2\az}}^{t_0}|t_0-s|\dz^\varepsilon
\|F(\cdot,s)\|_{L^p(\mathbb{M})} \Big(\int_{\mathbb{M}\setminus
B(x_0,3\dz)} \frac{d\mu(y)}
{d(x_0,y)^{p'(\beta^{\star}+\varepsilon+2\alpha)}}\Big)^{1/p'}\,ds.$$

Furthermore, using the H\"older inequality again,  we have
\begin{align*}
\mathrm{II}_1&\lesssim \int_{0}^{t_{0}-(2\dz)^{2\az}}\frac{\|F(\cdot,s)\|_{L^p(\mathbb{M})}}{
\Big(\int_{\mathbb{M}\setminus B(x_0,3\dz)} \frac{\dz^{\varepsilon p'}
|t_0-s|^{p'} d\mu(y)}
{(|t_0-s|^{1/{2\az}}+d(x_0,y))^{p'(\beta^{\star}+\varepsilon+2\alpha)}}\Big)^{-1/p'}}\,ds\\
&\ \lesssim \int_{0}^{t_0-(2\dz)^{2\az}}
\|F(\cdot,s)\|_{L^p(\mathbb{M})}
\Big(\sum^{\infty}_{j=1}\frac{\dz^{\varepsilon
p'}|t_0-s|^{p'}(3^j\delta)^{\beta^{\star}}}
{(|t_0-s|^{1/{2\az}}+3^j\delta)^{p'(\beta^{\star}+\varepsilon+2\alpha)}}\Big)^{1/p'}\,ds\\
&\ \lesssim
\int_{0}^{t_0-(2\dz)^{2\az}}
\|F(\cdot,s)\|_{L^p(\mathbb{M})}
\Big(\int^{\infty}_{3\delta}\frac{\dz^{\varepsilon p'-1} |t_0-s|^{p'}dr}
{(|t_0-s|^{1/{2\az}}+r)^{p'(\beta^{\star}+\varepsilon+2\alpha)-\beta^{\star}}}\Big)^{1/p'}\,ds\\
&\ \lesssim
\int_{0}^{t_0-(2\dz)^{2\az}}\|F(\cdot,s)\|_{L^p(\mathbb{M})}
\frac{\delta^{\varepsilon-{1}/{p'}}ds}
{(|t_0-s|^{1/{2\az}}+\delta)^{{\beta^{\star}}/{p}+\varepsilon-{1}/{p'}}}\\
&\ \lesssim\|F\|_{L^{q}((0,\infty);L^{p})}\int_{0}^{t_0-(2\dz)^{2\az}}
\frac{\delta^{\varepsilon-1/p'}ds}{(|t_0-s|^{1/{2\az}}+\delta)^{{\beta^{\star}}/{p}+\varepsilon-{1}/{p'}}}\\
&\ \lesssim\|F\|_{L^{q}((0,\infty);L^{p})}[d(x,x_0)]^{{2\az(q-1)}/{q}-{\beta^{\star}}/{p}},
\end{align*}
where ${1}/{p}+{1}/{p'}=1.$ A similar method as that of
$\mathrm{II}_1$ shows
\begin{eqnarray*}
\mathrm{II}_2&\lesssim&\|F\|_{L^{q}((0,\infty);L^{p})}d(x,x_0)^{{2\az(q-1)}/{q}-{\beta^{\star}}/{p}}.
\end{eqnarray*}
 Thus,
\begin{eqnarray*}
   &&|G(F)(t_0,x_0)- G(F)(t_0,x)|\lesssim
   \|F\|_{L^{q}((0,\infty);L^{p})}d(x,x_0)^{{2\az(q-1)}/{q}-{\beta^{\star}}/{p}}.
\end{eqnarray*}

Let $(x,t_1), (x,t_2)\in   \mathbb{M}_+$. Without loss of
generality, we may assume $t_1>t_2$, and   write
\begin{align*}
   &|G(F)(t_1,x)-G(F)(t_2,x)|\\
   &\ \le \int_0^{t_2}
   \lf|\big(e^{-(t_1-s){\mathcal{L}}^\az}-e^{-(t_2-s){\mathcal{L}}^\az}\big)F(x,s)\r|\,ds\\
   &\ \hs+\int_{t_2}^{t_1}
   \lf|(e^{-(t_1-s){\mathcal{L}}^\az})F(x,s)\r|\,ds\\
   &\ :=\mathrm{III}+\mathrm{IV}.
\end{align*}
Via Lemma \ref{le-4.1},  we obtain
\begin{align*}
   \mathrm{III}&\le \int_0^{t_2}\int_{t_2-s}^{t_1-s}|{\mathcal{L}}^\az
   e^{-r{\mathcal{L}}^\az}F(x,s)|\,dr\,ds\\
   &\ \le \int_0^{t_2}\int_{t_2-s}^{t_1-s}r^{-{1}/{2\alpha}-{\beta^{\star}}/{2\alpha p}}\|F(\cdot,s)\|_{L^p(\mathbb{M})}\,dr\,ds\\
   &\ \le \int_0^{t_2}\int_0^{t_1-t_2} (t_2-s+r)^{-{1}/{2\alpha}-{\beta^{\star}}/{2\az p}}
   \|F(\cdot,s)\|_{L^p(\mathbb{M})}\,dsdr\\
   &\ \le \int_0^{t_1-t_2}\int_0^{t_2} (t_2-s+r)^{-{1}/{2\alpha}-{\beta^{\star}}/{2\az p}}\|F(\cdot,s)\|_{L^p(\mathbb{M})}\,dsdr\\
   &\ \lesssim\|F\|_{L^{q}((0,\infty);L^{p})}\int_0^{t_1-t_2} r^{{(q-1)}/{q}-{1}/{2\alpha}-{\beta^{\star}}/{2\az p}}\,dr\\
   &\ \lesssim|t_2-t_1|^{2-{1}/{2\alpha}-{1}/{q}-{\beta^{\star}}/{2\az p}}\|F\|_{L^{q}((0,\infty);L^{p})}
\end{align*}
and
\begin{eqnarray*}
   \mathrm{IV}&\le& \int_{t_2}^{t_1}(t_1-s)^{-{\beta^{\star}}/{2\az p}}\|F(s,\cdot)\|_{L^p(\mathbb{M})}\,ds\\
   &\lesssim&|t_2-t_1|^{1-{1}/{q}-{\beta^{\star}}/{2\az p}}\|F\|_{L^{q}((0,\infty);L^{p})}.
\end{eqnarray*}
Therefore,
\begin{align*}
   &|G(F)(t_1,x)-G(F)(t_2,x)|\\
   &\ \lesssim  |t_2-t_1|^{1-{1}/{q}-{\beta^{\star}}/{2\az p}}
   \, \|F\|_{L^{q}((0,\infty);L^{p})}.
\end{align*}
Accordingly, if $(t,x)$ is close to
$(t_0,x_0)$, then
\begin{align*}
&|G(F)(t,x)-G(F)(t_0,x_0)|\\
&\ \leq |G(F)(t,x)-G(F)(t_0,x)|+|G(F)(t_0,x)-G(F)(t_0,x_0)|\\
&\ \lesssim\Big( \min\{ |t_2-t_1|^{2-{1}/{2\alpha}-{1}/{q}-{\beta^{\star}}/{2\az p}}, |t_2-t_1|^{1-{1}/{q}-{\beta^{\star}}/{2\az p}} \}\\
&\ \ +d(x,x_0)^{{2\az(q-1)}/{q}-{\beta^{\star}}/{p}}\Big)
\|F\|_{L^{q}((0,\infty);L^{p})},
\end{align*}
which completes the proof of (iii).
\end{proof}

 \section{$L^p$-capacities in $ \mathbb{M}_+$}\label{sec-cap}

Throughout this  section, we always assume that the heat kernel
$p_t(\cdot,\cdot)$ satisfies {\bf (A4)} and  the measure $\mu$
satisfies (\ref{eq1.3}). Since we assume that the metric space
$\mathbb{M}$ is locally compact, continuous  functions with compact
support are dense in $L^p(\mathbb{M})$ for $1<p<\infty$ (cf.
 \cite[Chapter 3]{Heinonen1}).  Denote by $C_0(\mathbb
M)$  and  $C_0(\mathbb M_+)$ the   spaces consisting  of continuous
functions with compact support on $ \mathbb M $ and  $ \mathbb M_+$,
respectively.

 For a compact subset $K$ of $
\mathbb{M}_+,$ let
$$
 {C}_{p}^{(\alpha)}(K) :=\inf\Big\{\|f\|^p_{ L^{p}(\mathbb{
M})}:\quad f\ge 0\ \&\ e^{-t{\mathcal{L}}^{\alpha}}f\geq 1_{K}\Big\}
$$ be $L^p$-capacities in $ \mathbb{M}_+$,
  where $1_{K}$ is the characteristic function of $K$. When $O$ is an open
subset of $ \mathbb{M}_+$, one defines
$$ {C}_{p}^{(\alpha)}(O) :=\sup\Big\{ {C}_{p}^{(\alpha)}(K): \mathrm{compact}\, K\subseteq O\Big\}$$
and hence for any set $E\subseteq \mathbb{M}_+$, one sets
$$ {C}_{p}^{(\alpha)}(E) :=\sup\Big\{ {C}_{p}^{(\alpha)}(K): \mathrm{open}\, O\supseteq  E\Big\}.$$

\subsection{Duality of $L^p$-capacity} To establish the adjoint formulation of $C^{(\alpha)}_{p}$, we need to find out the adjoint operator of $e^{-t{\mathcal{L}}^{\alpha}}$. Note that
for any $f\in C_{0}(\mathbb{M})$ and $G\in C_{0}(\mathbb{M}_+)$, one has
$$\int_{\mathbb M_+}e^{-t{\mathcal{L}}^{\alpha}} f(x)
G(t,x)\,d\mu(x)dt=\int_{\mathbb
M}f(y)\Big(\int_{\mathbb{M}_+}K^{\mathcal{L}}_{\alpha,t}(x,y)G(t,x)d\mu(x)dt\Big)
d\mu(y).$$ Thus, the adjoint operator, denoted by
$(e^{-t{\mathcal{L}}^{\alpha}})^{\ast}$, is defined as follows. For
all $(t,x)\in \mathbb M_+$,
$$
(e^{-t{\mathcal{L}}^{\alpha}})^{\ast}G(t,x):=\int_{\mathbb
M_+}K^{\mathcal{L}}_{\alpha,t}(x,y)G(t,x)d\mu(x)dt \ \forall\ \ G\in
C_0(\mathbb M_+).
$$
The definition of $(e^{-t{\mathcal{L}}^{\alpha}})^{\ast}$ can be
extended to the family of Borel measures $\nu$ with compact support
in $\mathbb M_+$. In fact, note that if $f$ is continuous and has a
compact support in $\mathbb M$ and $\|\nu\|_1$ stands for the total
variation of $\nu$,  then a simple calculation with the following
estimate (cf. Proposition \ref{prop-1})
$$
K^{\mathcal{L}}_{\alpha,t}(x,y)\lesssim\frac{
t}{(t^{1/2\alpha}+d(x,y))^{\beta^{\star}+2\alpha}} \quad\forall\ \
(t,x)\in\mathbb M_+,
$$
gives
$$
\Big|\int_{\mathbb M_+}e^{-t{\mathcal{L}}^{\alpha}}
f\,d\nu\Big|\lesssim\|\nu\|_1\sup_{x\in\mathbb M}|f(x)|.
$$
Hence, using Riesz representation theorem, we conclude that there exists
a Borel measure $\tilde{\nu}$ on $\mathbb M$ such that
$$
\int_{\mathbb M_+}e^{-t{\mathcal{L}}^{\alpha}} f\,d\nu=\int_{\mathbb
M}f\,d\tilde{\nu}.
$$
This indicates that $(e^{-t{\mathcal{L}}^{\alpha}})^{\ast}\nu$ may
 be given as
$$(e^{-t{\mathcal{L}}^{\alpha}})^{\ast}\nu(x)=\int_{\mathbb{M}_+}K^{\mathcal{L}}_{\alpha,t}(y,x)d\nu(t,y).$$

Next, we obtain a dual description of the capacity via the above
analysis.

\begin{proposition}\label{p22} Given $p\in(1,\infty)$. For a compact subset $K$ of $\mathbb M_+$,  let $\mathcal{M_+}(K)$ be the class of all positive measures on $\mathbb M_+$ supported
on $K$. Then

$\mathrm{ (i)}$
$$
C_{p}^{(\alpha)}(K)=\sup\big\{\|\nu\|_1^{p}:\ \nu\in\mathcal{M_+}(K)
\ \ \&\ \
\|(e^{-t{\mathcal{L}}^{\alpha}})^{\ast}\nu\|_{L^{p'}(\mathbb M)}\le
1\big\}:=\tilde{C}_{p}^{(\alpha)}(K).
$$

$\mathrm{ (ii)}$ There exists a $\nu_K\in \mathcal{M_+}(K)$ such
that
\begin{align*}\nu_K(K)&=\int_{\mathbb{M}} \big((e^{-t{\mathcal{L}}^{\alpha}})^{\ast}\nu_K(x)\big)^{p'}d\mu(x)\\
&=\int_{\mathbb M_+} e^{-t{\mathcal{L}}^{\alpha}}\big((e^{-t{\mathcal{L}}^{\alpha}})^{\ast}\nu_K\big)^{p'-1}d\nu_K
=C_{p}^{(\alpha)}(K).
\end{align*}

\end{proposition}

\begin{proof} (i) Since
\begin{eqnarray*}
\|\nu\|_1&=&\nu(K)\\
&\leq& \int_{\mathbb M_+}e^{-t{\mathcal{L}}^{\alpha}} f\,d\nu\\
&=&\int_{\mathbb M}f\,(e^{-t{\mathcal{L}}^{\alpha}})^{\ast}\nu\, d\mu(x)\\
&\leq&\|f\|_{L^{p}(\mathbb M)}
\|(e^{-t{\mathcal{L}}^{\alpha}})^{\ast}\nu\|_{L^{p'}(\mathbb M)},
\end{eqnarray*}
then we have
\begin{equation*}
\tilde{C}_{p}^{(\alpha)}(K)\leq C_{p}^{(\alpha)}(K)
\end{equation*}
for any compact set $K\subset \mathbb M_+$. Moreover, this last
inequality is actually an equality - in fact, if
$$
\begin{cases}
X=\{\nu:\ \ \nu\in\mathcal{M}_+(K)\ \&\ \nu(K)=1\};\\
Y=\Big\{f:\ \ 0\le f\in  L^p(\mathbb M)\ \&\ \|f\|_{ L^p(\mathbb M)}\le 1\Big\};\\
Z=\Big\{f:\ \ 0\le f\in L^p(\mathbb M)\ \&\ e^{-t{\mathcal{L}}^{\alpha}} f\ge 1_K\Big\};\\
\mathsf{E}(\nu,f)=\int_{\mathbb
M}[(e^{-t{\mathcal{L}}^{\alpha}})^{\ast}\nu]
f\,d\mu(x)=\int_{\mathbb M_+}e^{-t{\mathcal{L}}^{\alpha}} f\,d\nu,
\end{cases}
$$
then combining   the  easy computation  with
\cite[Theorem 2.4.1]{AH} gives
\begin{eqnarray*}
\min_{\nu\in\mathcal{M}_+(K)}
\frac{\|(e^{-t{\mathcal{L}}^{\alpha}})^{\ast}\nu\|_{L^{p'}(\mathbb
M)}}{\nu(K)}&=&
\min_{\nu\in X}\sup_{f\in Y}\mathsf{E}(\nu,f)\\
&=&\sup_{f\in Y}\min_{\nu\in X}\mathsf{E}(\nu,f)\\
&=&\sup_{0\le f\in L^p(\mathbb M)}\frac{\min_{(t,x)\in K}e^{-t{\mathcal{L}}^{\alpha}}f(x)}{\|f\|_{L^{p}(\mathbb M)}}\\
&=&\sup_{f\in Z}\|f\|^{-1}_{L^p(\mathbb M)}\\
&=&\big(C_{p}^{(\alpha)}(K)\big)^{-1/p},
\end{eqnarray*}
and hence $\tilde{C}_{p}^{(\alpha)}(K)\geq C_{p}^{(\alpha)}(K),$
which shows  the desired equality.

(ii) According to (i), we may select a sequence  $\{\nu_j\}$ in
$\mathcal{M}_+(K)$ such that
$$
\begin{cases}
\|(e^{-t{\mathcal{L}}^{\alpha}})^{\ast}\nu_j\|_{L^{p'}(M)}\le 1;\\
\lim_{j\rightarrow\infty}[\nu_j(K)]^{p}=C_{p}^{(\alpha)}(K);\\
\nu_j\   \mathrm{has }\   \mathrm{a}\   \mathrm{weak}^{*}\
\mathrm{limit}\ \nu\in \mathcal{M}_+(K).
\end{cases}
$$
Then $\nu(K)^p=C_{p}^{(\alpha)}(K).$ Note that
$(e^{-t{\mathcal{L}}^{\alpha}})^{\ast}\nu$ is lower semicontinuous
on $\mathcal{M}_+(K)$. Therefore,
$\|(e^{-t{\mathcal{L}}^{\alpha}})^{\ast}\nu_j\|_{L^{p'}(\mathbb{M})}\le
1$, and (i) implies  its equality holds.

Setting $\nu_K=C_{p}^{(\alpha)}(K)^{{1}/{p'}}\nu$, then
$$\nu_K(K)=\int_{\mathbb{M}}\big((e^{-t{\mathcal{L}}^{\alpha}})^{\ast}\nu_K (x)\big)^{p'}d\mu(x)=C_{p}^{(\alpha)}(K).$$

 Assume that $f_K$ is the function in the definition of
 $C_{p}^{(\alpha)}(K)$ satisfying
$$\| f_K\|^p_{L^p(\mathbb{M})}=C_{p}^{(\alpha)}(K)\ \mathrm{and} \ e^{-t{\mathcal{L}}^{\alpha}}f_K\ge 1\ \mathrm{on}\ K.$$

Using (i), we have $$\nu_k(\{(t,x)\in K:
e^{-t{\mathcal{L}}^{\alpha}}f_K(x)<1\})=0 $$ and furthermore,
$$e^{-t{\mathcal{L}}^{\alpha}}f_K=e^{-t{\mathcal{L}}^{\alpha}}( (e^{-t{\mathcal{L}}^{\alpha}})^{\ast}\nu_K)^{p'-1}\ge 1\ \mathrm{for}\ a.e. \ \nu_K\ \mathrm{on}\ K.$$
So, combining the Fubini theorem with the H\"{o}lder inequality, it
can be deduced that
\begin{eqnarray*}C_{p}^{(\alpha)}(K)&=& \nu_K(K)\le
\int_{\mathbb{M}_+}e^{-t{\mathcal{L}}^{\alpha}}f_Kd\nu_K\\
&\le&
\int_{\mathbb{M}}\big((e^{-t{\mathcal{L}}^{\alpha}})^{\ast}\nu_K\big)
f_Kd\mu(x)\le
\|(e^{-t{\mathcal{L}}^{\alpha}})^{\ast}\nu_K\|_{L^{p'}(\mathbb{M})}\|
f_K\|_{L^p(\mathbb{M})}\\
&=&C_{p}^{(\alpha)}(K).
\end{eqnarray*}
This implies
$f_K=\big((e^{-t{\mathcal{L}}^{\alpha}})^{\ast}\nu_K\big)^{{1}/{(p-1)}}$,
which completes the proof of (ii).

\end{proof}

\subsection{Further nature of $L^p$-capacity}\label{s2}

  Some fundamental
properties of the $L^p$-capacity are stated in the following
proposition, which can be easily obtained and see the Euclidean case
in  \cite[Proposition 2]{chang}.

\begin{proposition}\label{p23} The following properties are valid.
\begin{itemize}
\item[{\rm(i)}] $ {C}_{p}^{(\alpha)}(\emptyset)=0$.

\item[{\rm(ii)}] If $K_1\subseteq K_2\subset\mathbb M_+$, then $ C^{(\alpha)}_{p}(K_1)\le  {C}_{p}^{(\alpha)}(K_2)$.

\item[{\rm (iii)}] For any sequence $\{K_{j}\}_{j=1}^\infty$ of subsets of
$\mathbb{M}_{+}$
$$ {C}_{p}^{(\alpha)}\Big(\bigcup_{j=1}^\infty K_{j}\Big)\leq
\sum_{j=1}^\infty  {C}_{p}^{(\alpha)}(K_{j}).$$
\end{itemize}
\end{proposition}

The following theorem gives the spherical capacity.

\begin{theorem}\label{t12new}Assume that  the measure $\mu$
satisfies (\ref{eq1.2}) and (\ref{eq1.3}).
 If
$ 1\le p<\infty,$ then
$$
 r_0^{\beta^{\star}}\lesssim {C}_{p}^{(\alpha)}\big(B^{(\alpha)}_{r_0}(t_0,x_0)\big)\lesssim
(t^{{1}/{2\alpha}}_0+r_0)^{\beta}\quad\hbox{for}\  \
(r_0,x_0)\in\mathbb M_+.
$$ In particularly, if $t_0 \lesssim r^{2\alpha}_0$, then
$$r_0^{\beta^{\star}}\lesssim{C}_{p}^{(\alpha)}\big(B^{(\alpha)}_{r_0}(t_0,x_0)\big)\lesssim  r_0^{\beta}.$$
\end{theorem}

\begin{proof}
 If $f\geq 0$ and $e^{-t{\mathcal{L}}^{\alpha}}f\geq
1_{B^{(\alpha)}_{r_0}(t_{0},x_{0})}$, then, for $1\leq p<\infty$,
there exist $\tilde{p}$ and $\tilde{q}$ such that
$$
\begin{cases}
1\leq p <\tilde{p}< \frac{\beta^{\star} p}{\beta^{\star}-\min\{\beta^{\star},2\alpha\}};\\
1/\tilde{q}=\beta^{\star}(1/p-1/\tilde{p})/2\alpha.
\end{cases}
$$
Consequently, according to Theorem \ref{th-1.6} (i), we have
\begin{equation*}\label{ineq str}
r^{{2\alpha}/{\tilde{q}}+{\beta^{\star}}/{\tilde{p}}}_0\lesssim\|e^{-t{\mathcal{L}}^{\alpha}}f\|_{L^{\tilde{q}}((0,\infty);L^{\tilde{p}})}\lesssim
\|f\|_{ L^{p} (\mathbb{M})}.
\end{equation*}
This, together with the definition of $C^{(\alpha)}_{p}(\cdot)$,
implies that
$$
r_0^{\beta^{\star}}\lesssim
C^{(\alpha)}_{p}\big(B^{(\alpha)}_{r_0}(t_{0},x_{0})\big)
$$
thanks to ${\beta^{\star}}/{\tilde{p}}+{2\alpha}/{\tilde{q}}={\beta^{\star}}/{p}.$

To get the corresponding upper bound of
$C^{(\alpha)}_{p}\big(B^{(\alpha)}_{r_0}(t_{0},x_{0})\big)$, we
consider $ f=1_{B_{r_0}(x_{0})}$, where
$$B_{r_0}(x_{0})=\{x\in\mathbb{M}:\  d(x,x_0)<r_0\}. $$   Since
$x\in B^{(\alpha)}_{r_0}(x_{0})$,   one has $d(x,y)<2r_0$. For any
$$(t,x)\in B^{(\alpha)}_{r_0}(t_0,x_0)$$
we have
$$t\simeq t_0+r^{2\alpha}_0.$$
Using Proposition \ref{prop-4}, we have
\begin{eqnarray*}
e^{-t{\mathcal{L}}^{\alpha}} 1_{B^{(\alpha)}_{r_0}(x_{0})}(x)
&=& \int_{\mathbb{M}}K^{\mathcal{L}}_{\alpha,t}(x,y)1_{B_{r_0}(x_{0})}(y)\,d\mu(y)\\
&=&\int_{d(y,x_{0})<r_0}K^{\mathcal{L}}_{\alpha,t}(x,y)\,d\mu(y) \\
&\gtrsim&\int_{d(y,x_{0})<r_0}
\frac{t}{(t^{{1}/{2\alpha}}+2r_0)^{\beta+2\alpha}}\,d\mu(y)\\
&\gtrsim& \frac{r^{\beta}_0}{(t^{{1}/{2\alpha}}_0+r_0)^{\beta}}.
\end{eqnarray*}
Hence,
$$
e^{-t{\mathcal{L}}^{\alpha}}
\left(\frac{1_{B^{(\alpha)}_{r_0}(x_{0})}(x)}{{r^{\beta}_0}/{(t^{{1}/{2\alpha}}_0+r_0)^{\beta}}}\right)
\gtrsim 1
$$
holds for any $(t,x)\in B^{(\alpha)}_{r_0}(t_0,x_0)$. Therefore,  we
get
$$
C_p^{(\alpha)}\big(B^{(\alpha)}_{r_0}(t_0,x_0))\lesssim \left\|
\frac{1_{B^{(\alpha)}_{r_0}(x_{0})}(x)}{{r^{\beta}_0}/{(t^{{1}/{2\alpha}}_0+r_0)^{\beta}}}\right\|^p_{L^p(\mathbb{M})}\lesssim
(t^{{1}/{2\alpha}}_0+r_0)^{\beta}\quad\forall\quad r_0>0.
$$

\end{proof}

\section{$L^q(\mathbb M_+)$-extensions of $L^p(\mathbb M)$ via $\mathcal{L}^{(\alpha)}$}\label{sec-5}

Throughout this  section, we always assume that the heat kernel
$p_t(\cdot,\cdot)$ satisfies {\bf Assumption (A4)}.

 \subsection{Capacitary   strong type inequalities}\label{s31} Let $L^p_+(\mathbb M)$ be the class of all nonnegative functions in $L^p(\mathbb M)$.
 Then we   establish the following capacitary strong type
inequality.

\begin{lemma}\label{l3a} Let $p\in (1,\infty)$. Assume that the measure $\mu$
satisfies (\ref{eq1.3}), then
$$
\int_0^\infty C_p^{(\alpha)}\big(\{(t,x)\in \mathbb M_+:\
e^{-t{\mathcal{L}}^{\alpha}}
f(x)\ge\lambda\}\big)\,d\lambda^p\lesssim\|f\|_{L^p(\mathbb
M)}^p\quad\forall\ f\in L^p_+(\mathbb M),
 $$
 here and henceforth, $d\lambda^p=p\lambda^{p-1}d\lambda$.
\end{lemma}
\begin{proof} We prove this lemma by adopting the method in \cite[Theorem 7.1.1]{AH} or
\cite[Lemma 3.1]{chang}.  Since the desired strong type estimate
follows from the density of $C_0(\mathbb M)$ in $L^p(\mathbb M)$, we
are about to verify the result for any nonnegative $C_0(\mathbb
M)$-function.  For each $j=0,\pm 1, \pm 2,...$ and $0\le f\in
C_0(\mathbb M)$, we set
$$E_j=\Big\{(t,x)\in\mathbb M_+:\ e^{-t{\mathcal{L}}^{\alpha}}f(x)\ge
2^{j}\Big\}.$$
If $\nu_j$ stands for the measure obtained in Proposition \ref{p22}
(ii) for $E_j$, then
\begin{eqnarray*}
S&=&\sum_{j=-\infty}^\infty 2^{jp}\nu_j(E_{j})\le\sum_{j=-\infty}^\infty 2^{j(p-1)}\int_{\mathbb M_+}e^{-t{\mathcal{L}}^{\alpha}}f\,d\nu_j\\
&=&\sum_{j=-\infty}^\infty 2^{j(p-1)}\int_{\mathbb
M}f(x)(e^{-t{\mathcal{L}}^{\alpha}})^\ast
\nu_j(x)\,d\mu(x)\\
&\le&\|f\|_{L^p(\mathbb M)}
\Big\|\sum_{j=-\infty}^\infty 2^{j(p-1)}(e^{-t{\mathcal{L}}^{\alpha}})^\ast \nu_j\Big\|_{L^{p'}(\mathbb M)}\\
&=&\|f\|_{L^p(\mathbb M)}T^{1/{p'}},
\end{eqnarray*}
where $$T=\Big\|\sum\limits_{j=-\infty}^\infty 2^{j(p-1)}(e^{-t{\mathcal{L}}^{\alpha}})^\ast \nu_j\Big\|^{p'}_{L^{p'}(\mathbb
M)}.
$$

In what follows, we   prove $T\lesssim S$ by two cases.

{\it Case 1: $2\le p<\infty$}. Upon letting
$$
\begin{cases}
k=0,\pm 1,\pm 2,...;\\
\sigma_k(x)=\sum_{j=k}^{\infty} 2^{j(p-1)}(e^{-t{\mathcal{L}}^{\alpha}})^\ast\nu_j(x);\\
\sigma(x)=\sum_{j=-\infty}^\infty
2^{j(p-1)}(e^{-t{\mathcal{L}}^{\alpha}})^\ast\nu_j(x),
\end{cases}
$$
we have
$\sigma_k\in L^{p'}(\mathbb M)\ \&\
\lim_{k\to-\infty}\sigma_k=\sigma$. Note that
$$
\sigma(x)^{p'}=p'\sum_{k=-\infty}^\infty\sigma_k(x)^{p'-1}2^{k(p-1)}(e^{-t{\mathcal{L}}^{\alpha}})^\ast\nu_k(x)\quad\hbox{for\
\  a.e.}\quad x\in\mathbb M.
$$
Then using the H\"older inequality derives
$$
T\le p' T_1^{2-p'}T_2^{p'-1}\lesssim S,
$$
where
$$
T_1=\int_{\mathbb M}\sum_k
2^{kp}\big((e^{-t{\mathcal{L}}^{\alpha}})^\ast\nu_k(x)\big)^{p'}\,d\mu(x)=\sum_k
2^{kp}C_p^{(\alpha)}(E_k)=S
$$
and
\begin{eqnarray*}
T_2&=&\int_{\mathbb M}\sum_k \sigma_k(x)2^k\big((e^{-t{\mathcal{L}}^{\alpha}})^\ast\nu_k(x)\big)^{p'-1}\,d\mu(x)\\
&=&\sum_k\sum_{j\ge k}2^{j(p-1)+k}\int_{\mathbb M}\big((e^{-t{\mathcal{L}}^{\alpha}})^\ast\nu_j(x)\big)\big((e^{-t{\mathcal{L}}^{\alpha}})^\ast\nu_k(x)\big)^{p'-1}\,d\mu(x)\\
&\lesssim&\sum_k\sum_{j\ge k}2^{j(p-1)+k}C^{(\alpha)}_p(E_j)\\
&\approx&\sum_k 2^{kp}C_p^{(\alpha)}(E_k)\approx S.
\end{eqnarray*}

{\it Case 2: $2>p>1$}. Similarly, let
$$\begin{cases}
k=0,\pm 1,\pm 2,...;\\
\sigma_k(x)=\sum_{j=-\infty}^{k} 2^{j(p-1)}(e^{-t{\mathcal{L}}^{\alpha}})^\ast\nu_j(x);\\
\sigma(x)=\sum_{j=-\infty}^{\infty}
2^{j(p-1)}(e^{-t{\mathcal{L}}^{\alpha}})^\ast\nu_j(x).
\end{cases}$$
Then
$\sigma_k\in L^{p'}(\mathbb M)$ and $\lim_{k\to \infty}\sigma_k=\sigma$.
We obtain
$$
\sigma(x)^{p'}=p'\sum_{k=-\infty}^\infty\sigma_k(x)^{p'-1}2^{k(p-1)}(e^{-t{\mathcal{L}}^{\alpha}})^\ast\nu_k(x)\quad\hbox{for\
\ a.e.}\quad x\in\mathbb M.
$$
Hence,
\begin{eqnarray*}
T&=&p'\sum_{k=-\infty}^\infty 2^{k(p-1)}\int_{\mathbb M}\sigma_k(x)^{p'-1}(e^{-t{\mathcal{L}}^{\alpha}})^\ast\nu_k(x)\,d\mu(x)\\
&\lesssim&\sum_{k=-\infty}^\infty 2^{k(p-1)}\Big\{\sum_{j=-\infty}^k 2^{j(p-1)}\Big[\int_{\mathbb M}
\frac{\big((e^{-t{\mathcal{L}}^{\alpha}})^\ast\nu_j(x)\big)^{p'-1}}{\big((e^{-t{\mathcal{L}}^{\alpha}})^\ast\nu_k(x)\big)^{-1}}\,d\mu(x)\Big]^\frac{1}{p'-1}\Big\}^{p'-1}\\
&\lesssim&\sum_{k=-\infty}^\infty 2^{kp}C_p^{(\alpha)}(E_k)\approx
S.
\end{eqnarray*}

In a word,  we have
$$
S=\sum_{j=-\infty}^\infty 2^{jp}C_p^{(\alpha)}(E_j)\lesssim
\|f\|^p_{L^p(\mathbb M)}.
$$
The desired strong type inequality holds for $0\le f\in C_0(\mathbb M)$.
\end{proof}

\begin{remark}By the above capacitary strong type inequality, it is easy to get the following capacitary weak type inequality
 $$
 \lambda^p C_p^{(\alpha)}\big(\{(t,x)\in \mathbb M_+:\ e^{-t{\mathcal{L}}^{\alpha}} f(x)\ge\lambda\}\big)\le\|f\|_{L^p(\mathbb M)}^p\quad\forall\quad f\in L^p_+(\mathbb M).
 $$
\end{remark}

 \subsection{The lower sector case $1<p\le q<\infty$}\label{s32} In what follows, $\mathcal M_+(\mathbb M_+)$ represents the class of all nonnegative Randon measures on $\mathbb M_+$. For $\lambda>0$, define
 $$ \kappa(\nu;\lambda)=\inf\Big\{C_p^{(\alpha)}(K):\ \hbox{compact}\ K\subset\mathbb M_+\ \&\ \nu(K)\ge\lambda\Big\}.
 $$ Denote by $L^q(\mathbb
M_+,\nu)$, $1\le q\le \infty$,  the Lebesgue spaces on $ \mathbb
M_+$ with respect to the measure $\nu$.

\begin{theorem}\label{t3a}
Let $1<p\le q<\infty$ and $\nu\in\mathcal M_+(\mathbb
M_+)$. Assume that the measure $\mu$ satisfies (\ref{eq1.3}),  then
the extension $e^{-t{\mathcal{L}}^{\alpha}}: L^p(\mathbb M)\mapsto
L^{q}(\mathbb M_+,\nu)$ is bounded if and only if
$$
\sup_{\lambda\in\mathbb
R_+}{\lambda^{{p}/{q}}}/{\kappa(\nu;\lambda)}<\infty.
$$
Furthermore, assume that the measure $\mu$ satisfies (\ref{eq1.2})
and (\ref{eq1.3}) with $\beta=\beta^{\star}$. If $1<p<q<\infty$, then
${\lambda^{{p}/{q}}}\lesssim {\kappa(\nu;\lambda)} \ \forall\ \
\lambda\in\mathbb R_+$ can be replaced by
$\nu(B^{(\alpha)}_r(t_0,x_0))\lesssim r^{q\beta/p}\ \ \forall\
(r,t_0,x_0)\in\mathbb R_+\times\mathbb R_+\times\mathbb M$ with
$t_0\lesssim r^{2\alpha}$.
\end{theorem}
 \begin{proof} Suppose that $e^{-t{\mathcal{L}}^{\alpha}}: L^p(\mathbb M)\mapsto L^{q}(\mathbb M_+,\nu)$ is bounded. Then, for a given compact subset $K\subset\mathbb M_+$, we use Proposition \ref{p22} (ii) and H\"older's inequality with
 $$(p',q')=\Big({p}/{(p-1)}, {q}/{(q-1)}\Big) $$
 to derive
 \begin{align*}
 &\int_{\mathbb M}f(e^{-t{\mathcal{L}}^{\alpha}})^\ast\nu_Kd\mu\\
 &\ =\int_{\mathbb M_+}e^{-t{\mathcal{L}}^{\alpha}}f\,d\nu_K\\
 &\ \le
 \|e^{-t{\mathcal{L}}^{\alpha}}f\|_{L^q(\mathbb M_+,\nu)}\nu(K)^{{1}/{q'}}\\
 &\ \lesssim\|f\|_{L^p(\mathbb M)}\nu(K)^{{1}/{q'}},
 \end{align*}
 and consequently,
 $$
 \|(e^{-t{\mathcal{L}}^{\alpha}})^\ast\nu_K\|_{L^{p'}(\mathbb M)}\lesssim \nu(K)^{1/q'}.
 $$
 Via the above estimate, we have
\begin{align*}
\lambda\nu(E_\lambda(f))&\le\int_{\mathbb M_+}|e^{-t{\mathcal{L}}^{\alpha}}f|\,d\nu_{E_\lambda}\\
 &\ \lesssim\|f\|_{L^p(\mathbb M)}\|(e^{-t{\mathcal{L}}^{\alpha}})^\ast\nu_{E_\lambda}\|_{L^{p'}(\mathbb M)}\\
 &\ \lesssim \|f\|_{L^p(\mathbb
 M)}\nu(E_\lambda(f))^{1/{q'}},
 \end{align*}
 where
$$E_\lambda(f)=\Big\{(t,x)\in\mathbb M_+:\ |e^{-t{\mathcal{L}}^{\alpha}}
f(t,x)|\ge\lambda\Big\}.$$
Hence,
 $$ \sup_{\lambda\in\mathbb R_+}\lambda^q\nu(E_\lambda(f))\lesssim \|f\|_{L^p(\mathbb M)}^q. $$
This, upon choosing a function $f\in L^p(\mathbb M)$ such that
$e^{-t{\mathcal{L}}^{\alpha}}f\ge 1$ on a given compact set
$K\subset\mathbb M_+$, derives
$$\nu(K)^{1/q}\lesssim C_p^{(\alpha)}(K)^{1/p},$$
equivalently,
$$\sup_{\lambda\in\mathbb
R_+}{\lambda^\frac{p}{q}}/{\kappa(\nu;\lambda)}<\infty.$$

Conversely, assume that the last condition is true, i.e., the last
but one is valid for any compact set $K\subset\mathbb M_+$. Thus,
combining  Lemma \ref{l3a} with the capacitary strong type
inequality leads to
\begin{align*}
\int_{\mathbb M_+}|e^{-t{\mathcal{L}}^{\alpha}}f|^q\,d\nu &=\int_0^\infty\nu(E_\lambda(f))\,d\lambda^q\\
&\lesssim \int_0^\infty C_p^{(\alpha)}(E_\lambda(f))^{q/p}\,d\lambda^q\\
&\lesssim \|f\|_{L^p(\mathbb M)}^{q-p}\int_0^\infty C_p^{(\alpha)}(E_\lambda(f))\,d\lambda^p\\
&\lesssim \|f\|_{L^p(\mathbb M)}^{q}
\end{align*}
  for any $f\in C_0(\mathbb M)$, and then for $f\in L^p(\mathbb
M)$ via approximating with $ C_0(\mathbb M)$-functions.

Next, we verify that under $1<p<q<\infty$ the criterion
${\lambda^{p/q}}\lesssim {\kappa(\nu;\lambda)}$ can be replaced by
an easily-checked condition $\nu(B^{(\alpha)}_r(t_0,x_0))\lesssim
r^{q \beta/p}$.

The implication
$${\lambda^{{p}/{q}}}\lesssim
{\kappa(\nu;\lambda)}\Longrightarrow\nu(B^{(\alpha)}_r(t_0,x_0))\lesssim
r^{q\beta/p}$$ follows immediately from Theorem \ref{t12new}.
Conversely, for $(t,x)\in B^{(\alpha)}_r(t_0,x_0)$ with $t_0\lesssim
r^{2\alpha}$ and $(t_0,x_0)\in\mathbb M_+$, using Proposition
\ref{prop-4}, we have
  $K^{\mathcal{L}}_{\alpha,t}(x_0,x)\gtrsim r^{-\beta}$. This, along
  with
Fubini's theorem, implies
\begin{align*}
(e^{-t{\mathcal{L}}^{\alpha}})^\ast\nu_K(x_0)&=\int_{\mathbb M_+}K^{\mathcal{L}}_{\alpha,t}(x_0,x)\,d\nu_K\\
&\simeq\int_{\mathbb M_+}\Big(\int_{\big(K^{\mathcal{L}}_{\alpha,t}(x_0,x)\big)^{-1/\beta}}^\infty\frac{dr}{r^{1+\beta}}\Big)\,d\nu_K\\
&\lesssim\int_{\mathbb M_+}\Big(\int_0^\infty 1_{B^{(\alpha)}_r(t_0,x_0)}\,\frac{dr}{r^{1+\beta}}\Big)\,d\nu_K\\
&\simeq\int_0^\infty\nu_K\big(B^{(\alpha)}_r(t_0,x_0)\big)\,\frac{dr}{r^{1+\beta}}.
\end{align*}
It follows from Minkowski's inequality that
$$
\big\|(e^{-t{\mathcal{L}}^{\alpha}})^\ast\nu_K\big\|_{L^{p'}(\mathbb
M)}\lesssim\int_0^\infty\big\|\nu_K\big(B^{(\alpha)}_r(t_0,\cdot)\big)\big\|_{L^{p'}(\mathbb
M)} \,\frac{dr}{r^{1+\beta}}.
$$
In general,
$$
\|\nu_K\big(B^{(\alpha)}_r(t_0,\cdot)\big)\big\|^{p'}_{L^{p'}(\mathbb
M)}\lesssim\nu(K)^{p'-1}\int_{\mathbb
M}\nu_K(B^{(\alpha)}_r(t_0,x_0))\,d\mu(x_0)
\lesssim\nu(K)^{p'}r^{\beta}.
$$
This implies that for a later-decided number $\delta>0$,
$$
\int_\delta^\infty
\big\|\nu_K\big(B^{(\alpha)}_r(t_0,\cdot)\big)\big\|_{L^{p'}(\mathbb
M)}\frac{dr}{r^{1+\beta}}\lesssim\nu(K)\delta^{-{\beta}/{p}}.
$$
Meanwhile, using $\nu(B^{(\alpha)}_r(t_0,x_0))\lesssim r^{q\beta/p}$,
we have
\begin{eqnarray*}
\|\nu_K\big(B^{(\alpha)}_r(t_0,\cdot)\big)\big\|^{p'}_{L^{p'}(\mathbb
M)}&\lesssim& r^{\beta q(p'-1)/p}\int_{\mathbb
M}\nu_K(B^{(\alpha)}_r(t_0,x_0))\,d\mu(x_0)\\
&\lesssim&\nu(K)r^{\beta(1+{q}/{p(p-1)})}.
\end{eqnarray*}
Furthermore, we obtain
$$
\int_0^\delta\big\|\nu_K\big(B^{(\alpha)}_r(t_0,\cdot)\big)\big\|_{L^{p'}(\mathbb
M)}\frac{dr}{r^{1+\beta}}\lesssim\nu(K)^{1/{p'}}\delta^{{\beta(q-p)}/{p^2}}.
$$
Now, choosing $\delta=\nu(K)^{{p}/{\beta q}}$ and putting the above
estimates together, we find
$$\|(e^{-t{\mathcal{L}}^{\alpha}})^\ast\nu_K\|_{L^{p'}(\mathbb{M})}\lesssim\nu(K)^{1/{q'}},$$
 whence reaching
${\lambda^{{p}/{q}}}\lesssim {\kappa(\nu;\lambda)}$.
\end{proof}


 \subsection{The upper sector case $1<q< p<\infty$}\label{s34}
\subsubsection{Metric measure space case }

In what follows, we give the first main result of this section.
%
%
%
%
%

 \begin{theorem} \label{t4a}
  Let $\mu$ satisfy (\ref{eq1.2}), $1<q<p<\infty$ and $\nu\in\mathcal M_+(\mathbb{M}_+)$. Then the following two statements are equivalent:

 $\mathrm{(i)}$ The extension $e^{-t{\mathcal{L}}^{\alpha}}: L^p(\mathbb{M})\mapsto L^{q}( \mathbb{M}_+,\nu)$ is bounded.

 $\mathrm{(ii)}$
 $$
 \int_0^\infty\Big({\lambda^\frac{p}{q}}/{\kappa(\nu;\lambda)}\Big)^{q/(p-q)}\,{\lambda}^{-1}d{\lambda}<\infty.
 $$

 \end{theorem}
 \begin{proof}  At first, we show that $\mathrm{(i)}\Rightarrow \mathrm{(ii)}$. Suppose $e^{-t{\mathcal{L}}^{\alpha}}: L^p(\mathbb{M})\mapsto L^{q}( \mathbb{M}_+,\nu)$ is bounded. Then
 $$
 \Big(\int_{ \mathbb{M}_+} |e^{-t{\mathcal{L}}^{\alpha}}f|^{q}d\nu\Big)^{{1}/{q}}\lesssim\|f\|_{L^p(\mathbb{M})}\ \ \forall\ \  f\in L^p(\mathbb{M}).
 $$
Therefore,
 $$
 \sup_{\lambda>0}\lambda\Big(\nu\big(E_\lambda(f)\big)\Big)^{{1}/{q}}\lesssim\|f\|_{L^p(\mathbb{M})}\ \ \forall\ \  f\in L^p(\mathbb{M}).
 $$
For each integer $j$, there are a compact set $K_j\subset
\mathbb{M}_+$ and a function $f_j\in L^p(\mathbb{M})$ such that
$$
\begin{cases}
 C_p^{(\alpha)}(K_j)\le 2\kappa(\nu;2^j);\\
\nu(K_j)>2^j;\\
e^{-t{\mathcal{L}}^{\alpha}}f_j\ge 1_{K_j};\\
\|f_j\|_{L^p(\mathbb{M})}^p\le 2C_p^{(\alpha)}(K_j).
\end{cases}
$$
For the integers $i,k$ with $i<k$, let
$$
f_{i,k}=\sup_{i\le j\le
k}\Big(\frac{2^j}{\kappa(\nu;2^j)}\Big)^{1/(p-q)}f_j.
$$
Then
\begin{align*}
 \|f_{i,k}\|_{L^p(\mathbb{M})}^p&\lesssim\sum_{j=i}^k\Big(\frac{2^j}{\kappa(\nu;2^j)}\Big)^{{p}/{(p-q)}}\| f_j\|_{L^p(\mathbb{M})}^p\\
 &\lesssim\sum_{j=i}^k\Big(\frac{2^j}{\kappa(\nu;2^j)}\Big)^{{p}/{(p-q)}}\kappa(\nu;2^j).
 \end{align*}
 Note that for $i\le j\le k$,
 $$
 (t,x)\in K_j\Longrightarrow |e^{-t{\mathcal{L}}^{\alpha}}f_{i,k}(t,x)|\ge \Big(\frac{2^j}{\kappa(\nu;2^j)}\Big)^{{1}/{(p-q)}}.
 $$
 This in turn leads to
 $$
 2^j<\nu(K_j)\le
 \nu\Big(E_{\big({2^j}/{\kappa(\nu;2^j)}\big)^{{1}/{(p-q)}}}(f_{i,k})\Big).
 $$
 Hence,
 \begin{align*}
 \|f_{i,k}\|_{L^p(\mathbb{M})}^q&\gtrsim\int_{ \mathbb{M}_+}|e^{-t{\mathcal{L}}^{\alpha}}f_{i,k}|^q\,d\nu\\
 &\simeq\int_0^\infty \Big(\inf\{\lambda:\ \nu\big(E_\lambda(f_{i,k})\big)\le s\}\Big)^q ds\\
 &\gtrsim\sum_{j=i}^k\Big(\inf\{\lambda:\ \nu\big(E_{\lambda}(f_{i,k})\big)\le 2^j\}\Big)^q2^j\\
 &\gtrsim \sum_{j=i}^k\Big(\frac{2^j}{\kappa(\nu;2^j)}\Big)^{{q}/{(p-q)}}2^j\\
 &\gtrsim\left(\frac{\sum_{j=i}^k\Big({2^j}/{\kappa(\nu;2^j)}\Big)^{{q}/{(p-q)}}2^j}{\Big(\sum_{j=i}^k\Big({2^j}/{\kappa(\nu;2^j)}\Big)
 ^{{p}/{(p-q)}}\kappa(\nu;2^j)\Big)^{{q}/{p}}}\right)\|f_{i,k}\|_{L^p(\mathbb{M})}^q\\
 &\approx \Big(\sum_{j=i}^k\frac{2^{{jp}/{(p-q)}}}{\big(\kappa(\nu;2^j)\big)^{{q}/{(p-q)}}}\Big)^{{(p-q)}/{p}}
 \|f_{i,k}\|_{L^p(\mathbb{M})}^q.
 \end{align*}
 This implies
 $$
\int_0^\infty\Big(\lambda^{p/q}/\kappa(\nu;\lambda)\Big)^{{q}/{(p-q)}}\lambda^{-1}\,d\lambda\lesssim
\sum_{j=-\infty}^\infty\frac{2^{{jp}/{(p-q)}}}{\big(\kappa(\nu;2^j)\big)^{{q}/{(p-q)}}}\lesssim
1.
 $$

 Secondly, we prove that $\mathrm{(ii)}\Rightarrow \mathrm{(i)}$. Suppose
 $$
 I_{p,q}(\nu)=\int_0^\infty\Big(\frac{\lambda^{p/q}}{\kappa(\nu;\lambda)}\Big)^{{q}/{(p-q)}}\frac{d\lambda}{\lambda}<\infty.
 $$
 Now for each integer $j=0,\pm 1,\pm 2,...,$ and $f\in C_0(\mathbb{M})$, let
 $$
 S_{p,q}(\nu;f)=\sum_{j=-\infty}^\infty
 \frac{\big(\nu\big(E_{2^j}(f)\big)-\nu\big(E_{2^{j+1}}(f)\big)\big)^{{p}/{(p-q)}}}{\Big(C^{(\alpha)}_p\big(E_{2^j}(f)\big)\Big)^{{q}/{(p-q)}}}.
 $$
 Using integration-by-part, H\"older's inequality and Lemma \ref{l3a}, we obtain

 \begin{align*}
 &\int_{ \mathbb{M}_+}|e^{-t{\mathcal{L}}^{\alpha}}f|^q\, d\nu\\
 &\ \ =-\int_0^\infty\lambda^q\,d\nu(E_\lambda(f))\\
 &\ \ \lesssim\sum_{j=-\infty}^\infty\big(\nu\big(E_{2^j}(f)\big)-\nu\big(E_{2^{j+1}}(f)\big)2^{jq}\\
 &\ \ \lesssim(S_{p,q}(\nu;f))^{{(p-q)}/{p}}\Big(\sum_{j=-\infty}^\infty 2^{jp}C_p^{(\alpha)}\big(E_{2^j}(f)\big)\Big)^{q/p}\\
 &\ \ \lesssim(S_{p,q}(\nu;f))^{{(p-q)}/{p}}\Big(\int_0^\infty C_p^{(\alpha)}\big(\{(t,x)\in \mathbb{M}_+:\  |e^{-t{\mathcal{L}}^{\alpha}}f(t,x)|>\lambda\}\big)\,d\lambda^p\Big)^{{q}/{p}}\\
 &\ \ \lesssim(S_{p,q}(\nu;f))^{{(p-q)}/{p}}\|f\|_{L^p(\mathbb{M})}^q.
 \end{align*}
 Note also that

 \begin{align*}
& \left(S_{p,q}(\nu;f)\right)^{{(p-q)}/{p}}\\
 &\ \ =\Big(\sum_{j=-\infty}^\infty\frac{\big(\nu(E_{2^j}(f))-\nu(E_{2^{j+1}}(f))\big)^{{p}/{(p-q)}}}
 {\Big(C_p^{(\alpha)}\big(E_{2^j}(f)\big)\Big)^{{q}/{(p-q)}}}\Big)^{{(p-q)}/{p}}\\
  &\ \ =\Big(\sum_{j=-\infty}^\infty\frac{\big(\nu(E_{2^j}(f))-\nu(E_{2^{j+1}}(f))\big)^{{p}/{(p-q)}}}
  {\Big(\kappa(\nu;\nu\big(E_{2^j}(f)\big))\Big)^{{q}/{(p-q)}}}\Big)^{{(p-q)}/{p}}\\
  &\ \ =\Big(\sum_{j=-\infty}^\infty\frac{\big(\nu(E_{2^j}(f)\big)^{{p}/{(p-q)}}-\big(\nu(E_{2^{j+1}}(f)\big)^{{p}/{(p-q)}}}
  {\Big(\kappa(\nu;\nu\big(E_{2^j}(f)\big))\Big)^{{q}/{(p-q)}}}\Big)^{{(p-q)}/{p}}\\
    &\ \ \lesssim\Big(\int_0^\infty\frac{d s^{{p}/{(p-q)}}}{\big(\kappa(\nu; s)\big)^{{q}/{(p-q)}}}\Big)^{{(p-q)}/{p}}\\
 &\ \ \approx \big(I_{p,q}(\nu)\big)^{{(p-q)}/{p}}.
 \end{align*}
 Therefore,
 $$
 \Big(\int_{ \mathbb{M}_+}|e^{-t{\mathcal{L}}^{\alpha}}f|^q\,d\nu\Big)^{1/q}\lesssim \big(I_{p,q}(\nu)\big)^{{(p-q)}/{pq}}\|f\|_{L^p(\mathbb{M})}.
 $$
 \end{proof}

\subsubsection{Lie group case}
Theorems \ref{t3a}\ \&\ \ref{t4a} establish the relation between the extension operator $e^{-t\mathcal{L}^{\alpha}}$ and the capacities of the compact
sets in $\mathbb{M}_{+}$. In this section, if the metric measure spaces have some translation invariance: a family of dyadic cubes are still dyadic cubes
 under the translation, we can investigate   the boundedness of the extension
$e^{-t{\mathcal{L}}^{\alpha}}: L^p(\mathbb{M})\mapsto L^{q}(
\mathbb{M}_+,\nu)$ via the Hedberg-Wolff potential for
$e^{-t\mathcal{L}^{\alpha}}$ on these metric spaces. Without loss of
generality, in the sequel, we focus on a special class of metric
measure spaces which are called the stratified Lie groups or Carnot
groups.    At this time,  $\mathbb{M}=\mathbb{G}$, where
$\mathbb{G}$ is a stratified Lie group and   the parameters $\beta$
and $\beta^{\star}$ in (\ref{eq1.2}) and (\ref{eq1.3})  are exactly
$Q$, where $Q$ is the homogeneous dimension of $\mathbb{G}$.

In what follows,  we first recall some basic facts of stratified Lie
groups (cf. \cite{Folland}). A Lie group $\mathbb{G}$ is called
stratified if it is nilpotent, connected and simply  connected, and
its Lie algebra $\mathfrak{g}$ admits a vector space decomposition
$\mathfrak{g} = V_{1}\oplus \cdots \oplus V_{m}$ such that $[V_{1},
V_{k}] = V_{k+1}$ for $1\leq k < m$ and $[V_{1}, V_{m}] =0$. If
$\mathbb{G}$ is stratified, its Lie algebra admits a family of
dilations, namely,
\begin{equation*}
\delta_{r}(X_{1} + X_{2} + \cdots + X_{m}) = rX_{1} + r^{2}X_{2} +
\cdots + r^{m}X_{m}\quad  (X_{j}\in V_{j}).
\end{equation*}
Assume that $\mathbb{G}$ is a Lie group with underlying manifold
$\mathbb{R}^{n}$ for some positive integer $n$. $\mathbb{G}$
inherits dilations from $\mathfrak{g}\colon $ if $x \in \mathbb{G}$
and $r > 0$, we write
\begin{equation*}\label{equ2}
rx = (r^{d_{1}}x_{1}, \cdots , r^{d_{n}}x_{n}),
\end{equation*}
where $1\leq d_{1}\leq \cdots \leq d_{n}$. The map $x \rightarrow
rx$ is an automorphism of $\mathbb{G}$. The left (or right) Haar
measure on $\mathbb{G}$ is simply $dx_{1} \cdots dx_{n}$, which is
the Haar measure on $\mathfrak{g}$. We still denote by $\mu$ the
Haar measure on $\mathbb{G}$.  The inverse of any $x\in \mathbb{G}$
is simply $-x$. The group law must have the form
\begin{equation*}\label{equ3}
xy = (p_{1}(x, y), \cdots , p_{n}(x, y))
\end{equation*}
for some polynomials $p_{1}, \cdots , p_{n}$ about $x_{1}, \cdots
,x_{n}, y_{1}, \cdots , y_{n}$.

The number $Q = \sum _{j=1}^{m} j (\text{ dim }V_{j})$ is known as the
homogeneous dimension of $\mathbb{G}$. We define a homogeneous norm function
$\mid \cdot \mid$ on $\mathbb{G}$ which is smooth away from $0$. Therefore,
$\left| rx\right|  = r\left| x\right| $ for all $x \in \mathbb{G}$, $r > 0$,
$\left| x^{-1}\right|  = \left| x\right| $ for all $x \in \mathbb{G}$, and
$\left| x\right|  > 0$ if $x \neq 0$. The homogeneous norm induces a
quasi-distance $d$ which is defined by $d(x, y) \colon= \left|
x^{-1}y\right| $.

   Note that $\mu( B(x, r)) \simeq r^{Q}$.
Therefore,  $\mathbb{G}$ is an  Ahlfors-David regular space, that
is,  the index $\beta^{\star}$ in (\ref{eqa2.1}) is exactly $Q$.

Before proving the next main result, we recall the dyadic  type
partitions on spaces of homogeneous type (cf. \cite{Christ}). They
are analogues of   Euclidean dyadic cubes on the space of
homogeneous type.

Let $\mathcal{X}$ be a  space of homogeneous type equipped with a
quasi-metric $d_c$ and   a doubling measure $\mu$ such that the
associated balls are open.    Let $\delta$ be a small positive
number. For each $k\in \mathbb{Z}$, fix a maximal collection of
points $z^k_{\gamma}\in  \mathcal{X}$ satisfying
\begin{equation}d_c(z^k_{\gamma},z^k_{\beta'})\ge \delta^{k}
\forall\ \gamma\neq \beta'.
\end{equation}
Of course, by maximality there is the reverse inequality. For each
$k$ and each $x\in \mathcal{X}$, there exists $\alpha$ such that
\begin{equation}d_c(x,z^k_{\gamma})<\delta^k.\end{equation}
\begin{definition}A tree is a partial ordering $\le$ of the set of
all ordered pairs $(k,\gamma)$, which satisfies:
\begin{itemize}
\item[\rm{(a)}] $(k,\gamma)\le (l,\beta')\Rightarrow k\ge l.$

\item[\rm{(b)}] For each $(k,\gamma)$ and $l\le k$ there exists a
unique $\beta'$ such that $(k,\gamma)\le (l,\beta').$

\item[\rm{(c)}] $(k,\gamma)\le (k-1,\beta')\Rightarrow
d_c(z^k_{\gamma},z^{k-1}_{\beta'})<\delta^{k-1}.$

\item[\rm{(d)}]
$d_c(z^k_{\gamma},z^{k-1}_{\beta'})<\frac{\delta^{k-1}}{2}\Rightarrow(k,\gamma)\le
(k-1,\beta').$
\end{itemize}
\end{definition}

 Lemma 13 in  \cite{Christ} implies the existence of the above tree.
 Fix a tree, and let $a_0\in (0,1)$ be a small constant. Denote
  $$Q^k_{\gamma}=\bigcup_{(l,\beta')\le (k,\gamma)}B(z^l_{\beta'},a_0\delta^l),$$
where $B(z^l_{\beta'},a_0\delta^l)=\{x\in \mathcal{X}:
d_c(x,z^l_{\beta'})<a_0\delta^l\}.$

We conclude from the following theorem due to Christ \cite{Christ}
that $Q^k_{\gamma}$ is exactly an dyadic cube for every $k$ and
$\gamma$.
\begin{proposition}\label{prop5.8} There exists a collection of open sets $\{Q^k_{\gamma}\subseteq \mathcal{X}, k\in
 \mathbb{Z},\gamma\in I_k \}$, and constants $\delta\in (0,1)$,
 $a_0>0$, $\eta>0$, and $C_1,C_2<\infty$ such that

$\mathrm{(a)}$ $\mu\big(\mathcal{X}\backslash \bigcup_{\gamma\in
I_k}Q^k_{\gamma}\big)=0 \  \forall\ k.$

$\mathrm{(b)}$ if $l\ge k$, then either $Q^k_{\beta'}\subseteq
Q^k_{\gamma}$ or $Q^k_{\beta'}\bigcap Q^k_{\gamma}\neq \varnothing.$

$\mathrm{(c)}$ For each $(k,\gamma)$ and each $l<k$ there is a
unique $\beta'$ such that $Q^k_{\gamma}\subseteq Q^l_{\beta'}$.

$\mathrm{(d)}$ Diameter $(Q^k_{\gamma})\le C_1 \delta^k.$

$\mathrm{(e)}$ Each $(Q^k_{\gamma})$ contains some ball
$B(z^k_{\gamma},a_0\delta^k)$.

$\mathrm{(f)}$ $\mu\big(x\in Q^k_{\gamma}:
d_c(x,\mathcal{X}\backslash Q^k_{\gamma})\le t \delta^k \big)\le C_2
t^{\eta}\mu(Q^k_{\gamma})$  $\forall k,\gamma,\ \forall t>0.$
\end{proposition}

Below we give two technical lemmas which will be used in the sequel.
The first is about $L^p$-boundedness of the fractional maximal
operator of parabolic type on $\mathbb{G}$.
 \begin{lemma}\label{l21}  For a nonnegative Radon measure $\nu$ on $\mathbb{G}_+$, where $\mathbb{G}_+=\mathbb{G}\times (0,\infty)$, let
 $$
  M_{\alpha}\nu(x)=\sup_{r>0}{r^{-Q}}\nu\big({B^{(\alpha)}_{r}(r^{2\alpha}, x)}\big)
  $$
  be the fractional parabolic maximal function of $\nu$. Then
 $$ \|M_{\alpha}\nu\|_{L^{p}(\mathbb{G} )}\approx \|(e^{-t\mathcal{L}^{\alpha}})^{*}\nu\|_{L^p(\mathbb{G} )}\ \ \forall\  \ p\in (1,\infty). $$
 \end{lemma}
 \begin{proof} A straightforward estimation with $x\in \mathbb{G} $ and $(e^{-t\mathcal{L}^{\alpha}})^{*}\nu$ gives
 \begin{align*}
 (e^{-t\mathcal{L}^{\alpha}})^{*}\nu(x)&\gtrsim \int_{B^{(\alpha)}_{r}(r^{2\alpha},\ x)}\frac{t}{(t^{{1}/{2\alpha}}+d(x,y))^{Q+2\alpha}}d\nu(t,y)\\
 &\gtrsim
 \frac{\nu(B^{(\alpha)}_{r}(r^{2\alpha}, x))}{r^{Q}}\ \ \forall\ \ r>0,
 \end{align*}
 whence
 $$
 (e^{-t\mathcal{L}^{\alpha}})^{*}\nu(x)\gtrsim M_{\alpha}\nu(x).
 $$
 This implies
 $$
 \|M_{\alpha}\nu\|_{L^p(\mathbb{G} )}\lesssim \|R^\ast_{\alpha}\nu\|_{L^p(\mathbb{G} )}.
 $$

 To prove the converse inequality, we slightly modify \cite[(3.6.1)]{AH} to get two constants $a>1$ and $b>0$ such that for any $\lambda>0$ and $0<\varepsilon\leq 1$, one has the following good-$\lambda$ inequality
 \begin{align}\label{2.4}
 &\mu(\{x\in \mathbb{G}:\ (e^{-t\mathcal{L}^{\alpha}})^{*}\nu(x)>a\lambda\})\notag\\
 &\ \leq b\varepsilon^{(Q+2\alpha)/Q}\mu(\{x\in \mathbb{G}:\ (e^{-t\mathcal{L}^{\alpha}})^{*}\nu(x)>\lambda\})\\
 &\ \ +\mu(\{x\in \mathbb{G}:\ M_{\alpha}\nu(x)>\varepsilon\lambda\}).\nonumber
 \end{align}
 Inspired by \cite[Theorem 3.6.1]{AH}, we proceed the proof by using $(\ref{2.4})$.
 Multiplying $(\ref{2.4})$ by $\lambda^{p-1}$ and integrating in $\lambda$, we have for any $\gamma>0$,
 \begin{eqnarray*}
 &&\int_{0}^{\gamma}\mu(\{x\in \mathbb{G}:\ (e^{-t\mathcal{L}^{\alpha}})^{*}\nu(x)>a\lambda\})\lambda^{p-1}d\lambda\\
 &&\ \ \leq b\varepsilon^{(Q+2\alpha)/Q}\int_{0}^{\gamma}\mu(\{x\in \mathbb{G}:\ (e^{-t\mathcal{L}^{\alpha}})^{*}\nu(x)>\lambda\})\lambda^{p-1}d\lambda\\
 &&\quad+\int_{0}^{\gamma}\mu(\{x\in \mathbb{G}:\  M_{\alpha}\nu(x)>\varepsilon\lambda\})\lambda^{p-1}d\lambda.
 \end{eqnarray*}
 An equivalent formulation of the above inequality is
 \begin{eqnarray*}
 &&a^{-p}\int_{0}^{a\gamma}\mu(\{x\in \mathbb{G}:\ (e^{-t\mathcal{L}^{\alpha}})^{*}\nu(x)>\lambda\})\lambda^{p-1}d\lambda\\
 &&\ \ \leq b\varepsilon^{(Q+2\alpha)/Q}\int_{0}^{\gamma}\mu(\{x\in \mathbb{G}:\ (e^{-t\mathcal{L}^{\alpha}})^{*}\nu(x)>\lambda\})\lambda^{p-1}d\lambda\\
 &&\quad+\varepsilon^{-p}\int_{0}^{\varepsilon \gamma}\mu(\{x\in \mathbb{G}:\ M_{\alpha}\nu(x)>\lambda\})\lambda^{p-1}d\lambda.
 \end{eqnarray*}
 Let $\varepsilon$ be so small that $b\varepsilon^{(Q+2\alpha)/Q}\leq a^{-p}/2$ and $\gamma\to\infty$. Then
 $$
 a^{-p}\int_{\mathbb{G}}((e^{-t\mathcal{L}^{\alpha}})^{*}\nu(x))^{p}d\mu\leq 2\varepsilon^{-p}\int_{\mathbb{G}}(M_{\alpha}\nu(x))^{p}d\mu,
 $$
 that is,
 $$
 \|M_{\alpha}\nu\|_{L^p(\mathbb{G} )}\gtrsim \|(e^{-t\mathcal{L}^{\alpha}})^{*}\nu\|_{L^p(\mathbb{G} )}.
 $$
 \end{proof}

The second is about the Hedberg-Wolff potential for
$e^{-t\mathcal{L}^{\alpha}}$ on the stratified Lie group.

\begin{lemma}\label{l22}  Let $1<p<\infty$, $p'={p}/{(p-1)}$, and $\nu$ be a nonnegative Radon measure on $\mathbb{G}_+$. Then
$$
\|(e^{-t\mathcal{L}^{\alpha}})^{*}\nu\|_{L^{p'}(\mathbb{G})}^{p'}\approx
\int_{\mathbb{G}_{+}}P_{\alpha p}\nu\, d\nu,
$$
 where
\begin{equation}\label{eq6.2}
 P_{\alpha p} \nu(t,x)=\int_{0}^{\infty}\Big(\frac{\nu(B^{(\alpha)}_{r}(t, x))}{r^{Q}}\Big)^{p'-1}\frac{dr}{r} \,\,\,\,\forall\,\,(t,x)\in   \mathbb{G}_{+}.
\end{equation}
\end{lemma}
\begin{proof} We prove this lemma from two aspects.

{\it Part 1}. The first task is to show
$$
\|(e^{-t\mathcal{L}^{\alpha}})^{*}\nu\|_{L^{p'}(\mathbb{G})}^{p'}\lesssim
\int_{\mathbb{G}_{+}}P_{\alpha p}\nu\,d\nu.
$$
Note first that
\begin{eqnarray*}
\frac{\nu(B^{(\alpha)}_{r}(r^{2\alpha}, x))}{r^{Q}}&\approx&
\Big(\int_{r}^{2r}\Big(\frac{\nu(B^{(\alpha)}_{s}(s^{2\alpha},
x))}{s^{Q}}\Big)^{p'}\frac{ds}{s}\Big)^{{1}/{p'}}\\
&\lesssim&
\Big(\int_{0}^{\infty}\Big(\frac{\nu(B^{(\alpha)}_{s}(s^{2\alpha},
x))}{s^{Q}}\Big)^{p'}\frac{ds}{s}\Big)^{{1}/{p'}}.
\end{eqnarray*}
Therefore, one has
$$
M_{\alpha}\nu(x)\lesssim
\Big(\int_{0}^{\infty}\Big(\frac{\nu(B^{(\alpha)}_{s}(s^{2\alpha},
x))}{s^{Q}}\Big)^{p'}\frac{ds}{s}\Big)^{{1}/{p'}}.
$$
By Lemma \ref{l21}, it is sufficient to verify
$$
\int_{\mathbb{G}}\int_{0}^{\infty}\Big(\frac{\nu(B^{(\alpha)}_{r}(r^{2\alpha},
x))}{r^{Q}}\Big)^{p'}\frac{dr}{r}d\mu\lesssim
\int_{\mathbb{G}_{+}}\int_{0}^{\infty}\Big(\frac{\nu(B^{(\alpha)}_{r}(t,
x))}{r^{Q}}\Big)^{p'-1}\frac{dr}{r}d\nu.
$$
Using the Fubini theorem, one has
$$
\int_{
\mathbb{G}}\int_{0}^{\infty}\Big(\frac{\nu(B^{(\alpha)}_{r}(r^{2\alpha},
x))}{r^{Q}}\Big)^{p'}\frac{dr}{r}d\mu=\int_{0}^{\infty}\int_{
\mathbb{G}}\frac{\nu(B^{(\alpha)}_{r}(r^{2\alpha},
x))^{p'}}{r^{Q p'}}d\mu\frac{dr}{r}.
$$
A further application of Fubini's theorem yields
\begin{align*}
&\int_{\mathbb{G}}\nu(B^{(\alpha)}_{r}(r^{2\alpha}, x))^{p'}d\mu\\
&\ \lesssim\int_{B^{(\alpha)}_{r}(r^{2\alpha}, x)}\int_{\mathbb{G}}\nu(B^{(\alpha)}_{r}(r^{2\alpha}, x))^{p'-1}d\mu d\nu\\
 &\ \lesssim \int_{B^{(\alpha)}_{r}(r^{2\alpha}, x)}\int_{\mathbb{G}}\nu(B^{(\alpha)}_{r}(r^{2\alpha}, y))^{p'-1}d\mu d\nu\\
 &\ \lesssim r^{Q}\int_{B^{(\alpha)}_{r}(r^{2\alpha}, x)}\nu(B^{(\alpha)}_{r}(r^{2\alpha}, y))^{p'-1}d\nu.
\end{align*}
Therefore,
\begin{align*}
\int_{0}^{\infty}\int_{\mathbb{G}}\frac{\nu(B^{(\alpha)}_{r}(r^{2\alpha},
x))^{p'}}{r^{Q  p'}}d\mu\frac{dr}{r}&\approx
\int_{0}^{\infty}\int_{B^{(\alpha)}_{r}(r^{2\alpha}, x)}\frac{\nu(B^{(\alpha)}_{r}(r^{2\alpha}, y))^{p'-1}}{r^{ Q (p'-1)}}d\nu\frac{dr}{r}\\
&\approx \int_{B^{(\alpha)}_{r}(r^{2\alpha}, x)}\int_{0}^{\infty}\Big(\frac{\nu(B^{(\alpha)}_{r}(r^{2\alpha}, y))}{r^{Q}}\Big)^{p'-1}\frac{dr}{r}d\nu\\
&\lesssim \int_{
\mathbb{G}_{+}}\Big(\int_{0}^{\infty}\Big(\frac{\nu(B^{(\alpha)}_{r}(t,
x))}{r^{Q}}\Big)^{p'-1}\frac{dr}{r}\Big)d\nu(t,x),
\end{align*}
as desired.

{\it Part 2.} The second task is to prove
 $$
 \|(e^{-t\mathcal{L}^{\alpha}})^{*}\nu\|_{L^{p'}(\mathbb{G})}^{p'}\gtrsim \int_{\mathbb{G}_{+}}P_{\alpha p}\nu\,d\nu.
 $$

Note that
\begin{align*}
\|(e^{-t\mathcal{L}^{\alpha}})^{*}\nu\|_{L^{p'}(\mathbb{G})}^{p'}&=\int_{\mathbb{G}}
\big((e^{-t\mathcal{L}^{\alpha}})^{*}\nu(x)\big)^{p'-1}((e^{-t\mathcal{L}^{\alpha}})^{*}\nu(x))d\mu(x)\\
&=\int_{\mathbb{G}_{+}}\int_{
\mathbb{G}}((e^{-t\mathcal{L}^{\alpha}})^{*}\nu(x))^{p'-1}K^{\mathcal{L}}_{\alpha,t}(x,y)d\mu(x)\,d\nu(t,y).
\end{align*}
Upon writing
$$
\begin{cases}
K(t,y)=\int_{
\mathbb{G}}((e^{-t\mathcal{L}^{\alpha}})^{*}\nu(x))^{p'-1}K^{\mathcal{L}}_{\alpha,t}(x,y)d\mu(x);\\
B(y,2^{-m})=\{y\in \mathbb{G}:\ d(x,y)<2^{-m};\,
(2^{-m})^{2\alpha}<t<2(2^{-m})^{2\alpha}\}\\
\ \forall\ m\in\mathbb
Z\equiv\{0,\pm 1,\pm 2,...\},
\end{cases}
$$ and using Proposition \ref{prop-4},
we obtain
\begin{align*}
K(t,y)&\gtrsim \int_{\mathbb{G}}\frac{t}{(t^{{1}/{2\alpha}}+d(x,y))^{Q+2\alpha}}\Big(\int_{\mathbb{M}_{+}}\frac{s}{(s^{{1}/{2\alpha}}+d(x,z))^{Q +2\alpha}}d\nu\Big)^{p'-1}d\mu\\
&\gtrsim \sum_{m\in \mathbb{Z}}\int_{B(y,2^{-m})}t^{-{Q
}/{2\alpha}}
\Big(\int_{B^{(\alpha)}_{2^{-m}}(t, y)}s^{-{Q}/{2\alpha}}d\mu\Big)^{p'-1}d\mu\\
&\gtrsim \sum_{m\in \mathbb{Z}}\int_{B(y,2^{-m})}{2^{m Q }}
\Big(\frac{\nu(B^{(\alpha)}_{2^{-m}}(t, y))}{2^{-m Q }}\Big)^{p'-1}d\mu\\
&\gtrsim \int_{0}^{\infty} \Big(\frac{\nu(B^{(\alpha)}_{r}(t,
y))}{r^{Q}}\Big)^{p'-1}\frac{dr}{r},
\end{align*}
thereby reaching the required inequality.

\end{proof}

\begin{remark}
Under the assumption that $\mu$ satisfies (\ref{eq1.2}), Lemmas \ref{l21}\ \&\ \ref{l22} still hold for the metric measure space $(\mathbb{M}, d,\mu)$.
\end{remark}

The following theorem is the second main result in this  section.
\begin{theorem}
 \label{t4aa} Let $1<q<p<\infty$ and $\nu\in\mathcal M_+( \mathbb{G}_+)$. Then the follows are equivalent:

 $\mathrm{(i)}$ The extension $e^{-t{\mathcal{L}}^{\alpha}}: L^p(\mathbb{ G})\mapsto L^{q}(  \mathbb{G}_+,\nu)$ is bounded.

 $\mathrm{(ii)}$  $$P_{\alpha p} \nu \in
 L_\nu^{q(p-1)/(p-q)}( \mathbb{G}_{+}),$$ where
 $
 P_{\alpha p}
 $ is defined in (\ref{eq6.2}).
 \end{theorem}

 \begin{proof}
Firstly,  we show that  $\mathrm{(i)}\Rightarrow \mathrm{(ii)}$. To
do so, we first define   the $\alpha$-dyadic cube, which is denoted
by $Q^{(\alpha)}_{\delta,k,\gamma}$,  on the stratified Lie group as
follows:
$$
Q^{(\alpha)}_{\delta,k,\gamma}\equiv [k_{0}\delta^{2\alpha},
(k_{0}+1)\delta^{2\alpha})\times Q^k_{\gamma}
\,\,\,\,\hbox{as}\,\,\,\,k_{0}\in \mathbb{Z}_{+}, \  k_{i}\in
\mathbb{Z}\  \mathrm{and}\ \gamma\in I_k.
$$
It follows from Proposition \ref{prop5.8} that the family
$\{Q^{(\alpha)}_{\delta,k,\gamma}\}$ is dense in $ \mathbb{G}_+$.
  Next, we introduce the following fractional
heat Hedberg-Wolff potential   generated by $\mathcal{D}^{\alpha}$-
the family of all   above-defined $\alpha$-dyadic cubes in $
\mathbb{G}_+$:
$$
P_{\alpha
p}^{d}\nu(t,x)=\sum_{Q^{(\alpha)}_{\delta,k,\gamma}\in\mathcal{D}^{\alpha}}\Big(\frac{\nu(Q^{(\alpha)}_{\delta,k,\gamma})}{\delta^{k
Q}}\Big)^{p'-1}\textbf{1}_{Q^{(\alpha)}_{\delta,k,\gamma}}(t,x).
$$

Then we can show that
\begin{equation}\label{2.5}
\mathrm{(i)}\Rightarrow\int_{ \mathbb{G}_{+}}(P_{\alpha
p}^{d}\nu(t,x))^{{q(p-1)}/{(p-q)}}d\nu(t,x)<\infty.
\end{equation}

Indeed, by duality, $(\mathrm{i})$ is equivalent to the following
inequality
$$
\| (e^{-t \mathcal{L}^{\alpha}})^{*}(\mathbf{g}d\nu)\|_{L^{p'}( \mathbb{G}
)}^{p'}\lesssim \|\mathbf{g}\|_{L_\nu^{q'}(
\mathbb{G}_{+})}^{p'}\,\,\,\,\forall \,\,\mathbf{g}\in
L_\nu^{q'={q}/{(q-1)}}( \mathbb{G}_{+}).
$$
It is easy to check that Lemma \ref{l22} is also true with
$P_{\alpha p}^{d}\nu$ in place of $P_{\alpha p}\nu$ and
$\mathbf{g}d\nu$ in place of $d\nu$. So, one has
\begin{eqnarray*}
\|(e^{-t \mathcal{L}^{\alpha}})^{*}(\mathbf{g}d\nu)\|_{L^{p'}( \mathbb{G}
)}^{p'}
&\gtrsim& \int_{ \mathbb{G}_{+}}P_{\alpha
p}^{d}(\mathbf{g}d\nu)(t,x)\mathbf{g}(t,x)d\nu(t,x)\\
&\gtrsim&
\sum_{Q^{(\alpha)}_{\delta,k,\gamma}}\left(\frac{\int_{Q^{(\alpha)}_{\delta,k,\gamma}}\mathbf{g}(t,x)d\nu(t,x)}{\delta^{k
Q}}\right)^{p'}\delta^{k Q}.
\end{eqnarray*}
Consequently,
\begin{equation}\label{2.6}
\sum_{Q^{(\alpha)}_{\delta,k,\gamma}}\left(\frac{\int_{Q^{(\alpha)}_{\delta,k,\gamma}}\mathbf{g}(t,x)d\nu(t,x)}{\delta^{k
Q}}\right)^{p'}\delta^{k Q}\lesssim \|\mathbf{g}\|_{L_\nu^{q'}(
\mathbb{G}_{+})}^{p'}.
\end{equation}
Upon setting
$$
\lambda_{Q^{(\alpha)}_{\delta,k,\gamma}}=\left(\frac{\nu(Q^{(\alpha)}_{\delta,k,\gamma})}{\delta^{k
Q}}\right)^{p'}\delta^{k Q},
$$
one finds that $(\ref{2.6})$ is equivalent to
$$
\sum_{Q^{(\alpha)}_{\delta,k,\gamma}}\lambda_{Q^{(\alpha)}_{\delta,k,\gamma}}
\left(\frac{\int_{Q^{(\alpha)}_{\delta,k,\gamma}}\mathbf{g}\,d\nu}{\nu(Q^{(\alpha)}_{\delta,k,\gamma})}\right)^{p'}\lesssim
\|\mathbf{g}\|_{L_\nu^{q'}( \mathbb{G}_{+})}^{p'}.
$$
Define the following dyadic Hardy-Littlewood maximal function
$$
M_{\nu}^{d}h(t,x)=\sup_{(t,x)\in
Q^{(\alpha)}}\frac{1}{\nu(Q^{(\alpha)})}\int_{Q^{(\alpha)}}|h(s,y)|d\nu(s,y)\,\,\,\,\forall\,\,Q^{(\alpha)}\in
\mathcal{D}^{\alpha}.
$$
Then $M_{\nu}^{d}$ is bounded on $L^{p}_\nu( \mathbb{G}_{+})$ for
$1<p<\infty$ (cf. \cite{Aimar}). Writing
$$
\mathbf{g}(t,x)=(M_{\nu}^{d}h)^{{1}/{p'}}(t,x) \ \hbox{under}\
0\le h\in L_\nu^{{q'}/{p'}}(\mathbb{G}_{+}).
$$
It is easy to check that
$$
\left(\frac{\int_{Q^{(\alpha)}_{\delta,k,\gamma}}\mathbf{g}(t,x)d\nu(t,x)}{\nu(Q^{(\alpha)}_{\delta,k,\gamma})}\right)^{p'}\gtrsim
\frac{\int_{Q^{(\alpha)}_{\delta,k,\gamma}}h(t,x)d\nu(t,x)}{\nu(Q^{(\alpha)}_{\delta,k,\gamma})}
$$
and so that
$$
\|\mathbf{g}\|_{L_\nu^{q'}(\mathbb{G}_{+})}^{p'}\lesssim
\|h\|_{L_\nu^{{q'}/{p'}}( \mathbb{G}_{+})}.
$$
This in turn implies
$$
\sum_{Q^{(\alpha)}_{\delta,k,\gamma}}\lambda_{Q^{(\alpha)}_{\delta,k,\gamma}}
\frac{\int_{Q^{(\alpha)}_{\delta,k,\gamma}}h(t,x)d\nu(t,x)}{\nu(Q^{(\alpha)}_{\delta,k,\gamma})}\lesssim
\|h\|_{L_\nu^{{q'}/{p'}}( \mathbb{G}_{+})} \ \forall\ h\in L_\nu^{{q'}/{p'}}(
\mathbb{G}_{+}),
$$
and thus via duality,
$$
\sum_{Q^{(\alpha)}_{\delta,k,\gamma}}
\frac{\lambda_{Q^{(\alpha)}_{\delta,k,\gamma}}}{\nu(Q^{(\alpha)}_{\delta,k,\gamma})}\textbf{1}_{Q^{(\alpha)}_{\delta,k,\gamma}}\in
L_\nu^{{q'}/{(q'-p')}}(\mathbb{G}_{+}),
$$
namely,
$$
\sum_{Q^{(\alpha)}_{\delta,k,\gamma}}\left(\frac{\nu(Q^{(\alpha)}_{\delta,k,\gamma})}{\delta^{k
Q}}\right)^{p'-1}\textbf{1}_{Q^{(\alpha)}_{\delta,k,\gamma}}\in
L_\nu^{{q(p-1)}/{(p-q)}}(\mathbb{G}_{+}),
$$
which yields \eqref{2.5}.

Next, set
$$
\begin{cases}
P_{\alpha
p}^{d,\tau}\nu(t,x)=\sum_{Q^{(\alpha)}_{\delta,k,\gamma}\in\mathcal{D}^{\alpha}_{\tau}}\left(\frac{\nu(Q^{(\alpha)}_{\delta,k,\gamma})}{\delta^{k
Q}}\right)^{p'-1}\textbf{1}_{Q^{(\alpha)}_{\delta,k,\gamma}}(t,x);\\
\mathcal{D}^{\alpha}_{\tau}=\{\mathcal{F}_\tau(Q^{(\alpha)'}_{l})\}_{Q^{(\alpha)'}_{l}\in
\mathcal{D}^{\alpha}},
\end{cases}
$$
where  $\mathcal{F}_\tau: \mathbb{G}\rightarrow \mathbb{G}$ is a
  left translation. It is easy to see that the map $\mathcal{F}_\tau$ can ensure
$\mathcal{F}_\tau(Q^{(\alpha)'}_{l})$ is still an $\alpha$-dyadic
cube. Then \eqref{2.5} implies
\begin{equation}\label{2.7}
\sup_{\tau\in  \mathbb{G}_{+}}\int_{ \mathbb{G}_{+}}\big(P_{\alpha
p}^{d,\tau}\nu(t,x)\big)^{{q(p-1)}/{(p-q)}}d\nu(t,x)<\infty.
\end{equation}
Now, it remains to prove
$$
P_{\alpha p}\nu \in L_\nu^{q(p-1)/(p-q)}(\mathbb{G}_{+}).
$$
Two situations are considered in the sequel.

{\it Case 1.1. $\nu$ is a doubling measure.} In this case,
$P_{\alpha p} \nu \in L_\nu^{q(p-1)/(p-q)}(\mathbb{G}_{+})$ is a by-product
of $(\ref{2.5})$ and the following observation
$$
P_{\alpha p}\nu(t,x)\lesssim
\sum_{Q^{(\alpha)}_{\delta,k,\gamma}}\left(\frac{\nu(Q^{(\alpha)*}_{\delta,k,\gamma})}{\delta^{k
Q}}\right)^{p'-1}\textbf{1}_{Q^{(\alpha)}_{\delta,k,\gamma}}(t,x),
$$
where $Q^{(\alpha)*}_{\delta,k,\gamma}$ is the cube with the same
center as $Q^{(\alpha)}_{\delta,k,\gamma}$ and side length two times
as $Q^{(\alpha)}_{\delta,k,\gamma}$.

{\it Case 1.2. $\nu$ is a possibly non-doubling measure.} For any
$\rho>0$, write
$$
P_{\alpha
p,\rho}\nu(t,x)=\int_{0}^{\rho}\Big(\frac{\nu(B_{r}^{(\alpha)}(t,x))}{r^{
Q}}\Big)^{p'-1}\frac{dr}{r}.
$$
Then
\begin{equation*}\label{2.8}
P_{\alpha p,\rho}\nu(t,x)\lesssim \rho^{-(
 Q+1)}\int_{|\tau|\lesssim\rho}P_{\alpha
p}^{d,\tau}\nu(t,x)d\tau.
\end{equation*}
In fact, for a fixed $x\in  \mathbb{G }$ and $\rho>0$ with
$2^{i-1}\eta\leq \rho< 2^{i}\eta$, one has
$$
P_{\alpha p,\rho}\nu(t,x)\lesssim
\sum_{j=-\infty}^{i}\Big(\frac{\nu(B_{2^{j}\eta}^{(\alpha)}(t,x))}{(2^{j}\eta)^{
Q}}\Big)^{p'-1},
$$   where $i\in\mathbb{Z}$ and
$\eta>0$ will be determined later.  For $j\leq i$, let
$Q^{(\alpha)}_{\delta,j}$ be a cube centered at $x$ with
$2^{j-1}<\delta^k\leq2^{j}$. Then
$B_{2^{j}\eta}^{(\alpha)}(t,x)\subseteq Q^{(\alpha)}_{\delta,j}$ for
sufficiently small $\eta$. Assume not only that $E$ is the set of
all points $\tau\in  \mathbb{G}_{+}$ enjoying $|\tau|\lesssim\rho$,
but also that there exists $Q^{(\alpha),\tau}_{\delta}\in
\mathcal{D}^{\alpha}_{\tau}$ satisfying $\delta^k=2^{j+1}$ and
$Q^{(\alpha)}_{\delta,j}\subseteq Q^{(\alpha),\tau}_{\delta}$. A
geometric consideration produces a dimensional constant $c( Q)>0$
such that $$\widetilde{\mu}(E)\ge c( Q)\rho^{  Q+1}\ \forall\ j\leq
i,$$ where $\widetilde{\mu}$ is the product measure of $\mu$ and the
Lebesgue measure on $(0,\infty)$. Consequently, one has
\begin{align*}
\nu(B_{2^{j}\eta}^{(\alpha)}(t,x))^{p'-1}&\lesssim \widetilde{\mu}(E)^{-1}\int_{E}\sum_{l=2^{j+1}}\nu(Q^{(\alpha),\tau}_{l})^{p'-1}\textbf{1}_{Q^{(\alpha),\tau}_{\delta}}(t,x)d\widetilde{\mu}\\
&\lesssim \rho^{-(  Q+1)}\int_{|\tau|\lesssim
\rho}\sum_{l=2^{j+1}}\nu(Q^{(\alpha),\tau}_{l})^{p'-1}\textbf{1}_{Q^{(\alpha),\tau}_{l}}(t,x)d\widetilde{\mu}
\end{align*}
and so that
\begin{align*}
P_{\alpha p,\rho}\nu(t,x)&\lesssim \rho^{-(  Q+1)}\int_{|\tau|\lesssim\rho} \sum_{j=-\infty}^{i}\sum_{l=2^{j+1}}\Big(\frac{\nu(Q^{(\alpha),\tau}_{l})}{(2^{j}\eta)^{ Q}}\Big)^{p'-1}\textbf{1}_{Q^{(\alpha),\tau}_{l}}(t,x)ds\\
&\lesssim \rho^{-(  Q+1)}\int_{|\tau|\lesssim\rho}P_{\alpha
p}^{d,\tau,R}\nu(t,x)d\tau,
\end{align*}
whence reaching $(\ref{2.8})$.

Combining $(\ref{2.8})$ and the H\"{o}lder inequality with Fubini's
theorem, we can assert that
\begin{eqnarray*}
&&\int_{\mathbb{G}_{+}}\Big(P_{\alpha p,\rho}\nu(t,x)\Big)^{q(p-1)/(p-q)}d\nu(t,x)\\
&&\ \ \lesssim \int_{ \mathbb{G}_{+}} \Big[\rho^{-(  Q+1)}\Big(\int_{|\tau|\leq C\rho} \Big(P_{\alpha p}^{d,\tau}\nu\Big)^{{q(p-1)}/{(p-q)}}d\tau\Big)^{{(p-q)}/{q(p-1)}}\\
&&\quad\times\Big(\int_{|\tau|\lesssim\rho}d\tau\Big)^{1-{(p-q)}/{q(p-1)}}\Big]^{{q(p-1)}/{(p-q)}}d\mu\\
&&\ \  \lesssim\rho^{-(  Q+1)}\int_{|\tau|\lesssim\rho}\Big( \int_{ \mathbb{G}_{+}}\Big(P_{\alpha p}^{d,\tau}\nu\Big)^{{q(p-1)}/(p-q)}d\mu\Big) d\tau\\
&&\ \ \le \kappa( Q),
\end{eqnarray*}
where the last constant $\kappa( Q)$ is independent of $\rho$. This
clearly produces
$$
P_{\alpha p}\nu \in L_\nu^{q(p-1)/(p-q)}(\mathbb{G}_{+})
$$
via letting $\rho\rightarrow \infty$ and utilizing the monotone
convergence theorem.

{\it Step 2.} We prove
$$P_{\alpha p}\nu \in L_\nu^{q(p-1)/(p-q)}(\mathbb{G}_{+})\Rightarrow (\mathrm{i}).$$

Recall that (i) is equivalent to the following inequality
$$
\|(e^{-t \mathcal{L}^{\alpha}})^{*}(\mathbf{g}d\nu)\|_{L^{p'}( \mathbb{G}
)}\lesssim \|\mathbf{g}\|_{L_\nu^{q'}(
\mathbb{G}_{+})}\,\,\,\,\forall\,\,\mathbf{g}\in L_\nu^{q'}( \mathbb{G}_{+}).
$$
Thus, by Lemma \ref{l22}, it is sufficient to check that $P_{\alpha
p}\nu \in L_\nu^{q(p-1)/(p-q)}( \mathbb{G}_{+})$ implies
\begin{equation}\label{2.9}
\int_{ \mathbb{G}_{+}}P_{\alpha
p}(\mathbf{g}d\nu)(t,x)\mathbf{g}(t,x)d\mu\lesssim
\|\mathbf{g}\|_{L^{q'}_{\nu}(
\mathbb{G}_{+})}^{p'}\,\,\,\,\forall\,\,\mathbf{g}\in L_\nu^{q'}(
\mathbb{G}_{+}).
\end{equation}
There is no loss of generality in assuming $\mathbf{g}\geq 0$. Since
\begin{align*}
P_{\alpha p}(\mathbf{g}d\nu)(t,x)
&\approx\int_{0}^{\infty} \Big(\frac{\nu(B_{r}^{(\alpha)}(t,x))}{r^{ Q}}\Big)^{p'-1}\Big(\frac{\int_{B_{r}^{(\alpha)}(t,x)}\mathbf{g}(t,x)d\nu}{\nu(B_{r}^{(\alpha)}(t,x))}\Big)^{p'-1}\frac{dr}{r}\\
&\lesssim \Big(M_{\nu}\mathbf{g}(t,x)\Big)^{p'-1}P_{\alpha
p}\nu(t,x),
\end{align*}
an application of the H\"{o}lder inequality gives
\begin{align*}
&\int_{\mathbb{G}_{+}}P_{\alpha p}(\mathbf{g}d\nu)(t,x)d\nu(t,x)\\
&\ \ \lesssim \int_{ \mathbb{G}_{+}}\Big(M_{\nu}\mathbf{g}(t,x)\Big)^{p'-1}P_{\alpha p}\nu(t,x)\mathbf{g}(t,x)d\nu(t,x)\\
&\ \ \lesssim \frac{\Big(\int_{
\mathbb{G}_{+}}\Big(M_{\nu}\mathbf{g}(t,x)\Big)^{q'}d\nu(t,x)\Big)^{{(p'-1)}/{q'}}}{\Big(\int_{
\mathbb{G}_{+}}\Big(\mathbf{g}(t,x)P_{\alpha
p}\nu(t,x)\Big)^{{q'}/{(q'-p'+1)}}d\nu(t,x)\Big)^{-{(q'-p'+1)}/{q'}}},
\end{align*}
where
$$
M_{\nu}\mathbf{g}(t,x)=\sup_{r>0}\frac{1}{\nu(B_{r}^{(\alpha)}(t,x))}\int_{B_{r}^{(\alpha)}(t,x)}\mathbf{g}(s,y)d\nu(s,y)
$$
is the centered Hardy-Littlewood maximal function of $\mathbf{g}$
with respect to $\nu$ (cf. \cite{coifman}). The fact that $M_{\nu}$
is bounded on $L_\nu^{q'}(\mathbb{G}_{+})$  and H\"{o}lder's inequality
imply
$$
\int_{\mathbb{G}_{+}}P_{\alpha p}(\mathbf{g}d\nu)(t,x)d\nu(t,x)\lesssim
\frac{\|\mathbf{g}\|_{L_\nu^{q'}( \mathbb{G}_{+})}^{p'}}{\Big(\int_{
\mathbb{G}_{+}}\Big(P_{\alpha
p}\nu\Big)^{{q(p-1)}/{(p-q)}}d\nu\Big)^{-{(p-q)}/{q(p-1)}}},
$$
whence reaching  \eqref{2.9}.

 \end{proof}

\section{Acknowledgements}
\begin{itemize}
\item{ J. Huang was supported  by NNSF of China (No.11471018).}

\item{P. Li was supported by NNSF of
China (No.11871293, No.11571217); Shandong Natural Science
Foundation of China (No.ZR2017JL008, No.ZR2016AM05); University
Science and Technology Projects of Shandong Province (No.J15LI15).}

\item {Y. Liu was supported NNSF of China (No.11671031);
the Fundamental Research Funds for the Central Universities
(No.FRF-BR-17-004B) and Beijing
Municipal Science and Technology Project (No.Z17111000220000).}

\item {S. Shi was supported by NNSF of China (No.11771195) and the Applied Mathematics Enhancement Program of Linyi University (No.LYDX2013BS059).}

\end{itemize}

\end{document}